\documentclass[11pt]{amsart}
\usepackage{mathrsfs}
\usepackage{geometry}
\usepackage{graphicx}
\usepackage{amssymb}
\usepackage{epstopdf}
\usepackage{amsmath,amscd}
\usepackage{amsthm}
\usepackage{enumerate}
\usepackage{url,verbatim}
\usepackage{subfig}
\usepackage{enumitem}
\usepackage{xcolor}
\usepackage{tikz}
\usepackage{needspace}
\usepackage{placeins}
\usetikzlibrary{arrows.meta,calc,positioning}

\newif\ifrevisions
\revisionsfalse
\ifrevisions\colorlet{revisioncolor}{red}\else\colorlet{revisioncolor}{black}\fi
\DeclareRobustCommand{\rev}[1]{{\color{revisioncolor}#1}}
\newenvironment{revision}{\begingroup\color{revisioncolor}\hypersetup{citecolor=revisioncolor,linkcolor=revisioncolor,urlcolor=revisioncolor}}{\endgroup}
\newcommand{\revheading}[1]{\par\Needspace{4\baselineskip}\medskip\noindent\rev{\textbf{#1}}\par\nobreak\smallskip}
\newcommand{\currentrev}[1]{{\color{revisioncolor}#1}}

\RequirePackage[colorlinks,linkcolor=black,citecolor=blue,urlcolor=blue,breaklinks=true]{hyperref}
\theoremstyle{plain}
\calclayout

\newtheorem{theorem}{Theorem}
\newtheorem{definition}[theorem]{Definition}
\newtheorem{lemma}[theorem]{Lemma}
\newtheorem{proposition}[theorem]{Proposition}
\newtheorem{corollary}[theorem]{Corollary}
\newtheorem{example}[theorem]{Example}
\newtheorem{assumption}[theorem]{Assumption}
\newtheorem{remark}[theorem]{Remark}

\newcommand\RR{{\mathbb R}}
\newcommand\ZZ{{\mathbb Z}}

\newcommand\YY{{\mathbb {Y}}}

\renewcommand\a{\alpha}

\renewcommand\ell{l}

\newcommand\CC{\mathbb{C}}

\newcommand\bA{\mathbf{A}}
\newcommand\bB{\mathbf{B}}
\newcommand\MP{\mathbb{MP}}

\newcounter{mycount}

\numberwithin{equation}{section}
\numberwithin{theorem}{section}
\numberwithin{figure}{section}

\title[Doubly Free-Boundary Rail-Yard Dimers]
{Doubly Free-Boundary Rail-Yard Dimers and Annular Gaussian Fluctuations}

\author[Z. Li]{Zhongyang Li}
\address{Department of Mathematics,
University of Connecticut,
Storrs, CT 06269-3009, USA}
\email{zhongyang.li@uconn.edu}
\urladdr{https://mathzhongyangli.wordpress.com}

\thanks{\textbf{Corresponding author:} Zhongyang Li. \emph{Affiliation:} Department of Mathematics, University of Connecticut, Storrs, CT 06269-3009, USA. \emph{Email:} \texttt{zhongyang.li@uconn.edu}. \emph{ORCID:} \texttt{https://orcid.org/0000-0002-5020-9182}.}
\date{}

\begin{document}

\begin{abstract}
We study rail-yard dimer measures with free boundary conditions at both the left and the
right boundary.  The double free-boundary geometry produces an infinite family of reflected
Cauchy factors in the partition function and in the exact contour formulas for height
Laplace observables.  These factors are absent from the empty-boundary model and survive
in both the deterministic and second-order asymptotics.

\currentrev{For admissible piecewise periodic weights satisfying the stated
separation conditions, we prove a Laplace-transform law of large numbers.
In the natural moment variable
\[
        x=e^{-n\beta\kappa},
\]
this gives weak convergence of the slope measures and a corresponding
limit shape. For \(\beta=1\), under the stated regularity assumptions and
nonexceptional endpoint conditions at nearby regular slope points,
each frozen-boundary point \((\chi,\kappa)\) admits
\(w\in\mathbb R\setminus\{0\}\) such that
\[
        S_\chi(w)=e^{-n\kappa},
        \qquad
        \frac{S_\chi'(w)}{S_\chi(w)}=0,
        \qquad
        S_\chi(w):=\mathcal G_\chi(w)\prod_{r\ge1}\mathcal F_{u,v,r}(w).
\]
}

The main second-order result is a Gaussian fluctuation theorem for centered height
Laplace observables.  The covariance is the annular reflected-image
kernel
\[
        \mathsf K_{LL}(z,w)
        =
        \partial_z\partial_w
        \log\frac{\Theta_{\mathfrak q}(z/w)}
                  {\Theta_{\mathfrak q}(u^2zw)},
        \qquad \mathfrak q=(uv)^2,
\]
with specified contour interpretation.  Thus the two free
boundaries do not merely alter the deterministic limit shape: they replace the usual Gaussian free field
half-plane image structure by an annular prime-function covariance on the Laplace-test
class.

As a final exact-solvability consequence, we construct an exact growth-diagram
sampler for the finite reflected truncations of the doubly free-boundary model
and prove that the resulting \(K\)-truncated laws converge, as \(K\to\infty\),
in total variation to the full doubly free-boundary Gibbs measure on every fixed
rail-yard graph with finitely many columns.
\end{abstract}

\maketitle

\section{Introduction}

\subsection{Two free boundaries and the reflected-product regime}

Random planar dimer models provide a basic setting in which deterministic limit
shapes coexist with Gaussian height fluctuations.  In many simply connected
integrable models, the second-order field is governed by a Green kernel, or by
an image form of such a kernel in a spectral coordinate; see, for instance,
\cite{Ken01,KOS,KO07}.  This paper studies a different boundary mechanism.  We
consider rail-yard dimers with free boundary conditions at both the left and the
right boundary.  The two free endpoints repeatedly reflect the Cauchy factors in
the Schur-process description.  This reflection mechanism produces infinite
products in the exact formulas and, at the level of fluctuations, an annular
rather than simply connected covariance kernel.

Rail-yard graphs were introduced in \cite{bbccr}.  Their perfect matchings are
encoded by interlacing sequences of partitions, and the resulting measures are
closely related to Schur processes and their variants.  This class includes many
standard tiling and dimer models, such as Aztec-type models, steep tilings,
tower and pyramid-type examples, and contracting lattice examples; see, for
example,
\cite{EKLP92a,EKLP92b,KJ05,bk,BY09,BCC17,BF15,BL17,ZL18,ZL20,Li182,ZL202}.
Steep tilings \cite{BCC17} form an important free-boundary subclass in this
circle of models.  Exact partition functions and generating functions for such
models were obtained in earlier works, and the empty-boundary rail-yard
asymptotic analysis was developed in \cite{LV21}.

\begin{revision}
\label{rev:comment1}
The exact doubly free-boundary Schur process was already studied by Betea,
Bouttier, Nejjar and Vuleti\'c in \cite{BBNV}.  In particular, their work gives
the partition function and correlation formulas with both endpoint partitions
free.  Their approach uses free fermions and free-boundary states: with one free
boundary the associated point process is pfaffian, whereas with two free
boundaries a suitable random vertical shift produces a pfaffian point process.
Thus the exact two-boundary measure and its reflected infinite-product
structure are part of the earlier theory; the asymptotic applications in
\cite[Sections~4--6]{BBNV} concern the one-free-boundary setting.

Our analysis instead starts from contour formulas for height Laplace observables
obtained through the Schur/Macdonald-process representation and difference
operators.  We use one-point asymptotics and a connected-interaction expansion
of centered moments to study the limit shape and Gaussian fluctuations directly,
without first passing to the shift-mixed point process.  In the doubly
free-boundary regime, the reflected products persist in these moment formulas.
Their tails must be controlled uniformly on the admissible contours. The frozen-boundary analysis uses finite truncations.
Under explicit parameter inequalities ensuring zero--pole interlacing
and the stated nonexceptional endpoint conditions, we establish uniform
bounds preventing the nonreal roots of the truncated equations from
escaping to $0$ or $\infty$ as $K\to\infty$.
These are the extra analytic steps developed in Sections~\ref{sect:as}
and~\ref{sect:fb}.

The contribution of the present paper is therefore the global asymptotic
analysis for the stated class of doubly free-boundary rail-yard weights,
including the annular covariance on the height-Laplace test class, together
with an exact sampler for finite reflected truncations and its convergence.
The reflected geometry is already visible in the exact theory of \cite{BBNV};
here we track its effect through the deterministic and second-order scaling
limits and through the finite-truncation sampling construction.
\end{revision}

The main new analytic object is the spectral product
\[
   S_\chi(w)=G_\chi(w)\prod_{r\ge1}F_{u,v,r}(w).
\]
The finite factor \(G_\chi\) is the contribution of the unreflected rail-yard
word, while the infinite product records the successive reflections at the two
free boundaries.  This product is absent from the empty-boundary model and
collapses in the one-free-boundary specializations.  In the two-free-boundary
case it survives in the partition function, in the contour formulas for height
Laplace observables, in the limit shape, in the frozen-boundary equation, and in
the fluctuation kernel.

\subsection{Exact formulas, limit shape, and frozen boundary}

The first part of the paper is exact.  We identify the doubly free-boundary
rail-yard measure with a Schur-process-type measure with two free endpoints and
derive contour formulas for height Laplace observables.  The two-boundary
Cauchy identities generate the reflected products above, and these products are
the source of the new asymptotic behavior.

For admissible piecewise periodic weights, the one-point contour formula gives a
Laplace-transform law of large numbers for the rescaled height function; see
Theorem~\ref{l61}.  We then work in the natural moment variable
\[
   x=e^{-n\beta\kappa}
\]
and use a compact moment-determinacy argument to upgrade transform convergence
to a weak slope-measure limit shape; see Theorem~\ref{thm:weak-limit-shape-beta}.
The limiting height profile is recovered from this slope measure.

The frozen-boundary analysis is one of the main places where the doubly free-boundary
case differs from the finite rational-function setting.  For a finite truncation
of \(S_\chi\), the spectral equation is rational and one can use a finite
zero--pole picture.  In the infinite product limit, however, the real axis
contains infinitely many zeros and poles, and \(0\) is an accumulation point.
Thus the usual finite root-count argument does not apply directly.

We develop a finite-truncation, zero--pole interlacing, and no-escape framework
to pass from the rational truncations to the infinite reflected product.  The
interlacing hypotheses separate inward and outward zero--pole families, provide
the admissible contour and logarithmic branch, and prevent roots relevant to the
frozen-boundary problem from escaping to \(0\) or \(\infty\).  Under the stated
regularity and nonexceptional endpoint assumptions, regular frozen-boundary
points are therefore described by the real double-root system
\[
   S_\chi(w)=e^{-n\kappa},
   \qquad
   \frac{d}{dw}\log S_\chi(w)=0,
\]
in the case \(\beta=1\); see Theorem~\ref{p412}.  Equivalently, the regular
interface is obtained as the locus where the equation
\(S_\chi(w)=e^{-n\kappa}\) develops a multiple real root.

\subsection{Annular Gaussian fluctuations}

The main second-order result is the Gaussian fluctuation theorem for centered
height Laplace observables; see Theorem~\ref{t77}.  The covariance is not the
usual simply connected Green kernel and is not merely a half-plane image
correction.  It is an annular reflected-image covariance.  In the notation of
Definition~\ref{def:KLL},
\[
   \mathsf K_{LL}(z,w)
   =
   \partial_z\partial_w
   \log
   \frac{\Theta_{\mathfrak q}(z/w)}
        {\Theta_{\mathfrak q}(u^2zw)},
   \qquad
   \mathfrak q=(uv)^2 .
\]
The numerator records the direct annular images, while the denominator records
the images created by the two free boundaries.  Thus the two free boundaries do
not merely modify the deterministic limit shape.  They change the second-order
field itself, replacing the simply connected image structure by an annular
prime-function covariance on the Laplace-test class.

\subsection{Sampling}

We also prove an exact sampling result.  Its role is different from that of the
asymptotic theorems, but it is another consequence of the same reflected Cauchy
structure.  The SchurSample algorithm of \cite{bbbccv14} samples ordinary Schur
processes with both endpoint partitions fixed to \(\emptyset\).  Its
SymmetricSchurSample extension samples the right-free, equivalently symmetric,
one-sided free-boundary process.  The double-sided free-boundary process is
natural in the same formalism, but it lies outside the one-sided sampling
construction developed in \cite{bbbccv14}.

The doubly free-boundary rail-yard measure is governed by an infinite reflected
product, so it is not sampled directly by the finite SchurSample growth diagram
without first truncating the reflected product.  We construct instead a folded
reflected growth-diagram sampler for finite reflected truncations.  For every
truncation level \(K\), the algorithm samples exactly the \(K\)-truncated doubly
free-boundary law.  We then identify the limiting terminal-boundary weights and
normalizing constants and prove that these truncated laws converge in total
variation to the full doubly free-boundary Gibbs measure on every fixed
rail-yard graph with finitely many columns; see
Theorem~\ref{thm:truncated-sampler}.

\subsection{Organization of the paper}

Section~\ref{sect:bk} introduces rail-yard graphs, free-boundary dimer states,
the partition-sequence encoding, height functions, and the standing assumptions.
Section~\ref{sect:mh} derives the contour moment formulas for height Laplace
observables from the Schur/Macdonald-process representation and the
two-boundary Cauchy identities.

Section~\ref{sect:as} performs the asymptotic analysis of the moment formulas.
It proves the one-point asymptotics, the Gaussian cumulant expansion, and the
identification of the limiting covariance with the annular prime-function
derivative kernel.  Section~\ref{sect:fb} proves the Laplace-transform law of
large numbers, upgrades it to a weak slope-measure limit shape in the natural
moment variable \(x=e^{-n\beta\kappa}\), and derives the regular
frozen-boundary equations.

Section~\ref{sect:sampler} constructs the truncated growth-diagram sampler,
proves exactness for every fixed truncation level, and proves total-variation
convergence of the truncated laws to the full doubly free-boundary law.
Section~\ref{sect:ex} works out an explicit example.  Appendix~A collects the
background on Macdonald polynomials and the technical identities used in the
proofs.

\section{Background}\label{sect:bk}

In this section we introduce rail-yard graphs, dimer coverings (perfect matchings with free boundary),
and the associated height function. We also summarize the standing assumptions on the parameters and
state the main results of the paper.

\subsection{Weighted rail-yard graphs}

Fix integers $l,r\in\ZZ$ with $l\le r$. We write
\[
[l..r]:=[l,r]\cap\ZZ=\{l,l+1,\dots,r\},
\qquad
[m]:=\{1,2,\dots,m\}\ \ (m\in\ZZ_{>0}).
\]
Let $\underline a=(a_i)_{i\in[l..r]}\in\{L,R\}^{[l..r]}$ be an $LR$-word and
$\underline b=(b_i)_{i\in[l..r]}\in\{+,-\}^{[l..r]}$ a sign word.

\begin{definition}[Rail-yard graph]\label{def:ryg}
The \emph{rail-yard graph} $RYG(l,r,\underline a,\underline b)$ is the bipartite graph embedded in the
plane with vertex set
\[
V=[2l-1..2r+1]\times\Bigl(\ZZ+\tfrac12\Bigr).
\]
A vertex is called \emph{even} (resp.\ \emph{odd}) if its $x$-coordinate is an even (resp.\ odd) integer.
For each $i\in[l..r]$ and each $y\in\ZZ+\tfrac12$, the even vertex $(2i,y)$ is connected to three odd
vertices: two horizontal neighbors $(2i-1,y)$ and $(2i+1,y)$, and one diagonal neighbor determined by
$(a_i,b_i)$ as follows. Define
\[
\varepsilon_i:=\begin{cases}-1,& a_i=L,\\ +1,& a_i=R,\end{cases}
\qquad
\sigma_i:=\begin{cases}+1,& b_i=+,\\ -1,& b_i=-.\end{cases}
\]
Then the diagonal edge from $(2i,y)$ ends at $(2i+\varepsilon_i,\,y+\sigma_i)$.
\end{definition}

See Figure~\ref{fig:rye} for an example.

\begin{figure}
\centering
\includegraphics[width=.8\textwidth]{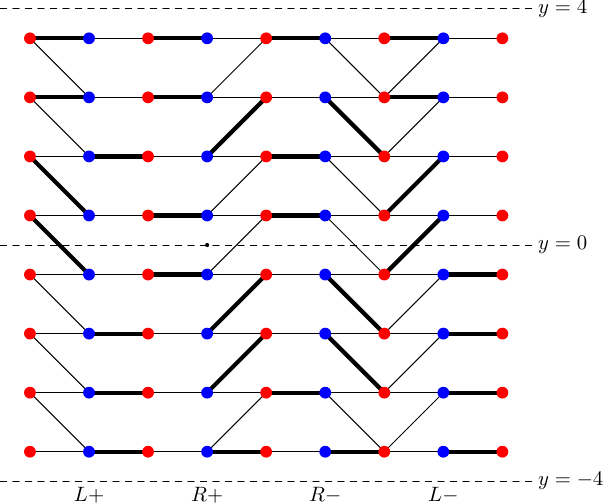}
\caption{A rail-yard graph with $LR$-word $\underline a=(L,R,R,L)$ and sign word $\underline b=(+,+,-,-)$.
Odd vertices are shown in red and even vertices in blue; thick edges depict one pure dimer covering.
Above the line $y=4$, only horizontal edges with an odd vertex on the left are present in the covering;
below the line $y=-4$, only horizontal edges with an even vertex on the left are present.
The corresponding sequence of partitions (from left to right) is
$\emptyset\prec(2,0,\ldots)\prec'(3,1,1,\ldots)\succ'(2,0,\ldots)\succ\emptyset$.}
\label{fig:rye}
\end{figure}

The \emph{left boundary} (resp.\ \emph{right boundary}) of $RYG(l,r,\underline a,\underline b)$ is the set of
odd vertices with $x$-coordinate $2l-1$ (resp.\ $2r+1$). Vertices not on the boundary are called \emph{inner}.
A face is called \emph{inner} if all its incident vertices are inner.

\paragraph{Edge weights.}
Fix parameters $\underline x=(x_l,\dots,x_r)$ (typically $x_i>0$). We assign weights as follows:
all horizontal edges have weight $1$, and each diagonal edge incident to the even column $x=2i$
has weight $x_i$.

\subsection{Dimer coverings and free-boundary states}

\begin{definition}[Dimer coverings]\label{def:dimer}
A \emph{dimer covering} of $RYG(l,r,\underline a,\underline b)$ is a subset $M$ of edges such that
\begin{enumerate}
\item every inner vertex is incident to exactly one edge of $M$;
\item every boundary vertex is incident to at most one edge of $M$;
\item \rev{There exists a positive integer $N>0$ such that, when $y<-N$ (resp.\ $y>N$), only horizontal edges with an even vertex on the left (resp.\ on the right) are present in $M$.}
\end{enumerate}
\begin{revision}
\label{rev:comment2}
Since the graph has finitely many columns, the last condition confines all present diagonal edges to a finite strip and therefore already implies that there are only finitely many such edges.
\end{revision}

A dimer covering $M$ is called \emph{pure} if, in addition,
\begin{itemize}
\item for $y>0$, each left boundary vertex $(2l-1,y)$ is matched to the interior, i.e.\ it is incident in $M$ to exactly one edge whose other endpoint has $x$-coordinate $2l$, while for $y<0$ it is unmatched;
\item for $y<0$, each right boundary vertex $(2r+1,y)$ is matched to the interior, i.e.\ it is incident in $M$ to exactly one edge whose other endpoint has $x$-coordinate $2r$, while for $y>0$ it is unmatched.
\end{itemize}
\end{definition}

\begin{definition}[Free-boundary dimer states]\label{def:free-boundary-state}
\begin{revision}
\label{rev:comment3}
A \emph{free-boundary dimer state} of $RYG(l,r,\underline a,\underline b)$ is a dimer covering in the sense of Definition~\ref{def:dimer}, viewed through its uniquely determined particle--hole encoding. At a matched odd vertex, the matching edge gives a hole if its even endpoint is on the right and a particle if its even endpoint is on the left. Each unmatched left boundary odd vertex is declared a particle, and each unmatched right boundary odd vertex is declared a hole. 
\end{revision}

\begin{revision}
\label{rev:comment4}
The tail condition in Definition~\ref{def:dimer}, together with the encoding just specified, implies that the Maya diagram on every odd column, including $x=2l-1$ and $x=2r+1$, consists entirely of holes sufficiently far above and entirely of particles sufficiently far below. Thus each odd column $2m-1$ determines a partition (a weakly decreasing sequence of nonnegative integers with only finitely many nonzero entries),
\end{revision}
\begin{align*}
\lambda^{(M,m)}:=(\lambda^{(M,m)}_1\geq \lambda^{(M,m)}_2\geq\ldots),
\end{align*}
where $\lambda^{(M,m)}_i$ is the number of holes below the $i$th highest particle.

\end{definition}

See Figure~\ref{fig:rye} for an example of a pure free-boundary dimer state.

For a free-boundary dimer state $M$ on the rail-yard graph $RYG(l,r,\underline a,\underline b)$, define the associated height function $h_M$ on faces of $RYG(l,r,\underline a,\underline b)$ as follows.
We first define a preliminary height function $\bar h_M$ on faces of $RYG(l,r,\underline a,\underline b)$.
Recall that there exists a positive integer $N>0$ such that, when $y<-N$ (resp.\ $y>N$), only horizontal edges with an even vertex on the left (resp.\ on the right) are present in $M$.
Fix a face $f_0$ of $RYG(l,r,\underline a,\underline b)$ such that the midpoint of $f_0$ lies on the horizontal line $y=-N$, and define
\[
\bar h_M(f_0)=0.
\]
For any two adjacent faces $f_1$ and $f_2$ sharing an edge,
\begin{itemize}
\item if moving from $f_1$ to $f_2$ crosses a horizontal edge with an odd vertex on the left, then $\bar h_M(f_2)-\bar h_M(f_1)=1$ if the edge is present in $M$, and $\bar h_M(f_2)-\bar h_M(f_1)=-1$ if it is absent from $M$;
\item if moving from $f_1$ to $f_2$ crosses a diagonal edge with an odd vertex on the left, then $\bar h_M(f_2)-\bar h_M(f_1)=2$ if the edge is present in $M$, and $\bar h_M(f_2)-\bar h_M(f_1)=0$ if it is absent from $M$.
\end{itemize}

Let $\bar h_0$ be the preliminary height function associated to the dimer configuration $M_0$ satisfying
\begin{itemize}
\item no diagonal edge is present; and
\item each present edge is horizontal with an even vertex on the left.
\end{itemize}
The height function $h_M$ associated to $M$ is then defined by
\begin{equation}
 h_M=\bar h_M-\bar h_0.
\label{dhm}
\end{equation}

Let $m\in[l..r]$. Let $x=2m-\frac12$ be a vertical line. Note that all the horizontal and diagonal edges of $RYG(l,r,\underline a,\underline b)$ crossed by $x=2m-\frac12$ have odd vertices on the left. Then for each point $\left(2m-\frac12,y\right)$ lying in a face of $RYG(l,r,\underline a,\underline b)$, we have
\begin{equation}
 h_M\left(2m-\frac12,y\right)=2\left[N_{h,M}\left(2m-\frac12,y\right)+N_{d,M}\left(2m-\frac12,y\right)\right],
\label{hm1}
\end{equation}
where $N_{h,M}\left(2m-\frac12,y\right)$ is the total number of horizontal edges present in $M$ crossed by $x=2m-\frac12$ below $y$, and $N_{d,M}\left(2m-\frac12,y\right)$ is the total number of diagonal edges present in $M$ crossed by $x=2m-\frac12$ below $y$.
From the definition of a dimer covering, both $N_{h,M}\left(2m-\frac12,y\right)$ and $N_{d,M}\left(2m-\frac12,y\right)$ are finite for each finite $y$.
To see why \eqref{hm1} is true, note that when $y\le -N$, $h_M\left(2m-\frac12,y\right)=0$.
For $y>-N$, moving from $\left(2m-\frac12,-N\right)$ to $\left(2m-\frac12,y\right)$ vertically, no edges present in $M_0$ are crossed.
Moreover, each time a horizontal or diagonal edge present in $M$ is crossed, $h_M=\bar h_M-\bar h_0$ increases by $2$; each time a horizontal or diagonal edge absent from $M$ is crossed, $h_M=\bar h_M-\bar h_0$ does not change.

Note also that $x=2m+\frac12$ is a vertical line such that all the horizontal and diagonal edges of $RYG(l,r,\underline a,\underline b)$ crossed by $x=2m+\frac12$ have even vertices on the left. Then for each point $\left(2m+\frac12,y\right)$ lying in a face of $RYG(l,r,\underline a,\underline b)$, we have
\begin{equation}
 h_M\left(2m+\frac12,y\right)=2\left[J_{h,M}\left(2m+\frac12,y\right)-N_{d,M}\left(2m+\frac12,y\right)\right],
\label{hm2}
\end{equation}
where $J_{h,M}\left(2m+\frac12,y\right)$ is the total number of horizontal edges absent from $M$ crossed by $x=2m+\frac12$ below $y$, and $N_{d,M}\left(2m+\frac12,y\right)$ is the total number of diagonal edges present in $M$ crossed by $x=2m+\frac12$ below $y$.
From the definition of a dimer covering, both $J_{h,M}\left(2m+\frac12,y\right)$ and $N_{d,M}\left(2m+\frac12,y\right)$ are finite for each finite $y$.
To see why \eqref{hm2} is true, note that when $y\le -N$, $h_M\left(2m+\frac12,y\right)=0$.
For $y>-N$, moving from $\left(2m+\frac12,-N\right)$ to $\left(2m+\frac12,y\right)$ vertically, no diagonal edges present in $M_0$ are crossed and no horizontal edges absent from $M_0$ are crossed.
Moreover, each time a diagonal edge present in $M$ is crossed, $h_M=\bar h_M-\bar h_0$ decreases by $2$; each time a horizontal edge absent from $M$ is crossed, $h_M=\bar h_M-\bar h_0$ increases by $2$; each time a horizontal edge present in $M$ or a diagonal edge absent from $M$ is crossed, $h_M=\bar h_M-\bar h_0$ does not change.

\bigskip

\revheading{Column-coordinate convention for asymptotic observables.}\label{rev:comment5}
The asymptotic theorems below use the column index as the horizontal macroscopic
coordinate, rather than the embedded Euclidean coordinate $2i\pm\frac12$.  For a column
$m\in[l..r]$ we set
\begin{equation}
        h_M^{\mathrm{col}}(m,y):=
        \begin{cases}
        h_M\!\left(2m-\frac12,y\right),& a_m=L,\\[1mm]
        h_M\!\left(2m+\frac12,y\right),& a_m=R.
        \end{cases}
\label{eq:column-height-convention}
\end{equation}
Thus, when $\epsilon i^{(\epsilon)}\to\chi$, expressions formerly written as
$h_M(\chi/\epsilon,y)$ should be read as
$h_M^{\mathrm{col}}(i^{(\epsilon)},y)$.  This convention removes the harmless but
potentially confusing factor of two between the embedded graph coordinate and the column
coordinate used in the scaling assumptions.  

\subsection{Partitions}

A partition is a weakly decreasing sequence $\lambda=(\lambda_i)_{i\ge 1}$ of non-negative integers such that $\lambda_i=0$ for all sufficiently large $i$.
Let $\mathbb{Y}$ be the set of all partitions.
The size and length of $\lambda$ are
\[
|\lambda|=\sum_{i\ge 1}\lambda_i,
\qquad
l(\lambda)=\max\{i:\lambda_i>0\}.
\]
For partitions $\lambda,\mu$, we write $\mu\subseteq\lambda$ if $\mu_i\le \lambda_i$ for all $i\ge 1$.

We identify a partition \(\lambda=(\lambda_1,\lambda_2,\ldots)\) with its
Young diagram
\[
        Y(\lambda)
        :=
        \{(i,j)\in\mathbb Z_{>0}^2:\ 1\le j\le \lambda_i\},
\]
\begin{revision}
\label{rev:comment6}
drawn in English convention, so that the \(i\)-th row contains
\(\lambda_i\) boxes and all rows are left-justified. For
\(\mu\subseteq\lambda\), the skew Young diagram is
\end{revision}
\[
        \lambda/\mu:=Y(\lambda)\setminus Y(\mu).
\]

We say that $\lambda$ and $\mu$ interlace, and write $\lambda\succ\mu$ (or $\mu\prec\lambda$), if
\[
\lambda_i\ge \mu_i\ge \lambda_{i+1},\qquad \forall i\ge 1.
\]
In terms of Young diagrams, $\lambda/\mu$ is then a horizontal strip.
The conjugate partition $\lambda'$ is defined by
\[
\lambda_i'=\big|\{j\ge 1:\lambda_j\ge i\}\big|,
\qquad i\ge 1,
\]
\rev{and we write $\mu\prec'\lambda$ (equivalently, $\lambda\succ'\mu$) if $\mu'\prec\lambda'$.}

\begin{definition}\label{dss}
Let \(\lambda,\mu\) be partitions.  Let \(h_r\) denote the complete
homogeneous symmetric function of degree \(r\), with the conventions
\[
        h_0=1,
        \qquad
        h_r=0 \quad \text{for } r<0.
\]
Equivalently, in variables \(\mathbf x=(x_1,x_2,\ldots)\),
\[
        \sum_{r\ge0} h_r(\mathbf x) z^r
        =
        \prod_{i\ge1}\frac{1}{1-x_i z}.
\]

If \(\mu\not\subseteq\lambda\), set \(s_{\lambda/\mu}=0\).  If
\(\mu\subseteq\lambda\), define the skew Schur function by the Jacobi--Trudi
formula
\[
        s_{\lambda/\mu}
        =
        \det\bigl(
        h_{\lambda_i-\mu_j-i+j}
        \bigr)_{i,j=1}^{\ell(\lambda)},
\]
where \(\mu_j\) is understood to be \(0\) for \(j>\ell(\mu)\).  We also write
\[
        s_\lambda=s_{\lambda/\emptyset}.
\]

After applying a specialization \(\rho\) (see Definition \ref{dsp1}) of the algebra of symmetric
functions, the same formula defines \(s_{\lambda/\mu}(\rho)\) with
\(h_r\) replaced by \(h_r(\rho)\).
\end{definition}

Fix $N\ge 1$. Let $\mu=(\mu^{(0)},\ldots,\mu^{(2N)})\in\mathbb{Y}^{2N+1}$ satisfy
\begin{equation}
 \mu^{(0)}\subseteq\mu^{(1)}\supseteq\mu^{(2)}\subseteq\cdots\subseteq\mu^{(2N-1)}\supseteq\mu^{(2N)}.
\label{sm}
\end{equation}
Let $\rho_1,\ldots,\rho_{2N}$ be specializations as in Definition~\ref{dsp1}, and $u,v\in (0,1)$. Define
\begin{equation}
 \mathbb{P}(\mu):=\frac{u^{|\mu^{(0)}|}v^{|\mu^{(2N)}|}\prod_{k=0}^{N-1}s_{\mu^{(2k+1)}/\mu^{(2k)}}(\rho_{2k+1})s_{\mu^{(2k+1)}/\mu^{(2k+2)}}(\rho_{2k+2})}{Z},
\label{ppt}
\end{equation}
where the partition function $Z$ is the sum of the numerator over all sequences satisfying \eqref{sm}.

The following proposition is proved in Section~5.3 of~\cite{BCC17}; see also Proposition~2.1 of~\cite{BBNV}.

\begin{proposition}\label{p23}
The partition function $Z$ can be computed as
\[
 Z=\prod_{1\le k\le l\le N}H\bigl(\rho_{2k-1};\rho_{2l}\bigr)
 \prod_{n\ge 1}
 \frac{\tilde{H}\bigl([u^{n-1}v^n\rho^{(o)}\bigr)\,\tilde{H}\bigl(u^nv^{n-1}\rho^{(e)}\bigr)\,H\bigl(u^{2n}\rho^{(o)};v^{2n}\rho^{(e)}\bigr)}{1-u^nv^n}.
\]
Here
\[
\rho^{(o)}=\bigcup_{i=1}^{N}\rho_{2i-1},
\qquad
\rho^{(e)}=\bigcup_{i=1}^{N}\rho_{2i}.
\]
See Definition~\ref{dsp1}, in particular \eqref{dru} \eqref{dh1} and \eqref{dht}, for the definitions of \(H\), \(\widetilde{H}\), the union \(\rho_1\cup\rho_2\) of specializations, and the notation \(c\rho\).
\end{proposition}

\subsection{Dimer coverings and sequences of partitions}

Let $G=RYG(l,r,\underline a,\underline b)$. 
For each free-boundary dimer state, see Definition~\ref{def:free-boundary-state} for the corresponding particle-hole configuration.
We define the \emph{charge} (sometimes also called the \emph{central charge}) $c^{(M,m)}$ on the odd column $x=2m-1$ by
\begin{equation}
 c^{(M,m)}=\#\{\text{particles on }x=2m-1\text{ with }y>0\}-\#\{\text{holes on }x=2m-1\text{ with }y<0\}.
\label{dcg}
\end{equation}
Under Definition~\ref{def:free-boundary-state}, this charge is finite and is determined by the boundary data (i.e., it is constant along the column evolution; see Lemma~\ref{le32}).

Definition \ref{def:free-boundary-state} associates a partition $\lambda^{(M,m)}$ to the particle--hole configuration on each odd column.
It is straightforward to verify the following lemma.

\begin{lemma}\label{lm24}
Let $G=RYG(l,r,\underline a,\underline b)$ be a rail-yard graph, and let $(\lambda^{(M,l)},\lambda^{(M,l+1)},\ldots,\lambda^{(M,r+1)})$ be the sequence of partitions corresponding to a free-boundary dimer state $M$ on $G$. Then for $i\in[l..r]$:
\begin{enumerate}
\item If $(a_i,b_i)=(L,-)$, then $\lambda^{(M,i+1)}\prec\lambda^{(M,i)}$;
\item If $(a_i,b_i)=(L,+)$, then $\lambda^{(M,i+1)}\succ\lambda^{(M,i)}$;
\label{rev:comment7}
\item \rev{If $(a_i,b_i)=(R,-)$, then $\lambda^{(M,i+1)}\prec'\lambda^{(M,i)}$;}
\item \rev{If $(a_i,b_i)=(R,+)$, then $\lambda^{(M,i+1)}\succ'\lambda^{(M,i)}$.}
\end{enumerate}
\end{lemma}

\begin{proof}
See Section~3 of~\cite{LV21}.
\end{proof}

For a free-boundary dimer state $M$, let
\[
\lambda^{(l)}(M):=\lambda^{(M,l)},
\qquad
\lambda^{(r+1)}(M):=\lambda^{(M,r+1)}
\]
be the induced left and right boundary partitions. Define the weight of $M$ by
\[
w(M):=u^{|\lambda^{(l)}(M)|}v^{|\lambda^{(r+1)}(M)|}\prod_{i=l}^{r}x_i^{d_i(M)},
\]
where $d_i(M)$ is the total number of present diagonal edges of $M$ incident to an even
vertex with abscissa $2i$.

For fixed partitions $\lambda^{(l)},\lambda^{(r+1)}$, let
\[
Z_{\lambda^{(l)},\lambda^{(r+1)}}(G,\underline{x})
\]
be the sum of $w(M)$ over all free-boundary dimer states $M$ satisfying
\[
\lambda^{(l)}(M)=\lambda^{(l)},
\qquad
\lambda^{(r+1)}(M)=\lambda^{(r+1)}.
\]

In the doubly free-boundary case, we identify free-boundary dimer states that differ by a common
vertical translation. As will be shown in Lemma~\ref{le32}, each such translation class
contains a unique representative for which
\[
c^{(M,m)}=0,\qquad \forall m\in[l..r+1].
\]
Let $\mathcal{M}_{f,f}(G)$ denote the set of these canonical free-boundary dimer states. The
free-boundary partition function is
\[
Z_{f,f}(G,\underline{x})
:=
\sum_{M\in\mathcal{M}_{f,f}(G)} w(M)
=
\sum_{\lambda^{(l)},\lambda^{(r+1)}\in\YY}
Z_{\lambda^{(l)},\lambda^{(r+1)}}(G,\underline{x}),
\]
and the induced probability measure is
\begin{equation}
\Pr(M)=\frac{w(M)}{Z_{f,f}(G,\underline{x})}.
\label{ppd}
\end{equation}
In particular, pure free-boundary dimer states correspond to
\[
\lambda^{(l)}=\lambda^{(r+1)}=\emptyset.
\]

\begin{lemma}\label{l25}
For each pair of sequences $\underline a,\underline b$, for dimer configurations on $RYG(l,r,\underline a,\underline b)$ with boundary conditions in one of the following cases:
\begin{enumerate}[label=(\alph*)]
\item left boundary partition arbitrary and right boundary partition $\emptyset$;
\item both boundary partitions equal to $\emptyset$ (pure dimer coverings);
\item both boundary partitions arbitrary,
\end{enumerate}
there exist an integer $N$ and specializations $\rho_1,\ldots,\rho_{2N}$ (and parameters $u,v$) such that the corresponding dimer measure \eqref{ppd} with the specified boundary condition coincides
with the corresponding Schur-process-type measure \eqref{ppt}.
\end{lemma}

\begin{proof}
We construct \(\rho_1,\ldots,\rho_{2N}\) recursively. Let
\(\rho^{\varnothing}\) denote the empty-alphabet specialization, defined by
\[
p_n(\rho^{\varnothing})=0,\qquad n\ge 1.
\]
Equivalently, \(H(\rho^{\varnothing};z)=1\), and hence
\(s_{\lambda/\mu}(\rho^{\varnothing})=\mathbf 1_{\lambda=\mu}\).

Initialize \(i=l\) and \(j=1\). We say that \(j\) is compatible with the sign
\(b_i\) if either \(j\) is odd and \(b_i=+\), or \(j\) is even and \(b_i=-\).
For each \(i\in[l..r]\), if \(j\) is not compatible with \(b_i\), set
\[
\rho_j=\rho^{\varnothing}
\]
and replace \(j\) by \(j+1\). Now \(j\) is compatible with \(b_i\). We then set
the nontrivial specialization corresponding to column \(i\) by
\[
p_n(\rho_j)=
\begin{cases}
x_i^n, & a_i=L,\\
(-1)^{n+1}x_i^n, & a_i=R,
\end{cases}
\qquad n\ge 1.
\]
Then replace \(i\) by \(i+1\) and \(j\) by \(j+1\), and continue.

After the last column and when $i=r+1$, choose \(2N\) to be the smallest even integer satisfying
\(2N\ge j-1\), and set any remaining
\(\rho_j,\ldots,\rho_{2N}\) equal to \(\rho^{\varnothing}\). The inserted
empty-alphabet specializations contribute only identity skew-Schur factors, so
they do not change the measure; they only adjust the parity of the Schur-process
sequence. Finally, in case~\emph{(a)} we set \(v=0\), while in case~\emph{(b)}
we set \(u=v=0\). These choices force the corresponding endpoint partitions to
be empty in \eqref{ppt}, and the claimed identification follows.
\end{proof}

\begin{revision}
\label{rev:comment8}
\begin{proposition}\label{p26}
Let $RYG(l,r,\underline a,\underline b)$ be a rail-yard graph, and let $x_i$ denote the weight of the diagonal edges incident with the vertical line $x=2i$.
\begin{enumerate}
\item We first give the partition function for free-boundary dimer states on $RYG(l,r,\underline a,\underline b)$ with free boundary conditions on both the left boundary and the right boundary.
For compactness, write $I_\pm:=\{i\in[l..r]:b_i=\pm\}$ and set
\[
\mathsf c_{ij}(s):=
\begin{cases}
1+s x_i x_j,&a_i\ne a_j,\\[1mm]
(1-s x_i x_j)^{-1},&a_i=a_j.
\end{cases}
\]
Define the following three families of factors:
\begin{align*}
\mathsf A_k
&:=\prod_{i\in I_+}(1-u^{k-1}v^k x_i)^{-1}
   \prod_{i\in I_-}(1-u^k v^{k-1}x_i)^{-1},\\[2mm]
\mathsf B_k
&:=\prod_{\substack{i\in I_+\\j\in I_-}}
       \mathsf c_{ij}(u^{2k}v^{2k}),\\[2mm]
\mathsf C_k
&:=\prod_{\substack{i,j\in I_+\\i<j}}
       \mathsf c_{ij}(u^{2k-2}v^{2k})
   \prod_{\substack{i,j\in I_-\\i<j}}
       \mathsf c_{ij}(u^{2k}v^{2k-2}).
\end{align*}
Then the partition function is
\begin{equation}\label{ffp}
Z_{f,f}
=
\left(\prod_{\substack{i\in I_+,\ j\in I_-\\i<j}}z_{ij}\right)
\prod_{k\ge1}
\frac{\mathsf A_k\mathsf B_k\mathsf C_k}{1-(uv)^k}.
\end{equation}
Here $z_{ij}=\mathsf c_{ij}(1)$, that is,
\begin{equation}
 z_{ij}=\begin{cases}
 1+x_ix_j, & \text{if } a_i\neq a_j,\\[2pt]
 \dfrac{1}{1-x_ix_j}, & \text{if } a_i=a_j.
 \end{cases}
\label{dzij}
\end{equation}

\item The partition function for free-boundary dimer states on $RYG(l,r,\underline a,\underline b)$ with the free boundary condition on the left boundary and the empty partition on the right boundary is
\[
\begin{aligned}
Z_{f,\emptyset}
&=\left(\prod_{\substack{i\in I_+,\ j\in I_-\\i<j}}z_{ij}\right)
  \left(\prod_{i\in I_-}(1-u x_i)^{-1}\right)\\
&\qquad\times
  \prod_{\substack{i,j\in I_-\\i<j}}\mathsf c_{ij}(u^2).
\end{aligned}
\]
\end{enumerate}
\end{proposition}

\begin{proof}
Part~(1) follows from Proposition~\ref{p23} together with Lemma~\ref{l25}.
For part~(2), specialize \eqref{ffp} at $v=0$.
Then $(1-(uv)^k)^{-1}=1$ and $\mathsf B_k=1$ for every $k\ge1$.
All $\mathsf A_k$ become $1$ except for the $I_-$ factors at $k=1$,
whose product is $\prod_{i\in I_-}(1-u x_i)^{-1}$.
Similarly, all $\mathsf C_k$ become $1$ except for the $I_-$ pair factors
at $k=1$, whose product is
\[
\prod_{\substack{i,j\in I_-\\i<j}}\mathsf c_{ij}(u^2)
=
\frac{\displaystyle\prod_{\substack{i,j\in I_-\\i<j,\ a_i\ne a_j}}
 (1+u^2x_ix_j)}
 {\displaystyle\prod_{\substack{i,j\in I_-\\i<j,\ a_i=a_j}}
 (1-u^2x_ix_j)}.
\]
This proves the stated formula for $Z_{f,\emptyset}$.
\end{proof}
\end{revision}

Now we discuss the assumptions used to obtain the asymptotic results.

\begin{assumption}\label{ap5}
Let $\{RYG(l^{(\epsilon)},r^{(\epsilon)},\underline a^{(\epsilon)},\underline b^{(\epsilon)})\}_{\epsilon>0}$ be a sequence of rail-yard graphs. For each $\epsilon>0$, denote by $x_i^{(\epsilon)}$ the weight of the diagonal edges incident with the vertical line $x=2i$. Fix an integer period $n\ge 1$.
\begin{enumerate}
\item \label{ap41} \textbf{Piecewise periodicity of the graph.}
Fix $m\in\ZZ_{>0}$ and real numbers $V_0<V_1<\cdots<V_m$. We say that the sequence is piecewise $n$-periodic with transition points $V_0,\dots,V_m$ as $\epsilon\to 0$ if:
\begin{enumerate}
\item For each $\epsilon>0$, there exist integers $v_0^{(\epsilon)}<v_1^{(\epsilon)}<\cdots<v_m^{(\epsilon)}$ with
\[
 l^{(\epsilon)}=v_0^{(\epsilon)},\qquad r^{(\epsilon)}=v_m^{(\epsilon)},\qquad v_p^{(\epsilon)}\in n\ZZ\ \ (0\le p\le m),
\]
such that $\lim_{\epsilon\to 0}\epsilon v_p^{(\epsilon)}=V_p$ for all $p\in\{0,1,\dots,m\}$.

\item The sequence $\underline a^{(\epsilon)}$ is $n$-periodic on $[l^{(\epsilon)}..r^{(\epsilon)}]$ and does not depend on $\epsilon$. More precisely, there exist $a_1,\dots,a_n\in\{L,R\}$ such that for every integer $i$,
\[
 a_i^{(\epsilon)}=a_{i_{\equiv_n}},
\]
where $i_{\equiv_n}\in[n]:=\{1,2,\dots,n\}$ denotes the unique representative of $i$ modulo $n$ in $[n]$.

\item For each $p\in[m]:=\{1,2,\dots,m\}$, the sequence $\underline b^{(\epsilon)}$ is $n$-periodic on the open interval $(v_{p-1}^{(\epsilon)},v_p^{(\epsilon)})$ and does not depend on $\epsilon$ (but may depend on $p$). More precisely, there exist $b_{p,1},\dots,b_{p,n}\in\{+,-\}$ such that for all integers $i$ with $v_{p-1}^{(\epsilon)}<i<v_p^{(\epsilon)}$,
\[
 b_i^{(\epsilon)}=b_{p,i_{\equiv_n}}.
\]
\currentrev{\hypersetup{linkcolor=revisioncolor}The values of \(b_i^{(\epsilon)}\) at the transition indices
\(i=v_p^{(\epsilon)}\), including the two endpoints, may be chosen separately,
provided that the uniform admissibility conditions in
Assumption~\ref{ass:admissible} hold. These values are excluded from the
interior periodic classes used in the limiting formulas.}
\end{enumerate}

\item \label{ap43} \textbf{Periodicity of weights.}
The weights $x_i^{(\epsilon)}$ are periodic $q$-volume weights. Precisely,
\[
 x_i^{(\epsilon)}=\begin{cases}
 e^{-\epsilon(i-i_{\equiv_n})}\tau_k, & b_i^{(\epsilon)}=+,\\[2pt]
 e^{\epsilon(i-i_{\equiv_n})}\tau_k^{-1}, & b_i^{(\epsilon)}=-,
 \end{cases}
 \qquad \text{where } k=i_{\equiv_n}\in[n],
\]
and $\tau_1,\dots,\tau_n$ are positive constants independent of $\epsilon$.

\item \label{ap52} \textbf{Limit regime.}
\begin{enumerate}
\item Throughout the asymptotic results we take
\begin{equation}
        t=t^{(\epsilon)}=e^{-n\beta\epsilon},
\label{eq:exact-t-scaling}
\end{equation}
where $\beta>0$ is independent of $\epsilon$.  This exact normalization is used only to
identify the height Laplace kernel $t^{ky}$ with $e^{-n\beta k\eta}$ after the change of
variables $\eta=\epsilon y$; the one-point moment asymptotics themselves remain valid under
the weaker condition $-\log t/(n\epsilon)\to\beta$.

\item Let $s\in\ZZ_{>0}$. For each $d\in[s]:=\{1,\dots,s\}$, assume that $i_d^{(\epsilon)}\in[l^{(\epsilon)}..r^{(\epsilon)}]$ satisfies
\[
 \lim_{\epsilon\to 0}\epsilon i_d^{(\epsilon)}=\chi_d,
 \qquad \chi_1\le \chi_2\le \cdots\le \chi_s,
\]
and that the residue class modulo $n$ is fixed in $\epsilon$, i.e.\ there exist $i_1^*,\dots,i_s^*\in[n]$ such that
\[
 (i_d^{(\epsilon)})_{\equiv_n}=i_d^*.
\]
When $s=1$ we drop the index and write $\lim_{\epsilon\to 0}\epsilon i^{(\epsilon)}=\chi$ and $(i^{(\epsilon)})_{\equiv_n}=i^*$.
\end{enumerate}
\end{enumerate}
\end{assumption}

\begin{assumption}[Admissibility of weights]\label{ass:admissible}
For every finite graph considered below the parameters are chosen so that the free-boundary
partition function in Proposition~\ref{p26} is finite and all factors appearing in denominators
are nonzero. 

In the scaling regime of Assumption~\ref{ap5} we assume this admissibility uniformly in
$\epsilon$: there is a constant $\delta>0$ such that every denominator of the form
\begin{align*}
&1-x_i^{(\epsilon)}x_j^{(\epsilon)},\quad
1-u^{k-1}v^kx_i^{(\epsilon)},\quad
1-u^kv^{k-1}x_i^{(\epsilon)},\\
&1-u^{2k}v^{2k}x_i^{(\epsilon)}x_j^{(\epsilon)},\quad
1-u^{2k-2}v^{2k}x_i^{(\epsilon)}x_j^{(\epsilon)},\quad
1-u^{2k}v^{2k-2}x_i^{(\epsilon)}x_j^{(\epsilon)}
\end{align*}
which occurs in Proposition~\ref{p26} has absolute value at least $\delta$.  In particular,
whenever the finite factor $z_{ij}=(1-x_ix_j)^{-1}$ occurs, we assume
$x_i^{(\epsilon)}x_j^{(\epsilon)}\le 1-\delta$.  Every geometric parameter used in Section~\ref{sect:sampler} is at most
\(1-\delta\), and every Bernoulli parameter of the form \(\xi/(1+\xi)\) is finite
and positive.
\end{assumption}

\subsection{Main Results}

We state the main results in a form that separates the exact formulas, asymptotic theorems, frozen-boundary criterion, and sampling statement.
The first theorem is a law of large numbers for Laplace-transform observables.  The following theorem upgrades it, in the natural moment variable, to a weak slope-measure limit shape.  This is the analogue, for the present doubly free-boundary contour formulas, of the moment-determinacy step in Ahn's analysis of Macdonald plane partitions~\cite{Ah18}.

\begin{theorem}[Laplace-transform law of large numbers]\label{l61}
For each $\epsilon>0$, let $M=M^{(\epsilon)}$ be a random free-boundary dimer state on the rail-yard
graph $RYG(l^{(\epsilon)},r^{(\epsilon)},\underline a^{(\epsilon)},\underline b^{(\epsilon)})$ with
probability distribution given by \eqref{ppd}.  Let $h_M$ be the height function associated to
$M$ as defined in \eqref{dhm}.  Suppose Assumptions~\ref{ap5}, \ref{ass:admissible},
\ref{ap64}, and~\ref{ap65} hold at \(\chi\).  Fix
$\chi\in(V_0,V_m)\setminus\{V_p\}_{p=0}^m$ and choose columns $i^{(\epsilon)}$ with
$\epsilon i^{(\epsilon)}\to\chi$, fixed residue class modulo $n$, and
$a_{i^{(\epsilon)}}^{(\epsilon)}=L$.  For each $k\in\mathbb Z_{>0}$ put
$\alpha=\beta k$ and define
\[
\mathcal L_{\beta k}^{(\epsilon)}(\chi)
:=
\int_{-\infty}^{\infty}e^{-n\beta k\kappa}\,
\epsilon h_M^{\mathrm{col}}\!\left(i^{(\epsilon)},\frac{\kappa}{\epsilon}\right)d\kappa .
\]
Set
\[
        S_\chi(w):=\mathcal G_\chi(w)\prod_{r\ge1}\mathcal F_{u,v,r}(w).
\]
Then $\mathcal L_{\beta k}^{(\epsilon)}(\chi)$ converges in probability to the non-random limit
\begin{align}
\mathcal L_\alpha(\chi)
&=
\frac{1}{n^2\alpha^2\pi\mathbf{i}}
\oint_{\mathcal C}[S_\chi(w)]^{\alpha}\frac{dw}{w},
\label{lph}
\end{align}
where the contour $\mathcal C$ satisfies the conditions of Proposition~\ref{p57} and is
branch-admissible in the sense of Definition~\ref{def:branch-admissible}, and
$\mathcal{G}_{\chi}$ and $\mathcal{F}_{u,v,r}$ are defined by \eqref{dgc} and \eqref{dfuvk}.
Here $\alpha=\beta k$ and the power is taken with the branch-admissible logarithm of the
single product $S_\chi$, normalized to be real on the positive real axis when that axis lies in
the chosen component.  
\end{theorem}

\begin{theorem}[Limit shape in the natural Laplace scale]\label{thm:weak-limit-shape-beta}
Assume the hypotheses of Theorem~\ref{l61}.  For a marked $L$-column
$i^{(\epsilon)}$ with $\epsilon i^{(\epsilon)}\to\chi$, set
\begin{align}
        H_\epsilon(\kappa):=
        \epsilon h_M^{\mathrm{col}}\!\left(i^{(\epsilon)},\frac{\kappa}{\epsilon}\right).\label{dhe}
\end{align}
Let $dH_\epsilon$ be the distributional slope measure in the $\kappa$-direction and push it
forward by the natural moment variable
\[
        x=e^{-n\beta\kappa}.
\]
Thus, for Borel sets $B\subset[0,\infty)$, define
\begin{equation}\label{dnu-eps}
\nu_\chi^{(\epsilon)}(B)
:=
\int_{\{\kappa:\,e^{-n\beta\kappa}\in B\}}
        n\beta e^{-n\beta\kappa}\,dH_\epsilon(\kappa).
\end{equation}
Then $\nu_\chi^{(\epsilon)}$ converges weakly in probability to a deterministic compactly
supported finite measure $\nu_\chi$ on $[0,\infty)$.  The limiting measure is independent of
the fixed residue class of $i^{(\epsilon)}$ modulo $n$.  Let \(S_\chi\) be defined in \eqref{dsc}.
Let \(\mathcal C_\chi\), \(U_\chi\), and \(L_\chi\) be supplied by
Corollary~\ref{cor:automatic-branch}, and define
\begin{equation}\label{eq:Tchi}
T_\chi(w)
:=
\exp\!\bigl(\beta L_\chi(w)\bigr),
\qquad w\in U_\chi.
\end{equation}  The moments of \(\nu_\chi\) are
\begin{equation}\label{nu-moments}
\int_0^\infty x^{k-1}\nu_\chi(dx)
=
\frac{1}{k\pi\mathbf i}
\oint_{\mathcal C_\chi}
T_\chi(w)^k\frac{dw}{w},
\qquad k\in\mathbb Z_{>0}.
\end{equation}
The macroscopic height profile in the column coordinate is recovered by
\begin{equation}\label{def-H-from-nu}
\mathcal H(\chi,\kappa)
:=
\frac{1}{n\beta}
\int_{(e^{-n\beta\kappa},\infty)}\frac{1}{x}\,\nu_\chi(dx).
\end{equation}
The function
$\mathcal H(\chi,\cdot)$ is locally absolutely continuous and
\[
        0\le \partial_\kappa\mathcal H(\chi,\kappa)\le 2
\]
for a.e. $\kappa$.  At Lebesgue points of the density of $\nu_\chi$, if
$x=e^{-n\beta\kappa}$ and $d\nu_\chi(x)=f_\chi(x)\,dx$, then
\begin{equation}\label{slope-from-nu-density}
        \partial_\kappa\mathcal H(\chi,\kappa)=f_\chi(e^{-n\beta\kappa}).
\end{equation}
\end{theorem}

\begin{definition}\label{df29}
Let \(\mathcal H\) be the deterministic limit shape obtained from
Theorem~\ref{thm:weak-limit-shape-beta}.  A point
\((\chi,\kappa)\), with
\(\chi\in (V_0,V_m)\setminus\{V_p\}_{p=0}^m\), is called a regular slope point if,
writing \(x=e^{-n\beta\kappa}\), \(x\) is a Lebesgue point of the density
\(f_\chi=d\nu_\chi/dx\), and the Stieltjes boundary value satisfies
\[
\partial_\kappa \mathcal H(\chi,\kappa)
=
f_\chi(x)
=
-\frac{1}{\pi}
\lim_{\delta\downarrow0}\Im \operatorname{St}_{\nu_\chi}(x+i\delta).
\]  The liquid region is
\[
 \mathcal{L}:=\left\{(\chi,\kappa):
 \frac{\partial\mathcal{H}}{\partial\kappa}(\chi,\kappa)\in(0,2)\right\},
\]
and the frozen region is
\[
 \left\{(\chi,\kappa):
 \frac{\partial\mathcal{H}}{\partial\kappa}(\chi,\kappa)\in\{0,2\}\right\}.
\]
The frozen boundary is the interface between the liquid and frozen regular slope
points. 
\end{definition}

The root-count assertions in the next theorem are understood through the finite truncations
$S_{\chi,K}$ and the no-escape/Hurwitz passage in Section~\ref{sect:fb}; we do not use a
polynomial root count directly for the infinite product.

\begin{theorem}[Regular frozen-boundary equation]
\label{p412}
Assume \(\beta=1\) and
Assumptions~\ref{ap5}, \ref{ass:admissible},
\ref{ap64}, and \ref{ap65}. Let \(H\) be the
limit shape supplied by Theorem~\ref{thm:weak-limit-shape-beta}.
Let
\[
p_0=(\chi_0,\kappa_0)
\]
be a point on the frozen boundary of Definition~\ref{df29}, where
\[
\chi_0\in
(V_0,V_m)\setminus\{V_0,\ldots,V_m\},
\qquad
x_0:=e^{-n\kappa_0}.
\]
Let \(S_\chi\) be the spectral product defined in \eqref{dsc}, and
let \(S_{\chi,K}\) and \(C_{\chi,K}\) be defined by
\eqref{dsck}.

 Suppose
further that there exists a neighborhood \(U_0\) of \(p_0\) such
that, for every regular slope point
\[
p=(\chi,\kappa)\in U_0
\]
in the sense of Definition~\ref{df29}, writing
\[
x=e^{-n\kappa},
\]
one has
\[
S_\chi(0)\neq x,
\]
and there exists \(K(p)\geq1\) such that, for all \(K\geq K(p)\),
\[
D_{1,K}(\chi)\neq\varnothing,
\qquad
C_{\chi,K}\neq x.
\]
Here \(D_{1,K}(\chi)\) denotes the inward pole set defined in
Assumption~\ref{ap62}; under Assumptions~\ref{ap64}--\ref{ap65},
Assumption~\ref{ap62} is available by Lemma~\ref{l64}.

Then there exists \(w_0\in\mathbb R\setminus\{0\}\) such that
\begin{equation}
\label{fb}
S_{\chi_0}(w_0)=e^{-n\kappa_0},
\qquad
\frac{S_{\chi_0}'(w_0)}{S_{\chi_0}(w_0)}=0.
\end{equation}
Equivalently, \(w_0\) is a multiple real root of
\[
S_{\chi_0}(w)=e^{-n\kappa_0}.
\]
\end{theorem}

\begin{theorem}\label{t77}
Let $\{RYG(l^{(\epsilon)},r^{(\epsilon)},\underline a^{(\epsilon)},\underline b^{(\epsilon)})\}_{\epsilon>0}$
be a sequence of rail-yard graphs. Suppose Assumptions~\ref{ap5}, \ref{ass:admissible},
\ref{ap64}, and~\ref{ap65} hold at each marked location.
For each $\epsilon>0$, let $M=M^{(\epsilon)}$ be the canonical free-boundary dimer state sampled from the probability measure
\eqref{ppd}. Fix finitely many $\chi_d\in(V_0,V_m)\setminus\{V_p\}_{p=0}^{m}$ and
positive integers $g_d$.  Choose columns $i_d^{(\epsilon)}$ with
\[
\epsilon i_d^{(\epsilon)}\to\chi_d,
\qquad
(i_d^{(\epsilon)})_{\equiv_n}\text{ fixed},
\qquad
 a_{i_d^{(\epsilon)}}^{(\epsilon)}=L.
\]
Define
\[
\mathcal X_{g_d}^{(\epsilon)}(\chi_d)
:=
\int_{-\infty}^{\infty}
\left(
 h_M^{\mathrm{col}}\!\left(i_d^{(\epsilon)},\frac{\eta}{\epsilon}\right)
 -
 \mathbb E\left[h_M^{\mathrm{col}}\!\left(i_d^{(\epsilon)},\frac{\eta}{\epsilon}\right)\right]
\right)e^{-n\beta g_d\eta}\,d\eta .
\]
Then the vector of these observables converges in distribution to a centered Gaussian
vector.  Its covariance is
\begin{align}
&\mathrm{Cov}\bigl[\mathcal X_{g_d}(\chi_d),\mathcal X_{g_h}(\chi_h)\bigr]
\notag\\
&\quad=
\frac{4}{n^2\beta^2g_dg_h(2\pi\mathbf i)^2}
\oint_{\mathcal C_d}\oint_{\mathcal C_h}
\Phi_{\chi_d,g_d}(z)\,
\Phi_{\chi_h,g_h}(w)\,
\mathsf K_{LL}(z,w)\,dz\,dw,
\label{eq:main-height-cov}
\end{align}
where $\Phi_{\chi,g}$ and the $L$-chart kernel $\mathsf K_{LL}$ are defined in
Definition~\ref{def:KLL}. The contours and logarithms are supplied by
Corollary~\ref{cor:automatic-branch}, and the contours are chosen
pairwise \(LL\)-admissibly as in that corollary. 
\end{theorem}

\section{Moment Formula for Height Function}\label{sect:mh}

In this section we compute moments of the height function for free-boundary dimer states on rail-yard
graphs by evaluating suitable observables in Macdonald processes with dual specialization;
see Lemma~\ref{l31}. Our approach builds on the integral representation for the Laplace
transform of the height function in \cite{GZ16}, which was used to obtain height-function
moments for perfect matchings on the hexagonal lattice with empty boundary conditions
in \cite{Ah18}, and, via dual specialization, for rail-yard graphs with empty boundary
conditions in \cite{LV21}. In the doubly free-boundary case, however, additional infinite
products appear, and several auxiliary combinatorial identities are needed; see
Lemmas~\ref{211}--\ref{l218}.

\subsection{Height function and Laplace transform}

For a partition $\lambda$, write $l(\lambda)$ for the number of its nonzero parts.

\begin{lemma}\label{l29}
Let $M$ be a free-boundary dimer state of $RYG(l,r,\underline a,\underline b)$, and let
$x=2m-\frac12$. Then
\begin{equation}\label{lts}
\int_{-\infty}^{\infty}h_M(x,y)t^{k(y-c^{(M,m)})}\,dy
=
\frac{2}{(k\log t)^2}
\left[
t^{-k l(\lambda^{(M,m)})}
+
(1-t^{-k})\sum_{i=1}^{l(\lambda^{(M,m)})}
t^{k(\lambda_i^{(M,m)}-i+1)}
\right].
\end{equation}
Here $c^{(M,m)}$ is the central charge defined in \eqref{dcg}.
\end{lemma}

\begin{proof}
Following the same argument as in Section~2.4 of \cite{LV21}, we obtain
\[
\int_{-\infty}^{\infty}h_M(x,y)t^{ky}\,dy
=
\frac{2t^{k c^{(M,m)}}}{(k\log t)^2}
\left[
t^{-k l(\lambda^{(M,m)})}
+
(1-t^{-k})\sum_{i=1}^{l(\lambda^{(M,m)})}
t^{k(\lambda_i^{(M,m)}-i+1)}
\right].
\]
Replacing $y$ by $y-c^{(M,m)}$ gives \eqref{lts}.
\end{proof}

The next lemma shows that the charge is independent of the odd column.

\begin{lemma}\label{le32}
Let $M$ be a free-boundary dimer state of $RYG(l,r,\underline a,\underline b)$. Then
$c^{(M,m)}$ is independent of $m\in[l..r+1]$.
\end{lemma}

\begin{proof}
It suffices to prove that
\begin{equation}\label{ace}
c^{(M,m)}=c^{(M,m+1)},\qquad \forall\,m\in[l..r].
\end{equation}

Fix $m\in[l..r]$. Since $M$ contains only finitely many diagonal edges, by applying
a common vertical translation to the whole configuration we may assume that every present
diagonal edge incident to the even column $x=2m$ lies in the upper half-plane. This
translation does not affect the argument below.

Let $\mathcal D_{M,m}$ be the set of present diagonal edges in $M$ incident to an even
vertex with abscissa $2m$. Since $\mathcal D_{M,m}$ is finite, it can be written as a disjoint
union of maximal consecutive blocks
\[
\mathcal D_{M,m}
=
\mathcal D_{M,m}^{(1)}\sqcup\cdots\sqcup \mathcal D_{M,m}^{(s)},
\]
where for each $i\in[s]$ there exist integers $p_{i,-}\le p_{i,+}$ such that
$\mathcal D_{M,m}^{(i)}$ consists exactly of the present diagonal edges adjacent to even
vertices
\[
\left(2m,p+\frac12\right),
\qquad p\in[p_{i,-}..p_{i,+}],
\]
and
\[
p_{1,-}\le p_{1,+}<p_{2,-}\le p_{2,+}<\cdots<p_{s,-}\le p_{s,+},
\qquad
p_{i,+}+1<p_{i+1,-}\quad (i\in[s-1]).
\]

For $j\in\ZZ$, define odd vertices on the neighboring columns by
\[
v_1^{(j)}:=\left(2m-1,j+\frac12\right),
\qquad
v_2^{(j)}:=\left(2m+1,j+\frac12\right).
\]
For each block $i\in[s]$, define
\[
q_{i,-}=
\begin{cases}
p_{i,-}, & b_m=+,\\
p_{i,-}-1, & b_m=-,
\end{cases}
\qquad
q_{i,+}=
\begin{cases}
p_{i,+}+1, & b_m=+,\\
p_{i,+}, & b_m=-.
\end{cases}
\]

A direct local inspection of the four possibilities
$(a_m,b_m)\in\{L,R\}\times\{+,-\}$ shows the following.
\begin{enumerate}
\item If $j\notin\{q_{i,-},q_{i,+}\}_{i\in[s]}$, then either both $v_1^{(j)}$ and $v_2^{(j)}$
are particles, or both are holes.
\item If $b_m=+$, then for each $i\in[s]$:
at height $q_{i,-}$ the pair $(v_1^{(q_{i,-})},v_2^{(q_{i,-})})$ is
(particle, hole), while at height $q_{i,+}$ it is (hole, particle).
\item If $b_m=-$, then for each $i\in[s]$:
at height $q_{i,-}$ the pair is (hole, particle), while at height $q_{i,+}$ it is
(particle, hole).
\end{enumerate}

For \(j\in\mathbb Z\), let
\[
\Delta_j:=
\mathbf 1_{\{v_1^{(j)}\text{ is a particle}\}}
-
\mathbf 1_{\{v_2^{(j)}\text{ is a particle}\}}.
\]
Equivalently, \(\Delta_j\) is the contribution at height \(j+\frac12\) to
\(c^{(M,m)}-c^{(M,m+1)}\).  By~(1), all heights other than \(q_{i,-}+\frac12\) and
\(q_{i,+}+\frac12\), \(i\in[s]\), contribute \(0\) to
\(c^{(M,m)}-c^{(M,m+1)}\).  By~(2) and~(3), the two heights associated
with each block contribute \(+1\) and \(-1\), in some order.  Hence
\(c^{(M,m)}=c^{(M,m+1)}\).
\end{proof}

From now on, in the doubly free-boundary setting we work with free-boundary dimer states modulo
a common vertical translation. Since a common vertical translation does not change the induced sequence of
partitions or the Boltzmann weight, and an upward translation by
\(k\in\mathbb Z\) changes every column charge by \(k\), Lemma~\ref{le32}
implies that each translation class contains a unique representative satisfying
\[
c^{(M,m)}=0,\qquad m\in[l..r+1].
\]
We always choose this canonical representative. For this representative, Lemma~\ref{l29}
reduces to
\begin{equation}\label{lhg}
\int_{-\infty}^{\infty}h_M(x,y)t^{ky}\,dy
=
\frac{2}{(k\log t)^2}
\left[
t^{-k l(\lambda^{(M,m)})}
+
(1-t^{-k})\sum_{i=1}^{l(\lambda^{(M,m)})}
t^{k(\lambda_i^{(M,m)}-i+1)}
\right].
\end{equation}

\subsection{Macdonald processes}

By Lemma~\ref{l25}, to study the double free-boundary Gibbs measure on free-boundary dimer states
it suffices to study the induced probability measure on the corresponding sequence of
partitions. We now recall the relevant generalized Macdonald process.

\begin{definition}\label{df22}
Let
\[
\bA=(A^{(l)},A^{(l+1)},\ldots,A^{(r)},A^{(r+1)}),
\qquad
\bB=(B^{(l)},B^{(l+1)},\ldots,B^{(r)},B^{(r+1)}),
\]
where each $A^{(i)}$ and $B^{(j)}$ is a countable set of variables, and let $u,v$ be two
parameters. Let $\mathcal P=\{\mathcal L,\mathcal R\}$ be a partition of $[l..r]$, i.e.
\[
\mathcal L\cup\mathcal R=[l..r],
\qquad
\mathcal L\cap\mathcal R=\emptyset.
\]
Define a formal probability measure on sequences of $(r-l+2)$ partitions
\[
\bigl(\lambda^{(l)},\lambda^{(l+1)},\ldots,\lambda^{(r)},\lambda^{(r+1)}\bigr)
\]
by
\begin{align}\label{pmd}
&\MP_{\bA,\bB,\mathcal P;q,t}(\lambda^{(l)},\ldots,\lambda^{(r+1)})
\propto
u^{|\lambda^{(l)}|}v^{|\lambda^{(r+1)}|}\\
&\times\left[\prod_{i\in\mathcal L}
\Psi_{\lambda^{(i)},\lambda^{(i+1)}}(A^{(i)},B^{(i+1)};q,t)\right]
\left[\prod_{j\in\mathcal R}
\Phi_{[\lambda^{(j)}]',[\lambda^{(j+1)}]'}(A^{(j)},B^{(j+1)};q,t)\right],\notag
\end{align}
where
\[
\Psi_{\lambda,\mu}(A,B;q,t)
=
\sum_{\nu\in\YY}P_{\lambda/\nu}(A;q,t)\,Q_{\mu/\nu}(B;q,t),
\]
and
\[
\Phi_{\lambda,\mu}(A,B;q,t)
=
\sum_{\nu\in\YY}Q_{\lambda/\nu}(A;t,q)\,P_{\mu/\nu}(B;t,q).
\]
\end{definition}

See Section~\ref{sc:dmp} for the definitions of Macdonald polynomials
$P_\lambda$, $Q_\lambda$, $P_{\lambda/\mu}$, and $Q_{\lambda/\mu}$.

\begin{remark}
In terms of the Macdonald scalar product \eqref{dsp}, we also have
\[
\Psi_{\lambda,\mu}(A,B;q,t)
=
\langle P_\lambda(A,Y;q,t),Q_\mu(Y,B;q,t)\rangle_Y,
\]
and
\[
\Phi_{\lambda,\mu}(A,B;q,t)
=
\langle Q_\lambda(A,Y;t,q),P_\mu(Y,B;t,q)\rangle_Y,
\]
where $Y$ is a countable set of variables.
\end{remark}

\begin{lemma}\label{l23}
Let $M$ be sampled from the double free-boundary Gibbs measure \eqref{ppd} on
$RYG(l,r,\underline a,\underline b)$. Then the induced sequence of partitions
\[
\bigl(\lambda^{(M,l)},\lambda^{(M,l+1)},\ldots,\lambda^{(M,r)},\lambda^{(M,r+1)}\bigr)
\]
forms a generalized Macdonald process of the form \eqref{pmd} with
\begin{enumerate}
\item $\mathcal L=\{i\in[l..r]:a_i=L\}$ and
$\mathcal R=\{i\in[l..r]:a_i=R\}$;
\item for each $i\in[l..r]$,
\[
A^{(i)}=\{x_i\},\ B^{(i+1)}=\{0\}\quad\text{if }b_i=-,
\qquad
A^{(i)}=\{0\},\ B^{(i+1)}=\{x_i\}\quad\text{if }b_i=+;
\]
\item $q=t$.
\end{enumerate}
\end{lemma}

\begin{proof}
This is exactly the specialization of Lemma~\ref{l25}(c) to the Schur case $q=t$, using
the standard fact that Macdonald polynomials reduce to Schur polynomials when $q=t$.
\end{proof}

For $k\in\ZZ_{>0}$ and $\lambda\in\YY$, define
\begin{equation}\label{dgm}
\gamma_k(\lambda;q,t)
=
(1-t^{-k})\sum_{i=1}^{l(\lambda)}q^{k\lambda_i}t^{k(-i+1)}+t^{-k l(\lambda)}.
\end{equation}
By \eqref{lhg}, for the canonical representative of a doubly free-boundary dimer state,
\[
\int_{-\infty}^{\infty}h_M(x,y)t^{ky}\,dy
=
\frac{2}{(k\log t)^2}\gamma_k(\lambda^{(M,m)};t,t).
\]

Define
\begin{align}
\Pi_{L,L}(X,Y)&=\Pi(X,Y;q,t),\notag\\
\Pi_{R,R}(X,Y)&=\Pi(X,Y;t,q),\label{PiLR}\\
\Pi_{L,R}(X,Y)&=\Pi_{R,L}(X,Y)=\Pi'(X,Y),\notag
\end{align}
where $\Pi$ and $\Pi'$ are defined by \eqref{defPPprime}. We also write
\[
L'=R,\qquad R'=L.
\]

Recall that the power sums $p_k$ are defined in \eqref{dpn}. For a countable set of
variables $X=(x_1,x_2,\ldots)$, write
\[
X^{-1}:=(x_1^{-1},x_2^{-1},\ldots).
\]
\begin{revision}
\label{rev:comment9}
For $W=(w_1,\ldots,w_k)$, define the kernel
\begin{equation}\label{dh}
L(W,X;q,t):=\prod_{i=1}^{k}\prod_{x\in X}
\frac{w_i-(q/t)x}{w_i-qx}.
\end{equation}
\end{revision}

\begin{lemma}\label{l210}
For any countable set of variables $X$ and any $c\in\{L,R\}$, define specializations by
\[
\rho_{X,1,c}(p_k):=
(-1)^{\sigma(c)}p_k\!\bigl[(-1)^{\sigma(c)}X\bigr],
\]
and
\[
\rho_{X,t,2,c}(p_k):=
(-1)^{\sigma(c)}
\bigl(1-t^{(2\sigma(c)-1)k}\bigr)
p_k\!\bigl((-1)^{\sigma(c)}t^{1-2\sigma(c)}X^{-1}\bigr),
\]
where
\[
\sigma(c)=
\begin{cases}
0,& c=L,\\
1,& c=R.
\end{cases}
\]
Then:
\begin{align*}
\rho_{X,t,2,c}
&=
[t^{\,1-2\sigma(c)}]\rho_{X^{-1},1,c}\cup
((-1)*\rho_{X^{-1},1,c}),\\[1mm]
\Pi(X,Y;t,t)
&=
H(\rho_{X,1,c};\rho_{Y,1,c}),\\[1mm]
\Pi'(X,Y)
&=
H(\rho_{X,1,c};\rho_{Y,1,c'}),\\[1mm]
\left(L\bigl(W,(-1)^{\delta_{c,d}-1}A;\omega(t,t,c)\bigr)\right)^{(-1)^{\delta_{c,d}-1}}
&=
H(\rho_{W,t,2,c};\rho_{A,1,d}),\\[1mm]
\frac{\Pi_{c,d}(B,W)}{\Pi_{c,d}(B,\xi(t,t,c)W)}
&=
H((-1)*\rho_{W^{-1},t^{-1},2,c};\rho_{B,1,d}),
\end{align*}
where $\delta_{c,d}$ is the Kronecker delta and
\begin{equation}\label{dox}
\omega(q,t,c)=
\begin{cases}
(q,t),& c=L,\\[1mm]
\left(\dfrac{1}{t},\dfrac{1}{q}\right),& c=R,
\end{cases}
\qquad
\xi(q,t,c)=
\begin{cases}
q^{-1},& c=L,\\
t,& c=R.
\end{cases}
\end{equation}
\end{lemma}

\begin{proof}
All identities follow directly from Lemma~\ref{la6}.
\end{proof}

\begin{lemma}\label{l41}
For every partition $\lambda$ and every $k\ge1$,
\[
\gamma_k(\lambda';t,q)
=
\gamma_k\!\left(\lambda;\frac{1}{q},\frac{1}{t}\right).
\]
\end{lemma}

\begin{proof}
The proof is the same as that of Lemma~4.1 in \cite{LV21}.
\end{proof}

We now state the operator formula for the observables \eqref{dgm}.

\begin{lemma}\label{l31}
Let $l_i$ be nonnegative integers for $i\in[l+1..r]$, and let
$W^{(i)}$ be a set of variables with $|W^{(i)}|=l_i$. Define
\[
\rho_A:=\bigcup_{i=l}^{r}\rho_{A^{(i)},1,a_i},
\qquad
\rho_B:=\bigcup_{i=l+1}^{r+1}\rho_{B^{(i)},1,a_{i-1}},
\]
and
\[
\rho_W:=\bigcup_{i=l+1}^{r}\rho_{W^{(i)},t,2,a_i},
\qquad
\rho_W^\circ:=\bigcup_{i=l+1}^{r}(-1)*\rho_{[W^{(i)}]^{-1},t^{-1},2,a_i}.
\]
Then
\begin{align*}
&\left.
\mathbb E_{\MP_{\bA,\bB,\mathcal P;q,t}}
\left[\prod_{i=l+1}^{r}\gamma_{l_i}(\lambda^{(i)};q,t)\right]
\right|_{q=t}\\
&=
\oint\cdots\oint
\left[\prod_{i=l+1}^{r}D(W^{(i)};\omega(t,t,a_i))\right]
\left[\prod_{l+1\le i\le j\le r}
H(\rho_{W^{(i)},t,2,a_i};\rho_{A^{(j)},1,a_j})\right]\\
&\qquad\times
\left[\prod_{l+1\le i<j\le r}
H((-1)*\rho_{[W^{(j)}]^{-1},t^{-1},2,a_j};\rho_{B^{(i)},1,a_{i-1}})\right]
\left[\prod_{l+1\le i<j\le r}T_{a_i,a_j}(W^{(i)},W^{(j)})\right]\\
&\qquad\times
\prod_{k\ge1}
\widetilde{H}([u^{k-1}v^k]\rho_W)\,
\widetilde{H}([u^k v^{k-1}]\rho_W^\circ)\,
H(u^{2k}\rho_A;v^{2k}\rho_W)\,
H(u^{2k}\rho_A;v^{2k-2}\rho_W^\circ)\\
&\qquad\times
H(u^{2k}\rho_B;v^{2k}\rho_W^\circ)\,
H(u^{2k-2}\rho_B;v^{2k}\rho_W)\,
H(u^{2k}\rho_W;v^{2k}\rho_W^\circ),
\end{align*}
where
\[
T_{c,d}(Z,W):=
\begin{cases}
\displaystyle
\prod_{z_i\in Z}\prod_{w_j\in W}
\frac{(1-w_jz_i^{-1})^2}{(1-t^{-1}w_jz_i^{-1})(1-tw_jz_i^{-1})},
& c=d,\\[4mm]
\displaystyle
\prod_{z_i\in Z}\prod_{w_j\in W}
\frac{(1+tw_jz_i^{-1})^2}{(1+t^2w_jz_i^{-1})(1+w_jz_i^{-1})},
& c=L,\ d=R,\\[4mm]
\displaystyle
\prod_{z_i\in Z}\prod_{w_j\in W}
\frac{(1+t^{-1}w_jz_i^{-1})^2}{(1+t^{-2}w_jz_i^{-1})(1+w_jz_i^{-1})},
& c=R,\ d=L.
\end{cases}
\]
The multiple contour integral is over all variables in the sets $W^{(i)}$, and the contour
for each variable $w_s^{(i)}\in W^{(i)}$ is positively oriented, encloses $0$ and all poles of
the integrand involving $w_s^{(i)}$, and is contained in the domain bounded by
$t\mathcal C_{i',s'}$ whenever $(i,s)<(i',s')$ in lexicographic order.
\end{lemma}

We first record four auxiliary identities.

\begin{lemma}\label{211}
Let $\{d_k\}_{k\ge1}$ and $\{s_k\}_{k\ge1}$ be sequences in a graded algebra $\mathcal A$,
and let $\{u_k\}_{k\ge1}$ be a sequence in a graded algebra $\mathcal B$, such that
\[
\lim_{k\to\infty}\mathrm{ldeg}(d_k)=\infty,
\qquad
\lim_{k\to\infty}\mathrm{ldeg}(s_k)=\infty,
\qquad
\lim_{k\to\infty}\mathrm{ldeg}(u_k)=\infty,
\]
where $\mathrm{ldeg}$ is defined in \eqref{dldeg}. Then
\[
\left\langle
\exp\left(\sum_{k=1}^{\infty}\frac{d_kp_k(Y)+s_k[p_k(Y)]^2}{k}\right),
\exp\left(\sum_{k=1}^{\infty}\frac{u_kp_k(Y)}{k}\right)
\right\rangle_{Y;q=t}
=
\exp\left(\sum_{k=1}^{\infty}\frac{d_ku_k+s_ku_k^2}{k}\right).
\]
\end{lemma}

\begin{proof}
When $s_k=0$, this is Proposition~2.3 of \cite{bcgs13}. For general $s_k$, expand the
first exponential into monomials in the power sums and use the definition of the Macdonald
scalar product \eqref{dsp}. More precisely, if
\[
f(Y)=
c_{(k_1,l_1),\ldots,(k_m,l_m)}
[p_{k_1}(Y)]^{l_1}\cdots [p_{k_m}(Y)]^{l_m},
\qquad
k_1<\cdots<k_m,
\]
is a monomial appearing in the first exponential, then in
\[
R(Y):=
\exp\left(\sum_{k=1}^{\infty}\frac{u_kp_k(Y)}{k}\right)
=
\sum_{n=0}^{\infty}\frac{1}{n!}
\left(\sum_{k=1}^{\infty}\frac{u_kp_k(Y)}{k}\right)^n
\]
the same monomial appears with coefficient
\[
\frac{u_{k_1}^{l_1}\cdots u_{k_m}^{l_m}}{l_1!\cdots l_m!\,k_1^{l_1}\cdots k_m^{l_m}}.
\]
Hence
\[
\langle f(Y),R(Y)\rangle_{Y;q=t}
=
c_{(k_1,l_1),\ldots,(k_m,l_m)}\,u_{k_1}^{l_1}\cdots u_{k_m}^{l_m},
\]
which is exactly what one obtains by replacing each $p_{k_i}(Y)$ by $u_{k_i}$.
Summing over all monomials proves the claim.
\end{proof}

\begin{lemma}\label{l214}
Let $\rho_1,\rho_2,\rho_3$ be specializations such that $\rho_1$ and $\rho_3$ are
independent of $Y$.
\begin{enumerate}
\item If $\rho_2(p_k)=p_k(Y)$ for all $k\ge1$, then
\[
\left\langle
\widetilde{H}(\rho_1\cup\rho_2),
\exp\left(\sum_{k=1}^{\infty}\frac{\rho_3(p_k)p_k(Y)}{k}\right)
\right\rangle_{Y;q=t}
=
\widetilde{H}(\rho_1\cup\rho_3).
\]
\item If $\rho_2(p_k)=(-1)^{k+1}p_k(Y)$ for all $k\ge1$, then
\[
\left\langle
\widetilde{H}(\rho_1\cup\rho_2),
\exp\left(\sum_{k=1}^{\infty}\frac{(-1)^{k+1}\rho_3(p_k)p_k(Y)}{k}\right)
\right\rangle_{Y;q=t}
=
\widetilde{H}(\rho_1\cup\rho_3).
\]
\end{enumerate}
\end{lemma}

\begin{proof}
This follows immediately from Definition~\ref{dht} of $\widetilde H$ and
Lemma~\ref{211}.
\end{proof}

\begin{lemma}\label{l217}
Let $\rho$ be a specialization. Then
\[
\sum_{\lambda\in\YY}s_{\lambda/\mu}(\rho)\,u^{|\lambda|}
=
\widetilde{H}(u\rho)\sum_{\xi\in\YY}s_{\mu/\xi}(u^2\rho)\,u^{|\xi|}.
\]
\end{lemma}

\begin{proof}
From Example~I.5.23(a)(3) in \cite{IGM15}, for a countable set of variables $X$ we have
\begin{equation}\label{ig1}
\sum_{\lambda\in\YY}s_{\lambda/\mu}(X)
=
\frac{1}{\prod_{x_i\in X}(1-x_i)\prod_{x_i,x_j\in X,\ i<j}(1-x_ix_j)}
\sum_{\xi\in\YY}s_{\mu/\xi}(X).
\end{equation}
Replace each $p_k(X)$ by $u^k\rho(p_k)$. Since skew Schur functions are homogeneous,
\eqref{ig1} becomes
\[
\sum_{\lambda\in\YY}u^{|\lambda|-|\mu|}s_{\lambda/\mu}(\rho)
=
\widetilde{H}(u\rho)\sum_{\xi\in\YY}u^{|\mu|-|\xi|}s_{\mu/\xi}(\rho).
\]
Multiplying both sides by $u^{|\mu|}$ gives
\[
\sum_{\lambda\in\YY}u^{|\lambda|}s_{\lambda/\mu}(\rho)
=
\widetilde{H}(u\rho)\sum_{\xi\in\YY}u^{2|\mu|-|\xi|}s_{\mu/\xi}(\rho)
=
\widetilde{H}(u\rho)\sum_{\xi\in\YY}u^{|\xi|}s_{\mu/\xi}(u^2\rho),
\]
as claimed.
\end{proof}

\begin{lemma}\label{l218}
Let $\rho_1,\rho_2$ be specializations. Let $Y$ be a countable set of variables, and let
$\rho$ be the specialization defined by
\[
\rho(p_k)=p_k(Y),\qquad k\ge1.
\]
Then for $u,v\in(0,1)$,
\[
\left\langle
\widetilde{H}(u[\rho_1\cup\rho]),
\widetilde{H}(v[\rho_2\cup\rho])
\right\rangle_{Y;q=t}
=
\prod_{n\ge1}
\frac{
\widetilde{H}([u^{n-1}v^n]\rho_2)\,
\widetilde{H}([u^nv^{n-1}]\rho_1)\,
H([u^{2n}]\rho_1;[v^{2n}]\rho_2)
}{1-u^nv^n}.
\]
\end{lemma}

\begin{proof}
Using the Cauchy identity for Schur functions, we obtain
\begin{align*}
&\left\langle
\widetilde{H}(u[\rho_1\cup\rho]),
\widetilde{H}(v[\rho_2\cup\rho])
\right\rangle_{Y;q=t}\\
&=
\sum_{\lambda,\mu\in\YY}
\left\langle
u^{|\lambda|}s_\lambda([\rho_1\cup\rho]),
v^{|\mu|}s_\mu([\rho_2\cup\rho])
\right\rangle_{Y;q=t}\\
&=
\sum_{\lambda,\mu,\nu\in\YY}
u^{|\lambda|}v^{|\mu|}s_{\lambda/\nu}(\rho_1)s_{\mu/\nu}(\rho_2).
\end{align*}
Apply Lemma~\ref{l217} to the sum over $\lambda$:
\[
=
\widetilde{H}(u\rho_1)
\sum_{\nu,\xi,\mu\in\YY}
u^{|\xi|}v^{|\mu|}
s_{\nu/\xi}(u^2\rho_1)s_{\mu/\nu}(\rho_2).
\]
Now use the skew Cauchy identity (Example~I.5.26(1) of \cite{IGM15}) to interchange
the two skew Schur functions, obtaining
\[
=
\widetilde{H}(u\rho_1)\widetilde{H}(v\rho_2)
\widetilde{H}([u^2v]\rho_1)\widetilde{H}([uv^2]\rho_2)
H([u^2v^2]\rho_1;\rho_2)\,S_1,
\]
where
\[
S_1
=
\sum_{\nu,\xi,\tau\in\YY}
u^{|\xi|}v^{|\tau|}
s_{\nu/\xi}([u^2v^2]\rho_1)\,
s_{\nu/\tau}([u^2v^2]\rho_2).
\]
Repeating the same argument with
\[
\rho_1^{(1)}=[u^2v^2]\rho_1,
\qquad
\rho_2^{(1)}=[u^2v^2]\rho_2,
\]
and then iterating $n$ times, we obtain
\begin{align*}
&\left\langle
\widetilde{H}(u[\rho_1\cup\rho]),
\widetilde{H}(v[\rho_2\cup\rho])
\right\rangle_{Y;q=t}\\
&=
\left[\prod_{i=1}^{2n}
\widetilde{H}([u^iv^{i-1}]\rho_1)\,
\widetilde{H}([v^iu^{i-1}]\rho_2)\right]
\left[\prod_{j=1}^{n}
H([u^{2j}v^{2j}]\rho_1;\rho_2)\right]
S_n,
\end{align*}
where
\[
S_n
=
\sum_{\nu,\xi,\tau\in\YY}
u^{|\xi|}v^{|\tau|}
s_{\nu/\xi}([u^{2n}v^{2n}]\rho_1)\,
s_{\nu/\tau}([u^{2n}v^{2n}]\rho_2).
\]
Since $u,v\in(0,1)$,
\[
\lim_{n\to\infty}S_n
=
\sum_{\nu\in\YY}u^{|\nu|}v^{|\nu|}
=
\prod_{n\ge1}\frac{1}{1-u^nv^n}.
\]
Collecting the factors gives the desired identity.
\end{proof}

\begin{proof}[Proof of Lemma~\ref{l31}]
By Lemma~\ref{l41},
\[
\gamma_{l_i}([\lambda^{(i)}]';t,q)
=
\gamma_{l_i}\!\left(\lambda^{(i)};\frac{1}{q},\frac{1}{t}\right).
\]
Using Definition~\ref{df22} and the scalar-product formulas in Remark~3.4, the
expectation in Lemma~\ref{l31} can therefore be written as a normalized nested scalar
product.

For $i\in[l+1..r]$, define local kernels by
\[
\mathbf E_i=
\begin{cases}
\displaystyle
\sum_{\lambda^{(i)}\in\YY}
\gamma_{l_i}(\lambda^{(i)};q,t)\,
P_{\lambda^{(i)}}(A^{(i)},Y^{(i)};q,t)\,
Q_{\lambda^{(i)}}(Y^{(i-1)},B^{(i)};q,t),
&
\substack{(i-1,i)\in\mathcal L\times\mathcal L},
\\[4mm]
\displaystyle
\sum_{\lambda^{(i)}\in\YY}
\gamma_{l_i}\!\left([\lambda^{(i)}]';\frac1t,\frac1q\right)\,
Q_{\lambda^{(i)}}(Y^{(i-1)},B^{(i)};q,t)\,
Q_{[\lambda^{(i)}]'}(A^{(i)},Y^{(i)};t,q),
&
\substack{(i-1,i)\in\mathcal L\times\mathcal R},
\\[4mm]
\displaystyle
\sum_{\lambda^{(i)}\in\YY}
\gamma_{l_i}(\lambda^{(i)};q,t)\,
P_{[\lambda^{(i)}]'}(Y^{(i-1)},B^{(i)};t,q)\,
P_{\lambda^{(i)}}(A^{(i)},Y^{(i)};q,t),
&
\substack{(i-1,i)\in\mathcal R\times\mathcal L},
\\[4mm]
\displaystyle
\sum_{\lambda^{(i)}\in\YY}
\gamma_{l_i}\!\left([\lambda^{(i)}]';\frac1t,\frac1q\right)\,
P_{[\lambda^{(i)}]'}(Y^{(i-1)},B^{(i)};t,q)\,
Q_{[\lambda^{(i)}]'}(A^{(i)},Y^{(i)};t,q),
&
\substack{(i-1,i)\in\mathcal R\times\mathcal R}.
\end{cases}
\]
Define boundary terms by
\[
\mathbf E_l=
\begin{cases}
\displaystyle
\sum_{\lambda^{(l)}\in\YY}u^{|\lambda^{(l)}|}
P_{\lambda^{(l)}}(A^{(l)},Y^{(l)};q,t),
& l\in\mathcal L,\\[3mm]
\displaystyle
\sum_{\lambda^{(l)}\in\YY}u^{|\lambda^{(l)}|}
Q_{[\lambda^{(l)}]'}(A^{(l)},Y^{(l)};t,q),
& l\in\mathcal R,
\end{cases}
\]
and
\[
\mathbf E_{r+1}=
\begin{cases}
\displaystyle
\sum_{\lambda^{(r+1)}\in\YY}v^{|\lambda^{(r+1)}|}
Q_{\lambda^{(r+1)}}(Y^{(r)},B^{(r+1)};q,t),
& r\in\mathcal L,\\[3mm]
\displaystyle
\sum_{\lambda^{(r+1)}\in\YY}v^{|\lambda^{(r+1)}|}
P_{[\lambda^{(r+1)}]'}(Y^{(r)},B^{(r+1)};t,q),
& r\in\mathcal R.
\end{cases}
\]
Then
\[
\left.
\mathbb E_{\MP_{\bA,\bB,\mathcal P;q,t}}
\left[\prod_{i=l+1}^{r}\gamma_{l_i}(\lambda^{(i)};q,t)\right]
\right|_{q=t}
=
\left.
\frac{1}{Z}
\bigl\langle
\mathbf E_l,\langle \mathbf E_{l+1},\ldots,\langle \mathbf E_r,\mathbf E_{r+1}\rangle_{Y^{(r)}}\cdots\rangle_{Y^{(l+1)}}
\bigr\rangle_{Y^{(l)}}
\right|_{q=t},
\]
where $Z$ is the normalization constant of \eqref{pmd}.

\begin{revision}
\label{rev:comment10}
Recall that $D_{-k,X;q,t}$ is the Macdonald difference operator
in \eqref{ngt}, characterized by
\[
D_{-k,X;q,t}P_\lambda(X;q,t)
=\gamma_k(\lambda;q,t)P_\lambda(X;q,t),\qquad k\ge1.
\]
We use the convention $D_{0,X;q,t}=\mathrm{Id}$.
In the following display it acts on the combined alphabet
$X=(A^{(i)},Y^{(i)})$, with the other alphabets held fixed.
This operator is distinct from the differential form $D(W;q,t)$
in \eqref{ddf}.
\end{revision}
For each $i\in[l+1..r]$, the kernel $\mathbf E_i$ can be rewritten as
\[
\mathbf E_i=
\begin{cases}
D_{-l_i,(A^{(i)},Y^{(i)});q,t}\,
\Pi_{a_{i-1},a_i}\bigl((A^{(i)},Y^{(i)}),(Y^{(i-1)},B^{(i)})\bigr),
& a_i=L,\\[2mm]
D_{-l_i,(A^{(i)},Y^{(i)});1/t,1/q}\,
\Pi_{a_{i-1},a_i}\bigl((A^{(i)},Y^{(i)}),(Y^{(i-1)},B^{(i)})\bigr),
& a_i=R.
\end{cases}
\]
Applying Proposition~\ref{pa2} at $q=t$, and then interchanging contour integrals with the
nested scalar products, we arrive at a multiple contour integral of the form
\[
\frac{1}{Z}\oint\cdots\oint
\left[\prod_{i=l+1}^{r}D(W^{(i)};\omega(t,t,a_i))\right]
\Bigl\langle F_l,\langle F_{l+1},\ldots,\langle F_r,F_{r+1}\rangle_{Y^{(r)}}\cdots\rangle_{Y^{(l+1)}}\Bigr\rangle_{Y^{(l)}},
\]
where
\[
F_l=\widetilde{H}(\rho_{uA^{(l)},1,a_l}\cup\rho_{uY^{(l)},1,a_l}),
\qquad
F_{r+1}=\widetilde{H}(\rho_{vB^{(r+1)},1,a_r}\cup\rho_{vY^{(r)},1,a_r}),
\]
and, for $i\in[l+1..r]$,
\begin{align*}
F_i
&=
H(\rho_{Y^{(i-1)},1,a_{i-1}};\rho_{A^{(i)},1,a_i})\,
H(\rho_{Y^{(i-1)},1,a_{i-1}};\rho_{Y^{(i)},1,a_i})\,
H(\rho_{Y^{(i)},1,a_i};\rho_{B^{(i)},1,a_{i-1}})\\
&\qquad\times
H(\rho_{W^{(i)},t,2,a_i};\rho_{Y^{(i)},1,a_i})\,
H((-1)*\rho_{[W^{(i)}]^{-1},t^{-1},2,a_i};\rho_{Y^{(i-1)},1,a_{i-1}}).
\end{align*}
Here we used Lemma~\ref{l210} to rewrite every occurrence of $\Pi$, $\Pi'$, and $L$ in
terms of the $H$-kernel.

Now contract the scalar products recursively from right to left. At each step,
Lemma~\ref{211} evaluates the quadratic-exponential part of the scalar product, and
Lemma~\ref{l214} evaluates the $\widetilde H$-terms. This produces:
\begin{enumerate}
\item all cross-interaction factors
\[
\prod_{l+1\le i<j\le r}T_{a_i,a_j}(W^{(i)},W^{(j)}),
\]
\item all one-body factors
\[
\prod_{l+1\le i\le j\le r}
H(\rho_{W^{(i)},t,2,a_i};\rho_{A^{(j)},1,a_j}),
\qquad
\prod_{l+1\le i<j\le r}
H((-1)*\rho_{[W^{(j)}]^{-1},t^{-1},2,a_j};\rho_{B^{(i)},1,a_{i-1}}),
\]
\item and the remaining boundary contraction
\[
\left\langle
\widetilde H(u[\rho_1\cup\rho]),
\widetilde H(v[\rho_2\cup\rho])
\right\rangle_{Y;q=t},
\]
to which Lemma~\ref{l218} applies, producing exactly the infinite product in the
statement.
\end{enumerate}

Finally, the normalization $Z$ equals the partition function in Proposition~\ref{p23},
and the factors coming from $Z$ cancel against the corresponding denominator produced
by Lemma~\ref{l218}. The remaining factors are precisely those displayed in the statement
of Lemma~\ref{l31}.
\end{proof}

Combining Lemmas~\ref{l23}, \ref{l29}, \ref{le32}, and \ref{l31} gives the moment formula
for the Laplace transform of the height function under the double free-boundary Gibbs
measure.

\section{Asymptotics}\label{sect:as}

In this section, we analyze the scaling limit of the moment formula obtained in
Section~\ref{sect:mh}. We first rewrite the multi-point observable formula in a form
adapted to the periodic $q$-volume weights, then establish the one-point asymptotics,
and finally prove Gaussian fluctuations by the method of moments.

\begin{lemma}\label{l331}
Let $i_1\le i_2\le \cdots \le i_m\in[l+1..r]$, and let $l_1,\ldots,l_m\in\ZZ_{>0}$.
For each occurrence $i_s$, let $W^{(i_s)}$ be a set of variables with
$|W^{(i_s)}|=l_s$. Define
\[
\rho_{W_I}:=\bigcup_{s=1}^{m}\rho_{W^{(i_s)},t,2,a_{i_s}},
\qquad
\rho_{W_I}^{\circ}:=\bigcup_{s=1}^{m}(-1)*\rho_{[W^{(i_s)}]^{-1},t^{-1},2,a_{i_s}}.
\]
Then
\begin{align*}
&\left.\mathbb{E}_{\Pr}\left[\prod_{s=1}^{m}\gamma_{l_s}(\lambda^{(M,i_s)};q,t)\right]\right|_{q=t}\\
&=
\oint\cdots\oint
\Biggl[\prod_{s=1}^{m}D(W^{(i_s)};\omega(t,t,a_{i_s}))\Biggr]
\Biggl[\prod_{s=1}^{m}\prod_{\substack{j\in[i_s..r]\\ b_j=-}}
H(\rho_{W^{(i_s)},t,2,a_{i_s}};\rho_{\{x_j\},1,a_j})\Biggr]\\
&\qquad\times
\Biggl[\prod_{s=1}^{m}\prod_{\substack{j\in[l..i_s-1]\\ b_j=+}}
H((-1)*\rho_{[W^{(i_s)}]^{-1},t^{-1},2,a_{i_s}};\rho_{\{x_j\},1,a_j})\Biggr]
\Biggl[\prod_{1\le u<v\le m}
T_{a_{i_u},a_{i_v}}(W^{(i_u)},W^{(i_v)})\Biggr]\\
&\qquad\times
\prod_{k\ge1}
\widetilde{H}([u^{k-1}v^k]\rho_{W_I})
\widetilde{H}([u^{k}v^{k-1}]\rho_{W_I}^{\circ})
H(u^{2k}\rho_A;v^{2k}\rho_{W_I})
H(u^{2k}\rho_A;v^{2k-2}\rho_{W_I}^{\circ})\\
&\qquad\times
\prod_{k\ge1}
H(u^{2k}\rho_B;v^{2k}\rho_{W_I}^{\circ})
H(u^{2k-2}\rho_B;v^{2k}\rho_{W_I})
H(u^{2k}\rho_{W_I};v^{2k}\rho_{W_I}^{\circ}).
\end{align*}
Here the expectation is with respect to the free-boundary dimer measure \eqref{ppd},
and the contours are as in Lemma~\ref{l31}.
\end{lemma}

\begin{proof}
First assume that $i_1,\ldots,i_m$ are pairwise distinct and that
$\{i_1,\ldots,i_m\}=[l+1..r]$. Then the formula is exactly Lemmas~\ref{l23}
and \ref{l31} specialized to the Schur case $q=t$.

If some columns in $[l+1..r]$ are not observed, we set the corresponding exponent
to zero and use $\gamma_0(\lambda;q,t)=1$.

If an index appears several times, we apply the Macdonald integral operator several
times at the same column, with nested contours for the different copies. This is the
same standard extension as in Corollary~3.11 of~\cite{bcgs13}. The resulting formula
is precisely the one displayed above.
\end{proof}

Set
\[
I_L:=\{i\in[l..r]:a_i=L\},
\qquad
I_R:=\{i\in[l..r]:a_i=R\}.
\]

\begin{lemma}\label{lLeftMoments}
Under the assumptions of Lemma~\ref{l331}, suppose in addition that
$\{i_1,\ldots,i_m\}\subseteq I_L$, i.e. $a_{i_s}=L$ for every $s\in[m]$.
Define
\begin{align}
G_{1,>}(W,x,t)&:=\prod_{w\in W}\frac{w-x}{w-tx},
&
G_{1,<}(W,x,t)&:=\prod_{w\in W}\frac{t-wx}{t(1-wx)},
\label{dg1}\\
G_{0,>}(W,x,t)&:=\prod_{w\in W}\frac{w+tx}{w+x},
&
G_{0,<}(W,x,t)&:=\prod_{w\in W}\frac{t(1+wx)}{t+wx},
\label{dg2}
\end{align}
and, for $i\in[l+1..r]$,
\begin{align*}
G_{1,>i}(W)&:=\prod_{\substack{j\in[i..r]\\ b_j=-,\ a_j=L}}G_{1,>}(W,x_j,t),
&
G_{1,<i}(W)&:=\prod_{\substack{j\in[l..i-1]\\ b_j=+,\ a_j=L}}G_{1,<}(W,x_j,t),\\
G_{0,>i}(W)&:=\prod_{\substack{j\in[i..r]\\ b_j=-,\ a_j=R}}G_{0,>}(W,x_j,t),
&
G_{0,<i}(W)&:=\prod_{\substack{j\in[l..i-1]\\ b_j=+,\ a_j=R}}G_{0,<}(W,x_j,t).
\end{align*}
Also define
\begin{equation}
T_{L,L}(Z,W):=
\prod_{z\in Z}\prod_{w\in W}
\frac{(z-w)^2}{(z-t^{-1}w)(z-tw)}.
\label{tll}
\end{equation}
\begin{revision}
\label{rev:comment11}
For the moment formula, abbreviate $W_s:=W^{(i_s)}$ for $s\in[m]$.
The products below have exactly the same contour interpretation as in
Lemma~\ref{l331}. Then
\begingroup
\allowdisplaybreaks[1]
\begin{align*}
&\mathbb{E}_{\Pr}\left[\prod_{s=1}^{m}\gamma_{l_s}(\lambda^{(M,i_s)};t,t)\right]\\
&=
\oint\cdots\oint
\prod_{s=1}^{m}D(W_s;t,t)\\
&\qquad\times
\Biggl[
\prod_{s=1}^{m}
G_{1,>i_s}(W_s)\,
G_{1,<i_s}(W_s)\,
G_{0,>i_s}(W_s)\,
G_{0,<i_s}(W_s)
\Biggr]\\
&\qquad\times
\Biggl[\prod_{1\le s<g\le m}T_{L,L}(W_s,W_g)\Biggr]\\
&\qquad\times
\prod_{k\ge1}
\widetilde{H}([u^{k-1}v^k]\rho_{W_I})\,
\widetilde{H}([u^{k}v^{k-1}]\rho_{W_I}^{\circ})\\
&\qquad\times
\prod_{k\ge1}\prod_{s=1}^{m}
\Biggl[
\prod_{\substack{j\in[l..r]\\ b_j=-,\ a_j=L}}
G_{1,>}(W_s,u^{2k}v^{2k}x_j,t)\,
G_{1,<}(W_s,u^{2k}v^{2k-2}x_j,t)
\Biggr]\\
&\qquad\times
\prod_{k\ge1}\prod_{s=1}^{m}
\Biggl[
\prod_{\substack{j\in[l..r]\\ b_j=-,\ a_j=R}}
G_{0,>}(W_s,u^{2k}v^{2k}x_j,t)\,
G_{0,<}(W_s,u^{2k}v^{2k-2}x_j,t)
\Biggr]\\
&\qquad\times
\prod_{k\ge1}\prod_{s=1}^{m}
\Biggl[
\prod_{\substack{j\in[l..r]\\ b_j=+,\ a_j=L}}
G_{1,<}(W_s,u^{2k}v^{2k}x_j,t)\,
G_{1,>}(W_s,u^{2k-2}v^{2k}x_j,t)
\Biggr]\\
&\qquad\times
\prod_{k\ge1}\prod_{s=1}^{m}
\Biggl[
\prod_{\substack{j\in[l..r]\\ b_j=+,\ a_j=R}}
G_{0,<}(W_s,u^{2k}v^{2k}x_j,t)\,
G_{0,>}(W_s,u^{2k-2}v^{2k}x_j,t)
\Biggr]\\
&\qquad\times
\prod_{k\ge1}\prod_{s,g=1}^{m}
T_{L,L}(W_s,u^{2k}v^{2k}W_g).
\end{align*}
\endgroup
\end{revision}
Moreover,
\begin{align}
\widetilde{H}([u^{k-1}v^k]\rho_{W_I})
&=
\prod_{s=1}^{m}\Biggl[
\prod_{w\in W^{(i_s)}}
\frac{1-u^{2k-2}v^{2k}tw^{-2}}
{(1-u^{k-1}v^ktw^{-1})(1+u^{k-1}v^kw^{-1})}
\Biggr]
\label{th1}\\
&\qquad\times
\prod_{1\le s<g\le m}
T_{L,L}(W^{(i_s)},u^{2k-2}v^{2k}t[W^{(i_g)}]^{-1})
\notag\\
&\qquad\times
\prod_{s=1}^{m}\Biggl[
\prod_{\substack{w_j,w_i\in W^{(i_s)}\\ j<i}}
\frac{(1-u^{2k-2}v^{2k}tw_j^{-1}w_i^{-1})^2}
{(1-u^{2k-2}v^{2k}t^2w_j^{-1}w_i^{-1})(1-u^{2k-2}v^{2k}w_j^{-1}w_i^{-1})}
\Biggr],
\notag
\end{align}
and
\begin{align}
\widetilde{H}([u^{k}v^{k-1}]\rho_{W_I}^{\circ})
&=
\prod_{s=1}^{m}\Biggl[
\prod_{w\in W^{(i_s)}}
\frac{1-u^{2k}v^{2k-2}t^{-1}w^{2}}
{(1-u^{k}v^{k-1}w)(1+u^{k}v^{k-1}t^{-1}w)}
\Biggr]
\label{th2}\\
&\qquad\times
\prod_{1\le s<g\le m}
T_{L,L}([W^{(i_s)}]^{-1},u^{2k}v^{2k-2}t^{-1}W^{(i_g)})
\notag\\
&\qquad\times
\prod_{s=1}^{m}\Biggl[
\prod_{\substack{w_j,w_i\in W^{(i_s)}\\ j<i}}
\frac{(1-u^{2k}v^{2k-2}t^{-1}w_jw_i)^2}
{(1-u^{2k}v^{2k-2}t^{-2}w_jw_i)(1-u^{2k}v^{2k-2}w_jw_i)}
\Biggr].
\notag
\end{align}
\end{lemma}

\begin{proof}
Write \(W_s:=W^{(i_s)}\).  Since \(a_{i_s}=L\) for every \(s\in[m]\),
Lemma~\ref{l210} gives
\begin{revision}
\label{rev:comment12}
\begin{equation}\tag{4.5a}\label{eq:one-body-HG}
\begin{aligned}
H(\rho_{W_s,t,2,L};\rho_{\{x_j\},1,L})
&=G_{1,>}(W_s,x_j,t),\\
H(\rho_{W_s,t,2,L};\rho_{\{x_j\},1,R})
&=G_{0,>}(W_s,x_j,t),\\
H((-1)*\rho_{W_s^{-1},t^{-1},2,L};\rho_{\{x_j\},1,L})
&=G_{1,<}(W_s,x_j,t),\\
H((-1)*\rho_{W_s^{-1},t^{-1},2,L};\rho_{\{x_j\},1,R})
&=G_{0,<}(W_s,x_j,t),
\end{aligned}
\end{equation}
\end{revision}
and
\[
T_{a_{i_s},a_{i_g}}(W_s,W_g)=T_{L,L}(W_s,W_g).
\]
Thus the finite one-body and cross-interaction factors in
Lemma~\ref{l331} have the stated form.

We next evaluate the two \(\widetilde H\)-factors.  Set
\[
A_s:=\rho_{W_s^{-1},1,L},
\qquad
B_s:=(-1)*A_s,
\qquad
C_s:=\rho_{W_s,1,L}.
\]
By Lemma~\ref{l210},
\[
\rho_{W_I}
=
\bigcup_{s=1}^m\bigl([t]A_s\cup B_s\bigr),
\]
whereas
\[
\rho_{W_I}^{\circ}
=
\bigcup_{s=1}^m
\left(
(-1)*([t^{-1}]C_s)\cup C_s
\right).
\]
We use the identity
\[
\widetilde H\left(\bigcup_{s=1}^m\eta_s\right)
=
\prod_{s=1}^m\widetilde H(\eta_s)
\prod_{1\le s<g\le m}H(\eta_s;\eta_g),
\]
which follows from Definition~\ref{dsp1}.

Fix \(k\ge1\) and put
\[
\alpha:=u^{k-1}v^k.
\]
For
\[
\eta_s:=[\alpha t]A_s\cup[\alpha]B_s,
\]
a direct evaluation from the definitions of \(H\) and \(\widetilde H\) gives
\begin{align*}
\widetilde H(\eta_s)
&=
\prod_{w\in W_s}
\frac{1-\alpha^2tw^{-2}}
     {(1-\alpha tw^{-1})(1+\alpha w^{-1})}\\
&\quad\times
\prod_{\substack{w_j,w_i\in W_s\\ j<i}}
\frac{(1-\alpha^2tw_j^{-1}w_i^{-1})^2}
     {(1-\alpha^2t^2w_j^{-1}w_i^{-1})
      (1-\alpha^2w_j^{-1}w_i^{-1})},
\end{align*}
while, for \(s<g\),
\[
H(\eta_s;\eta_g)
=
T_{L,L}\bigl(W_s,\alpha^2t\,W_g^{-1}\bigr).
\]
Since
\[
[\alpha]\rho_{W_I}=\bigcup_{s=1}^m\eta_s,
\]
these identities give \eqref{th1}.

Similarly, put
\[
\delta:=u^kv^{k-1},
\qquad
\eta_s^{\circ}
:=
(-1)*([\delta t^{-1}]C_s)\cup[\delta]C_s.
\]
Then
\begin{align*}
\widetilde H(\eta_s^{\circ})
&=
\prod_{w\in W_s}
\frac{1-\delta^2t^{-1}w^2}
     {(1-\delta w)(1+\delta t^{-1}w)}\\
&\quad\times
\prod_{\substack{w_j,w_i\in W_s\\ j<i}}
\frac{(1-\delta^2t^{-1}w_jw_i)^2}
     {(1-\delta^2t^{-2}w_jw_i)
      (1-\delta^2w_jw_i)},
\end{align*}
and, for \(s<g\),
\[
H(\eta_s^{\circ};\eta_g^{\circ})
=
T_{L,L}\bigl(W_s^{-1},\delta^2t^{-1}W_g\bigr).
\]
Since
\[
[\delta]\rho_{W_I}^{\circ}
=
\bigcup_{s=1}^m\eta_s^{\circ},
\]
this proves \eqref{th2}.

Finally, \(H\) is multiplicative in each argument with respect to unions.
Applying the four identities in \rev{\eqref{eq:one-body-HG}}, with the corresponding
\(u,v\)-scalings, to the four factors in Lemma~\ref{l331} involving
\(\rho_A\) or \(\rho_B\) gives the four displayed scaled one-body products.
The last interaction factor satisfies
\[
H([u^{2k}]\rho_{W_I};[v^{2k}]\rho_{W_I}^{\circ})
=
\prod_{s,g=1}^m
T_{L,L}\bigl(W_s,u^{2k}v^{2k}W_g\bigr).
\]
This accounts for every factor in Lemma~\ref{l331} and proves the result.
\end{proof}

Let $\epsilon>0$ be small. Under Assumption~\ref{ap5}, for each
$p\in[m]$, $j\in[n]$, and $i\in[l^{(\epsilon)}..r^{(\epsilon)}]$, define
\begin{align*}
I_{j,p,>i,1}^{(\epsilon)}
&:=
\left\{
u\in\currentrev{(v_{p-1}^{(\epsilon)},v_p^{(\epsilon)})}\cap(n\ZZ+j)\cap[i..r^{(\epsilon)}]:
b_u^{(\epsilon)}=-,\ a_u^{(\epsilon)}=a_i^{(\epsilon)}
\right\},\\
I_{j,p,<i,1}^{(\epsilon)}
&:=
\left\{
u\in\currentrev{(v_{p-1}^{(\epsilon)},v_p^{(\epsilon)})}\cap(n\ZZ+j)\cap[l^{(\epsilon)}..i-1]:
b_u^{(\epsilon)}=+,\ a_u^{(\epsilon)}=a_i^{(\epsilon)}
\right\},\\
I_{j,p,>i,0}^{(\epsilon)}
&:=
\left\{
u\in\currentrev{(v_{p-1}^{(\epsilon)},v_p^{(\epsilon)})}\cap(n\ZZ+j)\cap[i..r^{(\epsilon)}]:
b_u^{(\epsilon)}=-,\ a_u^{(\epsilon)}\neq a_i^{(\epsilon)}
\right\},\\
I_{j,p,<i,0}^{(\epsilon)}
&:=
\left\{
u\in\currentrev{(v_{p-1}^{(\epsilon)},v_p^{(\epsilon)})}\cap(n\ZZ+j)\cap[l^{(\epsilon)}..i-1]:
b_u^{(\epsilon)}=+,\ a_u^{(\epsilon)}\neq a_i^{(\epsilon)}
\right\}.
\end{align*}

\currentrev{\hypersetup{linkcolor=revisioncolor,citecolor=revisioncolor}
\begin{lemma}\label{aps5}
Suppose Assumption~\ref{ap5} holds. Fix \(j\in[n]\), \(p\in[m]\), and
observation indices \(i^{(\epsilon)}\) with fixed residue \(i^*\in[n]\) and
\begin{equation}
\lim_{\epsilon\to0}\epsilon\,i^{(\epsilon)}=\chi.
\label{dci}
\end{equation}
Define the indicators from the interior periodic data by
\[
\begin{aligned}
\mathbf1_{\mathcal E_{j,p,>,1}}
&:=\mathbf1_{\{b_{p,j}=-,\ a_j=a_{i^*}\}},&
\mathbf1_{\mathcal E_{j,p,<,1}}
&:=\mathbf1_{\{b_{p,j}=+,\ a_j=a_{i^*}\}},\\
\mathbf1_{\mathcal E_{j,p,>,0}}
&:=\mathbf1_{\{b_{p,j}=-,\ a_j\ne a_{i^*}\}},&
\mathbf1_{\mathcal E_{j,p,<,0}}
&:=\mathbf1_{\{b_{p,j}=+,\ a_j\ne a_{i^*}\}}.
\end{aligned}
\]
For all sufficiently small \(\epsilon>0\) and each \(c\in\{0,1\}\),
\[
\begin{aligned}
I_{j,p,>i^{(\epsilon)},c}^{(\epsilon)}\ne\varnothing
&\quad\Longleftrightarrow\quad
\mathbf1_{\mathcal E_{j,p,>,c}}=1,
&&\text{if }\chi<V_p,\\
I_{j,p,<i^{(\epsilon)},c}^{(\epsilon)}\ne\varnothing
&\quad\Longleftrightarrow\quad
\mathbf1_{\mathcal E_{j,p,<,c}}=1,
&&\text{if }\chi>V_{p-1}.
\end{aligned}
\]
\end{lemma}

\begin{proof}
If \(\chi<V_p\), the part of the interior of block \(p\) at or to the
right of \(i^{(\epsilon)}\) has positive limiting length after rescaling
by \(\epsilon\). It therefore contains indices of every residue class
for all sufficiently small \(\epsilon\). At an interior index of residue
\(j\), the labels are \(a_j\) and \(b_{p,j}\), which proves the first
equivalence. The left-hand case is identical.
\end{proof}

The reference type \(a_{i^*}\) is kept fixed in all subsequent
\(\mathcal E\), \(\mathcal G\), and \(\mathcal F\) formulas, including
when \(\chi\) is replaced by \(V_0\) or \(V_m\) in \eqref{dfuvk}.
}

\begin{lemma}\label{l55}
Suppose Assumption~\ref{ap5} and \eqref{dci} hold.
\currentrev{\hypersetup{linkcolor=revisioncolor}Keep the fixed residue \(i^*\) from Lemma~\ref{aps5}.} Define
\begin{align*}
G_{1,>i}^{(\epsilon)}(W)
&:=
\prod_{\substack{j\in[l^{(\epsilon)}..r^{(\epsilon)}],\,j\ge i\\
b_j^{(\epsilon)}=-,\ a_j^{(\epsilon)}=a_i^{(\epsilon)}}}
G_{1,>}(W,x_j^{(\epsilon)},t),\\
G_{1,<i}^{(\epsilon)}(W)
&:=
\prod_{\substack{j\in[l^{(\epsilon)}..r^{(\epsilon)}],\,j<i\\
b_j^{(\epsilon)}=+,\ a_j^{(\epsilon)}=a_i^{(\epsilon)}}}
G_{1,<}(W,x_j^{(\epsilon)},t),\\
G_{0,>i}^{(\epsilon)}(W)
&:=
\prod_{\substack{j\in[l^{(\epsilon)}..r^{(\epsilon)}],\,j\ge i\\
b_j^{(\epsilon)}=-,\ a_j^{(\epsilon)}\neq a_i^{(\epsilon)}}}
G_{0,>}(W,x_j^{(\epsilon)},t),\\
G_{0,<i}^{(\epsilon)}(W)
&:=
\prod_{\substack{j\in[l^{(\epsilon)}..r^{(\epsilon)}],\,j<i\\
b_j^{(\epsilon)}=+,\ a_j^{(\epsilon)}\neq a_i^{(\epsilon)}}}
G_{0,<}(W,x_j^{(\epsilon)},t).
\end{align*}
Also define
\begin{align}
\mathcal{G}_{1,>\chi}(w)
&:=
\prod_{\substack{p\in[m]\\ V_p>\chi}}
\prod_{j=1}^{n}
\left(
\frac{1-(w\tau_j)^{-1}e^{V_p}}
{1-e^{\max\{V_{p-1},\chi\}}(w\tau_j)^{-1}}
\right)^{\mathbf{1}_{\mathcal{E}_{j,p,>,1}}},
\label{dgs1}\\
\mathcal{G}_{1,<\chi}(w)
&:=
\prod_{\substack{p\in[m]\\ V_{p-1}<\chi}}
\prod_{j=1}^{n}
\left(
\frac{1-we^{-V_{p-1}}\tau_j}
{1-e^{-\min\{V_p,\chi\}}w\tau_j}
\right)^{\mathbf{1}_{\mathcal{E}_{j,p,<,1}}},
\label{dgs2}\\
\mathcal{G}_{0,>\chi}(w)
&:=
\prod_{\substack{p\in[m]\\ V_p>\chi}}
\prod_{j=1}^{n}
\left(
\frac{1+e^{\max\{V_{p-1},\chi\}}(w\tau_j)^{-1}}
{1+(w\tau_j)^{-1}e^{V_p}}
\right)^{\mathbf{1}_{\mathcal{E}_{j,p,>,0}}},
\label{dgs3}\\
\mathcal{G}_{0,<\chi}(w)
&:=
\prod_{\substack{p\in[m]\\ V_{p-1}<\chi}}
\prod_{j=1}^{n}
\left(
\frac{1+e^{-\min\{V_p,\chi\}}w\tau_j}
{1+we^{-V_{p-1}}\tau_j}
\right)^{\mathbf{1}_{\mathcal{E}_{j,p,<,0}}}.
\label{dgs4}
\end{align}
Then
\begin{align*}
\lim_{\epsilon\to0}G_{1,>i^{(\epsilon)}}^{(\epsilon)}(W)
&=
\left[\prod_{w_s\in W}\mathcal{G}_{1,>\chi}(w_s)\right]^{\beta},
&
\lim_{\epsilon\to0} G_{1,<i^{(\epsilon)}}^{(\epsilon)}(W)
&=
\left[\prod_{w_s\in W}\mathcal{G}_{1,<\chi}(w_s)\right]^{\beta},\\
\lim_{\epsilon\to0} G_{0,>i^{(\epsilon)}}^{(\epsilon)}(W)
&=
\left[\prod_{w_s\in W}\mathcal{G}_{0,>\chi}(w_s)\right]^{\beta},
&
\lim_{\epsilon\to0}G_{0,<i^{(\epsilon)}}^{(\epsilon)}(W)
&=
\left[\prod_{w_s\in W}\mathcal{G}_{0,<\chi}(w_s)\right]^{\beta}.
\end{align*}
Here the logarithmic branches are chosen so that, when $z$ approaches the positive real axis,
the imaginary part of $\log z$ tends to $0$.
\end{lemma}

\begin{proof}
\currentrev{\hypersetup{linkcolor=revisioncolor,citecolor=revisioncolor}
First separate out the finitely many transition indices, including the
two endpoints. Their weights are bounded, and \(t=1+O(\epsilon)\).
On compact sets where all single-column denominators, including those
of transition columns, stay uniformly away from zero, each factor in
\eqref{dg1}--\eqref{dg2} is \(1+O(\epsilon)\). The product of the
transition-column factors therefore tends uniformly to \(1\).
The remaining products use only interior indices, so their active
periodic classes are exactly those in Lemma~\ref{aps5}.}
This is the same computation as Lemma~5.6 of~\cite{LV21}.
\end{proof}

\begin{definition}\label{def:R-sets}
For later use, we record the corresponding zero and pole sets. Define
\begin{align*}
\mathcal{R}_{\chi,1,1}
&:=
\left\{
e^{\max\{V_{p-1},\chi\}}\tau_j^{-1}:
p\in[m],\ V_p>\chi,\ \mathbf{1}_{\mathcal E_{j,p,>,1}}=1
\right\},\\
\mathcal{R}_{\chi,1,2}
&:=
\left\{
e^{V_p}\tau_j^{-1}:
p\in[m],\ V_p>\chi,\ \mathbf{1}_{\mathcal E_{j,p,>,1}}=1
\right\},\\
\mathcal{R}_{\chi,2,1}
&:=
\left\{
e^{\min\{V_p,\chi\}}\tau_j^{-1}:
p\in[m],\ V_{p-1}<\chi,\ \mathbf{1}_{\mathcal E_{j,p,<,1}}=1
\right\},\\
\mathcal{R}_{\chi,2,2}
&:=
\left\{
e^{V_{p-1}}\tau_j^{-1}:
p\in[m],\ V_{p-1}<\chi,\ \mathbf{1}_{\mathcal E_{j,p,<,1}}=1
\right\},\\
\mathcal{R}_{\chi,3,1}
&:=
\left\{
-e^{V_p}\tau_j^{-1}:
p\in[m],\ V_p>\chi,\ \mathbf{1}_{\mathcal E_{j,p,>,0}}=1
\right\},\\
\mathcal{R}_{\chi,3,2}
&:=
\left\{
-e^{\max\{V_{p-1},\chi\}}\tau_j^{-1}:
p\in[m],\ V_p>\chi,\ \mathbf{1}_{\mathcal E_{j,p,>,0}}=1
\right\},\\
\mathcal{R}_{\chi,4,1}
&:=
\left\{
-e^{V_{p-1}}\tau_j^{-1}:
p\in[m],\ V_{p-1}<\chi,\ \mathbf{1}_{\mathcal E_{j,p,<,0}}=1
\right\},\\
\mathcal{R}_{\chi,4,2}
&:=
\left\{
-e^{\min\{V_p,\chi\}}\tau_j^{-1}:
p\in[m],\ V_{p-1}<\chi,\ \mathbf{1}_{\mathcal E_{j,p,<,0}}=1
\right\}.
\end{align*}
Thus, for $\chi\notin\{V_p\}_{p=0}^{m}$, the zeros of $\mathcal{G}_{1,>\chi}$ are
$\mathcal{R}_{\chi,1,2}\setminus \mathcal{R}_{\chi,1,1}$ and its poles are
$\mathcal{R}_{\chi,1,1}\setminus \mathcal{R}_{\chi,1,2}$; similarly for
$\mathcal{G}_{1,<\chi}$, $\mathcal{G}_{0,>\chi}$, and $\mathcal{G}_{0,<\chi}$.

For $k_1,k_2\ge1$, define the scaled zero/pole sets
\begin{align*}
\mathcal{R}_{5,1,k_1,k_2}
&:=
\left\{
u^{2k_1}v^{2k_2}e^{V_{p-1}}\tau_j^{-1}:
p\in[m],\ \mathbf{1}_{\mathcal E_{j,p,>,1}}=1
\right\},\\
\mathcal{R}_{5,2,k_1,k_2}
&:=
\left\{
u^{2k_1}v^{2k_2}e^{V_p}\tau_j^{-1}:
p\in[m],\ \mathbf{1}_{\mathcal E_{j,p,>,1}}=1
\right\},\\
\mathcal{R}_{6,1,k_1,k_2}
&:=
\left\{
u^{-2k_1}v^{-2k_2}e^{V_p}\tau_j^{-1}:
p\in[m],\ \mathbf{1}_{\mathcal E_{j,p,<,1}}=1
\right\},\\
\mathcal{R}_{6,2,k_1,k_2}
&:=
\left\{
u^{-2k_1}v^{-2k_2}e^{V_{p-1}}\tau_j^{-1}:
p\in[m],\ \mathbf{1}_{\mathcal E_{j,p,<,1}}=1
\right\},\\
\mathcal{R}_{7,1,k_1,k_2}
&:=
\left\{
-u^{2k_1}v^{2k_2}e^{V_p}\tau_j^{-1}:
p\in[m],\ \mathbf{1}_{\mathcal E_{j,p,>,0}}=1
\right\},\\
\mathcal{R}_{7,2,k_1,k_2}
&:=
\left\{
-u^{2k_1}v^{2k_2}e^{V_{p-1}}\tau_j^{-1}:
p\in[m],\ \mathbf{1}_{\mathcal E_{j,p,>,0}}=1
\right\},\\
\mathcal{R}_{8,1,k_1,k_2}
&:=
\left\{
-u^{-2k_1}v^{-2k_2}e^{V_{p-1}}\tau_j^{-1}:
p\in[m],\ \mathbf{1}_{\mathcal E_{j,p,<,0}}=1
\right\},\\
\mathcal{R}_{8,2,k_1,k_2}
&:=
\left\{
-u^{-2k_1}v^{-2k_2}e^{V_{p}}\tau_j^{-1}:
p\in[m],\ \mathbf{1}_{\mathcal E_{j,p,<,0}}=1
\right\}.
\end{align*}
\end{definition}

For later use, set
\begin{equation}\label{dgc}
\mathcal G_{\chi}(w)
:=
\mathcal G_{1,>\chi}(w)
\mathcal G_{1,<\chi}(w)
\mathcal G_{0,>\chi}(w)
\mathcal G_{0,<\chi}(w),
\end{equation}
and, for \(k\ge1\),
\begin{align}
\mathcal F_{u,v,k}(w)
&:=
\mathcal G_{1,>V_0}(u^{-2k}v^{-2k}w)\,
\mathcal G_{1,<V_m}(u^{2k}v^{2k}w)\,
\mathcal G_{0,>V_0}(u^{-2k}v^{-2k}w)\,
\mathcal G_{0,<V_m}(u^{2k}v^{2k}w)
\nonumber\\
&\quad\times
\mathcal G_{1,>V_0}(u^{2-2k}v^{-2k}w)\,
\mathcal G_{1,<V_m}(u^{2k}v^{2k-2}w)\,
\mathcal G_{0,>V_0}(u^{2-2k}v^{-2k}w)\,
\mathcal G_{0,<V_m}(u^{2k}v^{2k-2}w).
\label{dfuvk}
\end{align}
Finally, write
\begin{equation}\label{dsc}
S_\chi(w)
:=
\mathcal G_\chi(w)\prod_{k\ge1}\mathcal F_{u,v,k}(w).
\end{equation}

\begin{definition}[Branch-admissible contour systems]
\label{def:branch-admissible}
Fix \(\chi\), and let \(\mathscr K\) denote the union of the contours under
consideration and of the regions swept out by all contour deformations used
below.  The contour system is called \emph{branch-admissible} if there exists
an open neighbourhood
\[
U\subset\mathbb C^*:=\mathbb C\setminus\{0\}
\]
of \(\mathscr K\) such that \(S_\chi\) is holomorphic and non-vanishing on
\(U\) and
\[
\frac{1}{2\pi i}\oint_\gamma
\frac{S_\chi'(w)}{S_\chi(w)}\,dw=0
\]
for every closed curve \(\gamma\subset U\).

Equivalently, \(S_\chi\) admits a holomorphic logarithm \(L_\chi\) on \(U\).
For \(\alpha>0\), we define
\[
[S_\chi(w)]^\alpha
:=
\exp\!\bigl(\alpha L_\chi(w)\bigr).
\]
If a connected component of \(U\) contains a point
\(w_0>0\) with \(S_\chi(w_0)>0\), we normalize the logarithm by
\[
L_\chi(w_0)=\ln S_\chi(w_0).
\]
where \(\ln\) denotes the ordinary real logarithm.  Consequently,
\[
L_\chi(x)=\ln S_\chi(x)\in\RR
\]
on the connected component of
\[
\{x\in U\cap(0,\infty):S_\chi(x)>0\}
\]
containing \(w_0\).
\end{definition}

\noindent\textbf{One-point contours.}
For
\[
\chi\in(V_0,V_m)\setminus\{V_p\}_{p=0}^m,
\]
define the required pole set
\[
\begin{aligned}
D_\chi:={}&
\bigl(R_{\chi,1,1}\setminus R_{\chi,1,2}\bigr)
\cup
\bigl(R_{\chi,3,1}\setminus R_{\chi,3,2}\bigr)
\\
&\cup
\bigcup_{k\ge1}
\Bigl[
\bigl(R_{5,1,k,k}\setminus R_{5,2,k,k}\bigr)
\cup
\bigl(R_{7,1,k,k}\setminus R_{7,2,k,k}\bigr)
\\
&\hspace{27mm}\cup
\bigl(R_{5,1,k-1,k}\setminus R_{5,2,k-1,k}\bigr)
\cup
\bigl(R_{7,1,k-1,k}\setminus R_{7,2,k-1,k}\bigr)
\Bigr].
\end{aligned}
\]
A \emph{one-point contour at \(\chi\)} is a positively oriented contour
\(\mathcal C\), possibly a finite union of pairwise disjoint simple closed
curves, that encloses \(0\) and every point of \(D_\chi\), and no other pole
of \(S_\chi\).

\begin{proposition}\label{p57}
Suppose Assumptions~\ref{ap5}, \ref{ass:admissible},
\ref{ap64}, and~\ref{ap65} hold at \(\chi\).  Let
\(i^{(\epsilon)}\) satisfy \eqref{dci}, and assume that
\[
a^{(\epsilon)}_{i^{(\epsilon)}}=L
\]
for all \(\epsilon>0\).  Let
\(\mathcal C_\chi\), \(U_\chi\), and \(L_\chi\) be supplied by
Corollary~\ref{cor:automatic-branch}.  Let \(\Pr^{(\epsilon)}\) be the
corresponding probability measure.  Then, for every
\(g\in\mathbb Z_{>0}\),
\[
\lim_{\epsilon\to0}
\mathbb E_{\Pr^{(\epsilon)}}\!
\left[
\gamma_g\!\left(
\lambda^{(M,i^{(\epsilon)})};t,t
\right)
\right]
=
\frac{1}{2\pi\mathrm i}
\oint_{\mathcal C_\chi}
\exp\!\bigl(g\beta L_\chi(w)\bigr)\frac{dw}{w}.
\]
\end{proposition}

\begin{proof}
Apply Lemma~\ref{lLeftMoments} with $m=1$ and $l_1=g$. Let
$W=(w_1,\ldots,w_g)$ and let $\mathcal C_1,\ldots,\mathcal C_g$ be nested contours
satisfying the contour conditions of Lemma~\ref{l31}. Choose them as sufficiently close nested copies of \(C_\chi\)
contained in \(U_\chi\), so that every annulus swept out during
the contour collapse is also contained in \(U_\chi\).

The contours may be chosen at positive distance from the compact accumulation
sets of the prelimit poles.  The explicit pole formulas then imply that, for
all sufficiently small \(\epsilon>0\), no pole crosses a contour and the
prescribed nesting and pole-enclosing conditions are unchanged.  Thus the same
contours remain admissible for the prelimit integrand.

Let $\Theta_k^{(\epsilon)}(W)$ denote the $k$-th factor in the infinite product of
Lemma~\ref{lLeftMoments} in the one-point case. By Lemma~\ref{l55},
\[
\mathcal{G}_{1,>i^{(\epsilon)}}^{(\epsilon)}(W)\,
\mathcal{G}_{1,<i^{(\epsilon)}}^{(\epsilon)}(W)\,
\mathcal{G}_{0,>i^{(\epsilon)}}^{(\epsilon)}(W)\,
\mathcal{G}_{0,<i^{(\epsilon)}}^{(\epsilon)}(W)
\longrightarrow
\prod_{s=1}^{g}[\mathcal{G}_{\chi}(w_s)]^{\beta}
\]
uniformly on $\mathcal C_1\times\cdots\times\mathcal C_g$.

For each fixed \(k\ge1\), Lemma~\ref{l55}, applied to the
\(u,v\)-scaled one-body products displayed in
Lemma~\ref{lLeftMoments}, shows that their product converges uniformly on
\(C_1\times\cdots\times C_g\) to
\[
\prod_{s=1}^g
\bigl[\mathcal F_{u,v,k}(w_s)\bigr]^\beta,
\]
where \(\mathcal F_{u,v,k}\) is defined in \eqref{dfuvk}.
The remaining factors in $\Theta_k^{(\epsilon)}(W)$ are
$\widetilde H([u^{k-1}v^k]\rho_W)$,
$\widetilde H([u^kv^{k-1}]\rho_W^{\circ})$, and
$T_{L,L}(W,u^{2k}v^{2k}W)$; by \eqref{th1}, \eqref{th2}, and the fact that
$u,v\in(0,1)$, these converge uniformly to $1$.

Let \(q_0:=uv\in(0,1)\).  Since the contours lie in a fixed compact
annulus, both \(w\) and \(w^{-1}\) are uniformly bounded on the product
contour.  The explicit formulas \eqref{th1}, \eqref{th2}, and
\eqref{tll}, together with the uniform separation of the contours from
their denominator loci, show that the logarithmic contribution of the
remaining factors is \(O(q_0^k)\), uniformly for sufficiently small
\(\epsilon\).

For the scaled one-body products, each local logarithmic contribution is
\(O((1-t)q_0^{2k})\).  Since the number of columns is
\(O(\epsilon^{-1})\) and \(1-t=O(\epsilon)\), their total contribution is
\(O(q_0^{2k})\).
\currentrev{The finitely many transition columns contribute only
\(O(\epsilon q_0^{2k})\); summing over \(k\ge1\) gives \(O(\epsilon)\).}
Consequently, there exist \(K_0\ge1\), \(C>0\), and
\(r\in(q_0,1)\) such that
\[
\sup_{W\in\mathcal C_1\times\cdots\times\mathcal C_g}
\left|\mathrm{Log}\Theta_k^{(\epsilon)}(W)\right|
\le C r^k,
\qquad k\ge K_0,
\]
uniformly for all sufficiently small \(\epsilon>0\). 
Here \(\mathrm{Log}\) denotes the holomorphic logarithm on
\(D(1,\frac12)\) normalized by \(\mathrm{Log} 1=0\); equivalently, it is the
restriction of the principal logarithm to this disk.
The Weierstrass
\(M\)-test therefore gives uniform convergence of the corresponding
infinite products.  Together with the fixed-\(k\) convergence proved
above, the same tail estimate yields
\[
\prod_{k\ge1}\Theta_k^{(\epsilon)}(W)
\longrightarrow
\prod_{s=1}^g\prod_{k\ge1}
[\mathcal F_{u,v,k}(w_s)]^\beta
\]
uniformly on the product contour. Therefore dominated convergence applies to the
multiple contour integral.

Since the contours are strictly nested, they are uniformly separated from the
diagonals \(w_a=w_b\).  Hence, by the definition \eqref{ddf} of
\(D(W;t,t)\),
\[
D(W;t,t)
\longrightarrow
\frac{1}{(2\pi\mathrm i)^g}
\frac{\displaystyle\sum_{s=1}^g w_s^{-1}}
     {(w_2-w_1)\cdots(w_g-w_{g-1})}
\,dw_1\cdots dw_g
\]
uniformly on
\(\mathcal C_1\times\cdots\times\mathcal C_g\).
Combining this with the preceding uniform limits and applying dominated
convergence, we obtain the limiting \(g\)-fold integral
\[
\frac{1}{(2\pi\mathrm i)^g}
\oint_{\mathcal C_1}\cdots\oint_{\mathcal C_g}
\frac{\displaystyle\sum_{s=1}^g w_s^{-1}}
     {(w_2-w_1)\cdots(w_g-w_{g-1})}
\prod_{s=1}^g f_\chi(w_s)\,
dw_1\cdots dw_g,
\]
where
\[
f_\chi(w)
:=
\exp\bigl(\beta L_\chi(w)\bigr)
=
[S_\chi(w)]^\beta
\]
on the annular domain \(U_\chi\) supplied by
Corollary~\ref{cor:automatic-branch}.

For \(g=1\), this is exactly the asserted formula.
Suppose \(g\ge2\).  Although \(f_\chi\) need not be globally meromorphic
when \(\beta\notin\mathbb Z\), branch-admissibility ensures that
\[
f_\chi(w)=\exp\bigl(\beta L_\chi(w)\bigr)
\]
is holomorphic on an open neighbourhood of the nested contours and of all
annuli swept out during their collapse.  The proof of Lemma~\ref{lb2} uses
only contour deformations through these annuli, Cauchy's theorem for the
one-body factor, and residues at the Cauchy poles
\(w_{j+1}=w_j\).  Hence the same argument applies on this branch domain,
without crossing a branch cut.

Applying this argument with
\[
k=g,\qquad d=1,\qquad f=f_\chi,\qquad
g_1(w)=\frac{1}{w},
\]
and observing that the only pole of \(g_1\) is \(0\), which is enclosed by
every contour, we obtain, after deforming the remaining contour to
\(\mathcal C\) within the same branch domain,
\[
\frac{1}{2\pi\mathrm i}
\oint_{\mathcal C}
f_\chi(w)^g\,\frac{dw}{w}
=
\frac{1}{2\pi\mathrm i}
\oint_{\mathcal C}
[S_\chi(w)]^{g\beta}\,\frac{dw}{w},
\]
where
\[
f_\chi(w)^g
=
\exp\bigl(g\beta L_\chi(w)\bigr).
\]
This proves the asserted formula.
\end{proof}

\begin{remark}
When \(u=v=0\), the formula of Proposition~\ref{p57}
reduces to the one-point formula for pure free-boundary dimer
states obtained in Proposition 5.7 of~\cite{LV21}.
When \(v=0\), it yields the corresponding one-point asymptotics
for free left boundary and empty right boundary.
\end{remark}

\begin{definition}[The $L$-chart annular prime-function covariance kernel]\label{def:KLL}
Set
\[
        \mathfrak q:=u^2v^2,
        \qquad
        S_\chi(z):=\mathcal G_\chi(z)\prod_{r\ge1}\mathcal F_{u,v,r}(z),
        \qquad
        \Phi_{\chi,g}(z):=[S_\chi(z)]^{g\beta}.
\]
The power \([S_\chi(z)]^{g\beta}\) is defined using the holomorphic
logarithm \(L_\chi\) associated with the chosen branch-admissible contour
system, as in Definition~\ref{def:branch-admissible}.
For $r\ge1$ put
\[
        c_r:=\mathfrak q^r,
        \qquad
        a_r:=u^{2r-2}v^{2r}=v^2\mathfrak q^{r-1},
        \qquad
        b_r:=u^{2r}v^{2r-2}=u^2\mathfrak q^{r-1}.
\]
Define
\begin{align}
\mathsf K_{LL}(z,w)
&:=\frac{1}{(z-w)^2}
\notag\\
&\quad+
\sum_{r\ge1}
\left[
\frac{c_r}{(z-c_rw)^2}
+
\frac{c_r}{(w-c_rz)^2}
+
\frac{a_r}{(zw-a_r)^2}
+
\frac{b_r}{(1-b_rzw)^2}
\right].
\label{eq:KLL-series}
\end{align}
A pair of contours $(\mathcal C_z,\mathcal C_w)$ is called $LL$-admissible if the two
contours satisfy the contour requirements of Proposition~\ref{p57} for the corresponding
one-point integrands and, in addition, none of
\[
        z=w,
        \qquad z=c_rw,
        \qquad w=c_rz,
        \qquad zw=a_r,
        \qquad b_rzw=1,
        \qquad r\ge1,
\]
occurs with $z\in\mathcal C_z$ and $w\in\mathcal C_w$.
\end{definition}

\begin{lemma}[Theta form of the $L$-chart kernel]\label{lem:KLL-theta}
Let
\[
        \Theta_{\mathfrak q}(x):=(x;\mathfrak q)_\infty(\mathfrak q/x;\mathfrak q)_\infty.
\]
Then, on every $LL$-admissible product contour,
\begin{equation}
        \mathsf K_{LL}(z,w)
        =
        \partial_z\partial_w
        \log\frac{\Theta_{\mathfrak q}(z/w)}
                  {\Theta_{\mathfrak q}(u^2zw)} .
\label{eq:KLL-theta}
\end{equation}
\end{lemma}

\begin{proof}
Using the product definition of $\Theta_{\mathfrak q}$ gives
\[
\partial_z\partial_w\log\Theta_{\mathfrak q}(z/w)
=
\frac1{(z-w)^2}
+
\sum_{r\ge1}\left[
\frac{\mathfrak q^r}{(z-\mathfrak q^rw)^2}
+
\frac{\mathfrak q^r}{(w-\mathfrak q^rz)^2}
\right],
\]
and
\[
-\partial_z\partial_w\log\Theta_{\mathfrak q}(u^2zw)
=
\sum_{r\ge1}\left[
\frac{v^2\mathfrak q^{r-1}}{(zw-v^2\mathfrak q^{r-1})^2}
+
\frac{u^2\mathfrak q^{r-1}}{(1-u^2\mathfrak q^{r-1}zw)^2}
\right].
\]
Adding the two identities proves the claim.
\end{proof}

\begin{theorem}[$L$-chart central limit theorem]\label{t58}
Suppose Assumptions~\ref{ap5}, \ref{ass:admissible},
\ref{ap64}, and~\ref{ap65} hold at each marked location. Let $s\in\ZZ_{>0}$, let
$g_1,\ldots,g_s\in\ZZ_{>0}$, and let
$i_1^{(\epsilon)},\ldots,i_s^{(\epsilon)}\in[l^{(\epsilon)}..r^{(\epsilon)}]$ satisfy
\[
\lim_{\epsilon\to0}\epsilon\,i_d^{(\epsilon)}=\chi_d,
\qquad
\chi_1\le \chi_2\le \cdots\le \chi_s.
\]
Assume that
\begin{equation}
        a_{i_1^{(\epsilon)}}^{(\epsilon)}=
        a_{i_2^{(\epsilon)}}^{(\epsilon)}=
        \cdots=
        a_{i_s^{(\epsilon)}}^{(\epsilon)}=L.
\label{ael}
\end{equation}
Let $\Pr^{(\epsilon)}$ be the corresponding probability measure, and define
\[
Q_{g_d}^{(\epsilon)}(\epsilon i_d^{(\epsilon)})
:=
\frac1\epsilon
\left(
\gamma_{g_d}(\lambda^{(M,i_d^{(\epsilon)})};t,t)
-
\mathbb E_{\Pr^{(\epsilon)}}
\gamma_{g_d}(\lambda^{(M,i_d^{(\epsilon)})};t,t)
\right).
\]
Then the vector of $Q$-observables converges in distribution to a centered Gaussian vector.
For $d,h\in[s]$, its covariance is
\begin{align}
&\mathrm{Cov}\bigl[Q_{g_d}(\chi_d),Q_{g_h}(\chi_h)\bigr]
\notag\\
&\quad=
\frac{n^2\beta^2g_dg_h}{(2\pi\mathbf i)^2}
\oint_{\mathcal C_d}\oint_{\mathcal C_h}
\Phi_{\chi_d,g_d}(z)\,
\Phi_{\chi_h,g_h}(w)\,
\mathsf K_{LL}(z,w)\,dz\,dw,
\label{eq:Q-cov}
\end{align}
The contours and logarithms are supplied by
Corollary~\ref{cor:automatic-branch}, and the contours are chosen
pairwise \(LL\)-admissibly as in that corollary.
\end{theorem}

\begin{lemma}[Mixed-factor decomposition]
\label{lem:mixed-factor-decomposition}
Fix \(d,h\in[s]\).  For \(a\in\{d,h\}\), set
\[
\Gamma_a^{(\epsilon)}
:=
\gamma_{g_a}\!\left(
\lambda^{(M,i_a^{(\epsilon)})};t,t
\right).
\]
Introduce two ordered sets of integration variables
\[
Z=(z_1,\ldots,z_{g_d}),
\qquad
W=(w_1,\ldots,w_{g_h}).
\]
Let
\[
\mathcal C_{d,1},\ldots,\mathcal C_{d,g_d},
\qquad
\mathcal C_{h,1},\ldots,\mathcal C_{h,g_h}
\]
be positively oriented contours satisfying the pole-enclosing and
lexicographic nesting conditions of Lemma~\ref{l31} for the two marked
blocks.  When the two marked columns coincide, distinct nested contour
copies are used.  The contours are chosen inside the corresponding
one-point contour neighbourhoods so that the two blocks are
\(LL\)-admissible.

Write
\[
dZ:=dz_1\cdots dz_{g_d},
\qquad
dW:=dw_1\cdots dw_{g_h}.
\]
For \(a\in\{d,h\}\), let
\(\mathcal I_a^{(\epsilon)}\) denote the scalar one-point integrand obtained
from Lemma~\ref{lLeftMoments} by taking
\[
m=1,\qquad i_1=i_a^{(\epsilon)},\qquad l_1=g_a,
\]
after removing the factor \((2\pi\mathbf i)^{-g_a}\) and the integration
differentials.  Thus
\[
\mathbb E_{\Pr^{(\epsilon)}}[\Gamma_a^{(\epsilon)}]
=
\frac{1}{(2\pi\mathbf i)^{g_a}}
\oint_{\mathcal C_{a,1}}\cdots
\oint_{\mathcal C_{a,g_a}}
\mathcal I_a^{(\epsilon)}(Z_a)\,dZ_a,
\]
where \(Z_d=Z\) and \(Z_h=W\).

Then
\[
\begin{aligned}
\mathbb E_{\Pr^{(\epsilon)}}[
\Gamma_d^{(\epsilon)}\Gamma_h^{(\epsilon)}]
&=
\frac{1}{(2\pi\mathbf i)^{g_d+g_h}}
\oint_{\mathcal C_{d,1}}\cdots
\oint_{\mathcal C_{d,g_d}}
\oint_{\mathcal C_{h,1}}\cdots
\oint_{\mathcal C_{h,g_h}}  \\
&\qquad
\mathcal I_d^{(\epsilon)}(Z)
\mathcal I_h^{(\epsilon)}(W)
\Lambda^{(\epsilon)}(Z,W)\,dZ\,dW .
\end{aligned}
\] The mixed interaction factor is
\[
        \Lambda^{(\epsilon)}(Z,W)=T_{L,L}(Z,W)\Xi^{(\epsilon)}(Z,W),
\]
with
\begin{align}
\Xi^{(\epsilon)}(Z,W)
&:=
\prod_{r\ge1}
T_{L,L}(Z,\mathfrak q^r W)
T_{L,L}(W,\mathfrak q^r Z)
\notag\\
&\quad\times
\prod_{r\ge1}
T_{L,L}(Z,u^{2r-2}v^{2r}t\,W^{-1})
T_{L,L}(Z^{-1},u^{2r}v^{2r-2}t^{-1}W).
\label{eq:Xi-eps-correct}
\end{align}
Moreover, the product of the two one-point moments is the same double-block integral with
$\Lambda^{(\epsilon)}(Z,W)$ replaced by $1$.
\end{lemma}

\begin{proof}
Insert the two marked columns in Lemma~\ref{lLeftMoments}.  The factors whose variables
all belong to the same marked block are precisely the self-interaction factors already present
in the corresponding one-point formula, and are absorbed into
$\mathcal I_d^{(\epsilon)}$ and $\mathcal I_h^{(\epsilon)}$.  The factor
$T_{L,L}(Z,W)$ comes from the finite product over ordered marked columns.  The factors
$T_{L,L}(Z,\mathfrak q^rW)$ and $T_{L,L}(W,\mathfrak q^rZ)$ come from the infinite product
$\prod_{r\ge1}T_{L,L}(W^{(i_s)},u^{2r}v^{2r}W^{(i_g)})$ after separating the $Z$--$W$
cross terms from the $Z$--$Z$ and $W$--$W$ self terms.  Finally, the two displayed factors
with $W^{-1}$ and $Z^{-1}$ are the cross terms in the two $\widetilde H$ factors, using
\eqref{th1} and \eqref{th2}.  This accounts for every factor in Lemma~\ref{lLeftMoments}.
\end{proof}

\begin{lemma}[Uniform expansion of $L$-chart pair interactions]\label{lem:uniform-pair-expansion}
For every $LL$-admissible product contour, uniformly in $Z$ and $W$ on that contour,
\begin{equation}
\frac{\Lambda^{(\epsilon)}(Z,W)-1}{\epsilon^2}
\longrightarrow
n^2\beta^2\sum_{z\in Z}\sum_{w\in W}zw\,\mathsf K_{LL}(z,w).
\label{eq:uniform-Lambda-expansion}
\end{equation}
In fact,
\[
\Lambda^{(\epsilon)}(Z,W)
=
1+\epsilon^2n^2\beta^2
\sum_{z\in Z}\sum_{w\in W}
zw\,\mathsf K_{LL}(z,w)
+O(\epsilon^3),
\]
uniformly for \(Z\) and \(W\) on the chosen \(LL\)-admissible product
contour.
\end{lemma}

\begin{proof}
Put \(c:=n\beta\), so that \(t=e^{-c\epsilon}\).  Let
\(\mathscr C_Z\) and \(\mathscr C_W\) denote the unions of the finitely many
contours carrying the variables in \(Z\) and \(W\), respectively, and set
\[
m_Z:=\min_{z\in\mathscr C_Z}|z|,
\qquad
M_Z:=\max_{z\in\mathscr C_Z}|z|,
\qquad
m_W:=\min_{w\in\mathscr C_W}|w|,
\qquad
M_W:=\max_{w\in\mathscr C_W}|w|.
\]
These constants are finite and \(m_Z,m_W>0\), since the contours are compact
subsets of \(\mathbb C^*=\mathbb{C}\setminus\{0\}\).

Recall that
\[
\mathfrak q=u^2v^2,\qquad
a_r=v^2\mathfrak q^{r-1},\qquad
b_r=u^2\mathfrak q^{r-1}.
\]
Choose \(R_0\) so large that, for every \(r\ge R_0\),
\[
\mathfrak q^rM_W\le\frac{m_Z}{4},
\qquad
\mathfrak q^rM_Z\le\frac{m_W}{4},
\qquad
a_r\le\frac{m_Zm_W}{4},
\qquad
b_rM_ZM_W\le\frac14.
\]
The reverse triangle inequality then gives uniform positive lower bounds for
\[
|z-\mathfrak q^rw|,
\qquad
|w-\mathfrak q^rz|,
\qquad
|zw-a_r|,
\qquad
|1-b_rzw|,
\qquad r\ge R_0.
\]
For the finitely many indices \(1\le r<R_0\), and for the direct diagonal
\(z=w\), the corresponding lower bounds follow from \(LL\)-admissibility and
compactness.  Since \(t\to1\), all \(t\)-dependent denominators occurring in
\eqref{eq:Xi-eps-correct} are therefore uniformly bounded away from zero for
all sufficiently small \(\epsilon>0\).

For \(\alpha\in\{-1,0,1\}\), define
\[
R_{\alpha}^{(\epsilon)}(x,y)
:=
\frac{(x-t^\alpha y)^2}
     {(x-t^{\alpha-1}y)(x-t^{\alpha+1}y)}.
\]
Since \(t=e^{-c\epsilon}\), the exact identity
\[
R_{\alpha}^{(\epsilon)}(x,y)-1
=
\frac{x t^\alpha y\,(t+t^{-1}-2)}
     {(x-t^{\alpha-1}y)(x-t^{\alpha+1}y)}
\]
gives
\[
R_{\alpha}^{(\epsilon)}(x,y)
=
1+c^2\epsilon^2\frac{xy}{(x-y)^2}
+O\!\left(\epsilon^3|xy|\right),
\]
uniformly for \(\alpha\in\{-1,0,1\}\) and for \((x,y)\) in compact sets
away from the diagonal \(x=y\).

Applying this expansion to the scalar factors in
\(\Lambda^{(\epsilon)}(Z,W)\), we obtain, uniformly for
\(z\in\mathscr C_Z\) and \(w\in\mathscr C_W\),
\begin{align*}
T_{L,L}(z,w)
&=
1+c^2\epsilon^2\frac{zw}{(z-w)^2}
+O(\epsilon^3),\\
T_{L,L}(z,\mathfrak q^rw)
&=
1+c^2\epsilon^2
\frac{\mathfrak q^rzw}{(z-\mathfrak q^rw)^2}
+O(\epsilon^3\mathfrak q^r),\\
T_{L,L}(w,\mathfrak q^rz)
&=
1+c^2\epsilon^2
\frac{\mathfrak q^rzw}{(w-\mathfrak q^rz)^2}
+O(\epsilon^3\mathfrak q^r),\\
T_{L,L}(z,a_rt\,w^{-1})
&=
1+c^2\epsilon^2
\frac{a_rzw}{(zw-a_r)^2}
+O(\epsilon^3a_r),\\
T_{L,L}(z^{-1},b_rt^{-1}w)
&=
1+c^2\epsilon^2
\frac{b_rzw}{(1-b_rzw)^2}
+O(\epsilon^3b_r).
\end{align*}

For \(z\in Z\) and \(w\in W\), set
\[
\begin{aligned}
B^{(\epsilon)}(z,w)
&:=
T_{L,L}(z,w)
\prod_{r\ge1}
T_{L,L}(z,\mathfrak q^r w)\,
T_{L,L}(w,\mathfrak q^r z)\\
&\qquad\qquad\times
T_{L,L}(z,a_rt\,w^{-1})\,
T_{L,L}(z^{-1},b_rt^{-1}w).
\end{aligned}
\]
Then
\[
\Lambda^{(\epsilon)}(Z,W)
=
\prod_{z\in Z}\prod_{w\in W}
B^{(\epsilon)}(z,w).
\]

Since
\[
a_r,b_r=O(\mathfrak q^r),
\qquad
\sum_{r\ge1}\mathfrak q^r<\infty,
\]
the second-order coefficients and the third-order remainders in the five
expansions above are absolutely and uniformly summable over \(r\).  Hence,
uniformly for \(z\) and \(w\) on the chosen contours,
\[
\begin{aligned}
B^{(\epsilon)}(z,w)
&=
1+c^2\epsilon^2
\Biggl[
\frac{zw}{(z-w)^2}\\
&\qquad\quad+
\sum_{r\ge1}
\left(
\frac{\mathfrak q^rzw}{(z-\mathfrak q^rw)^2}
+
\frac{\mathfrak q^rzw}{(w-\mathfrak q^rz)^2}
+
\frac{a_rzw}{(zw-a_r)^2}
+
\frac{b_rzw}{(1-b_rzw)^2}
\right)
\Biggr]
+O(\epsilon^3)\\
&=
1+c^2\epsilon^2zw\,\mathsf K_{LL}(z,w)
+O(\epsilon^3).
\end{aligned}
\]
Here the products of two or more second-order terms contribute only
\(O(\epsilon^4)\).

Finally, since \(Z\) and \(W\) contain only finitely many variables,
\[
\begin{aligned}
\Lambda^{(\epsilon)}(Z,W)
&=
\prod_{z\in Z}\prod_{w\in W}
\left(
1+c^2\epsilon^2zw\,\mathsf K_{LL}(z,w)
+O(\epsilon^3)
\right)\\
&=
1+c^2\epsilon^2
\sum_{z\in Z}\sum_{w\in W}
zw\,\mathsf K_{LL}(z,w)
+O(\epsilon^3).
\end{aligned}
\]
Since \(c=n\beta\), this is the claimed expansion.

\end{proof}

\begin{lemma}[Wick formula for centered moments]
\label{lem:wick-moments}
Under the assumptions of Theorem~\ref{t58}, fix \(r\ge1\) and
\(d_1,\ldots,d_r\in[s]\), with repetitions allowed.  For \(a\in[r]\), write
\[
g_a:=g_{d_a},
\qquad
\chi_a:=\chi_{d_a},
\qquad
i_a^{(\epsilon)}:=i_{d_a}^{(\epsilon)},
\]
For brevity, set
\[
Q_a^{(\epsilon)}
:=
Q_{g_a}^{(\epsilon)}
\!\left(\epsilon i_a^{(\epsilon)}\right),
\qquad a\in[r],
\]
where \(Q_g^{(\epsilon)}\) is defined in
Theorem~\ref{t58}.
When an index is repeated, distinct admissible contour copies are used.

For \(a,b\in[r]\), define
\begin{equation}
\Sigma_{ab}
:=
\frac{n^2\beta^2g_ag_b}{(2\pi\mathbf i)^2}
\oint_{\mathcal C_a}\oint_{\mathcal C_b}
\Phi_{\chi_a,g_a}(z)\,
\Phi_{\chi_b,g_b}(w)\,
\mathsf K_{LL}(z,w)\,dz\,dw .
\label{eq:Sigma-ab}
\end{equation}
Then
\begin{equation}
\lim_{\epsilon\to0}
\mathbb E_{\Pr^{(\epsilon)}}
\left[
\prod_{a=1}^r Q_a^{(\epsilon)}
\right]
=
\sum_{\pi\in\mathcal P_2([r])}
\prod_{\{a,b\}\in\pi}\Sigma_{ab},
\label{eq:wick-moments}
\end{equation}
where \(\mathcal P_2([r])\) denotes the set of pairings of \([r]\), and is
empty when \(r\) is odd.  In particular, the limiting covariance is
\eqref{eq:Q-cov}.
\end{lemma}

\begin{proof}
For each \(a\in[r]\), let
\[
\Gamma_a^{(\epsilon)}
:=
\gamma_{g_a}\!\left(
\lambda^{(M,i_a^{(\epsilon)})};t,t
\right),
\qquad
\overline{\Gamma}_a^{(\epsilon)}
:=
\Gamma_a^{(\epsilon)}
-
\mathbb E_{\Pr^{(\epsilon)}}\Gamma_a^{(\epsilon)}.
\]
Thus
\[
Q_a^{(\epsilon)}
=
\epsilon^{-1}\overline{\Gamma}_a^{(\epsilon)}.
\]

Let
\[
Z_a=(z_{a,1},\ldots,z_{a,g_a}),
\qquad
dZ_a:=dz_{a,1}\cdots dz_{a,g_a}.
\]
Write the one-point formula for the \(a\)-th block as
\[
\mathbb E_{\Pr^{(\epsilon)}}\Gamma_a^{(\epsilon)}
=
\frac{1}{(2\pi\mathbf i)^{g_a}}
\oint_{\mathcal C_{a,1}}\cdots\oint_{\mathcal C_{a,g_a}}
\mathcal I_a^{(\epsilon)}(Z_a)\,dZ_a,
\]
and abbreviate
\[
d\mu_a^{(\epsilon)}(Z_a)
:=
\frac{1}{(2\pi\mathbf i)^{g_a}}
\mathcal I_a^{(\epsilon)}(Z_a)\,dZ_a,
\]
with the contours understood.
Applying Lemma~\ref{lLeftMoments} to the blocks indexed by \(S\), group the
factors according to the blocks on which they depend.  The factors involving
only \(Z_a\) form the one-point contour form
\(d\mu_a^{(\epsilon)}(Z_a)\), while the factors involving two distinct
blocks \(Z_a\) and \(Z_b\) form the mixed factor
\(\Lambda_{ab}^{(\epsilon)}(Z_a,Z_b)\) identified in
Lemma~\ref{lem:mixed-factor-decomposition}.  The explicit formulas
\eqref{th1} and \eqref{th2} show that no factors involving three or more
distinct blocks occur.  Hence
\[
\mathbb E_{\Pr^{(\epsilon)}}
\left[
\prod_{a\in S}\Gamma_a^{(\epsilon)}
\right]
=
\int
\prod_{a\in S}d\mu_a^{(\epsilon)}(Z_a)
\prod_{\substack{a<b\\a,b\in S}}
\Lambda_{ab}^{(\epsilon)}(Z_a,Z_b).
\]

Expanding the centered product and inserting the one-point integrals for
the indices outside \(S\), we obtain
\begin{align}
\mathbb E_{\Pr^{(\epsilon)}}
\left[
\prod_{a=1}^r\overline{\Gamma}_a^{(\epsilon)}
\right]
&=
\int
\prod_{a=1}^r d\mu_a^{(\epsilon)}(Z_a)
\sum_{S\subseteq[r]}(-1)^{r-|S|}
\prod_{\substack{a<b\\a,b\in S}}
\Lambda_{ab}^{(\epsilon)}(Z_a,Z_b).
\label{eq:centered-subset-expansion}
\end{align}

Put
\[
\Delta_{ab}^{(\epsilon)}
:=
\Lambda_{ab}^{(\epsilon)}-1.
\]
For a simple graph \(G=([r],E(G))\), an edge \(\{a,b\}\) records the
choice of the factor \(\Delta_{ab}^{(\epsilon)}\).  Thus, for fixed
\(S\subseteq[r]\),
\[
\prod_{\substack{a<b\\a,b\in S}}
\Lambda_{ab}^{(\epsilon)}
=
\sum_{\substack{G\text{ simple on }[r]\\V_+(G)\subseteq S}}
\prod_{\{a,b\}\in E(G)}
\Delta_{ab}^{(\epsilon)},
\]
where \(V_+(G)\) denotes the set of non-isolated vertices of \(G\).
Consequently, the coefficient of a fixed graph \(G\) in
\eqref{eq:centered-subset-expansion} is
\[
\sum_{S\supseteq V_+(G)}(-1)^{r-|S|}
=
\begin{cases}
1,&V_+(G)=[r],\\
0,&V_+(G)\ne[r].
\end{cases}
\]
Hence centering removes precisely the graphs with an isolated vertex.
Consequently, \eqref{eq:centered-subset-expansion} becomes the exact identity
\begin{equation}
\mathbb E_{\Pr^{(\epsilon)}}
\left[
\prod_{a=1}^r\overline{\Gamma}_a^{(\epsilon)}
\right]
=
\int
\prod_{a=1}^r d\mu_a^{(\epsilon)}(Z_a)
\sum_{\substack{G\text{ simple on }[r]\\
                 G\text{ has no isolated vertices}}}
\prod_{\{a,b\}\in E(G)}
\Delta_{ab}^{(\epsilon)}(Z_a,Z_b).
\label{eq:centered-graph-expansion}
\end{equation}

Lemma~\ref{lem:uniform-pair-expansion} gives, uniformly on the chosen
product contours,
\begin{equation}
\Delta_{ab}^{(\epsilon)}(Z_a,Z_b)
=
\epsilon^2 A_{ab}(Z_a,Z_b)
+
O(\epsilon^3),
\label{eq:Delta-pair-expansion}
\end{equation}
where
\[
A_{ab}(Z_a,Z_b)
:=
n^2\beta^2
\sum_{z\in Z_a}\sum_{w\in Z_b}
zw\,\mathsf K_{LL}(z,w).
\]
The proof of Proposition~\ref{p57} also gives uniform convergence, and hence
uniform boundedness, of the one-point contour densities
\(d\mu_a^{(\epsilon)}\) on the fixed contours.

If \(G\) has \(e(G)\) edges, its contribution to
\[
\mathbb E_{\Pr^{(\epsilon)}}
\left[
\prod_{a=1}^r Q_a^{(\epsilon)}
\right]
=
\epsilon^{-r}
\mathbb E_{\Pr^{(\epsilon)}}
\left[
\prod_{a=1}^r\overline{\Gamma}_a^{(\epsilon)}
\right]
\]
is therefore
\[
O\!\left(\epsilon^{2e(G)-r}\right).
\]
A graph on \(r\) vertices without isolated vertices has at least
\(\lceil r/2\rceil\) edges.  Hence every contribution vanishes when \(r\)
is odd.

Now let \(r=2m\).  The only terms that may survive have \(e(G)=m\).
A graph with \(2m\) vertices, \(m\) edges, and no isolated vertices is
necessarily a perfect matching.  Thus, for
\(\pi\in\mathcal P_2([r])\), uniform convergence in
\eqref{eq:Delta-pair-expansion} and dominated convergence give
\begin{align}
&\lim_{\epsilon\to0}
\epsilon^{-r}
\int
\prod_{a=1}^r d\mu_a^{(\epsilon)}(Z_a)
\prod_{\{a,b\}\in\pi}
\Delta_{ab}^{(\epsilon)}(Z_a,Z_b)
\notag\\
&\qquad=
\int
\prod_{a=1}^r d\mu_a(Z_a)
\prod_{\{a,b\}\in\pi}
A_{ab}(Z_a,Z_b)
\notag\\
&\qquad=
\prod_{\{a,b\}\in\pi}
\left[
\int
d\mu_a(Z_a)\,d\mu_b(Z_b)\,
A_{ab}(Z_a,Z_b)
\right].
\label{eq:matching-factorization}
\end{align}
The last equality holds because the pairs in \(\pi\) are disjoint.

It remains to evaluate the two-block integral associated with a pair
\(\{a,b\}\).  Write
\[
Z_a=(z_1,\ldots,z_{g_a}),
\qquad
Z_b=(w_1,\ldots,w_{g_b}),
\]
and set
\[
f_a(z):=[S_{\chi_a}(z)]^\beta,
\qquad
f_b(w):=[S_{\chi_b}(w)]^\beta.
\]
The limiting one-point contour form for the \(a\)-th block is
\[
d\mu_a(Z_a)
=
\frac{1}{(2\pi\mathbf i)^{g_a}}
\frac{\displaystyle\sum_{i=1}^{g_a}z_i^{-1}}
     {(z_2-z_1)\cdots(z_{g_a}-z_{g_a-1})}
\prod_{i=1}^{g_a}f_a(z_i)\,dZ_a,
\]
and analogously for \(d\mu_b(Z_b)\).

For fixed \(Z_b\), define
\[
H_{Z_b}(z)
:=
\sum_{w\in Z_b}w\,\mathsf K_{LL}(z,w).
\]
The \(Z_a\)-integral appearing in
\[
\int d\mu_a(Z_a)\,d\mu_b(Z_b)\,
A_{ab}(Z_a,Z_b)
\]
therefore contains
\[
\frac{1}{(2\pi\mathbf i)^{g_a}}
\oint\cdots\oint
\frac{
 \left(\sum_{i=1}^{g_a}z_i^{-1}\right)
 \left(\sum_{i=1}^{g_a}z_iH_{Z_b}(z_i)\right)}
 {(z_2-z_1)\cdots(z_{g_a}-z_{g_a-1})}
\prod_{i=1}^{g_a}f_a(z_i)\,dZ_a .
\]
Applying Lemma~\ref{lb2} with
\[
k=g_a,\qquad d=2,\qquad
g_1(z)=z^{-1},\qquad
g_2(z)=zH_{Z_b}(z),
\]
gives
\[
\frac{g_a}{2\pi\mathbf i}
\oint_{\mathcal C_a}
f_a(z)^{g_a}H_{Z_b}(z)\,dz.
\]

We next apply Lemma~\ref{lb2} to the \(Z_b\)-block with
\[
k=g_b,\qquad d=2,\qquad
g_1(w)=w^{-1},\qquad
g_2(w)=w\,\mathsf K_{LL}(z,w).
\]
This yields
\[
\int d\mu_a(Z_a)\,d\mu_b(Z_b)\,
A_{ab}(Z_a,Z_b)
=
\frac{n^2\beta^2g_ag_b}{(2\pi\mathbf i)^2}
\oint_{\mathcal C_a}\oint_{\mathcal C_b}
f_a(z)^{g_a}f_b(w)^{g_b}
\mathsf K_{LL}(z,w)\,dz\,dw.
\]
Since
\[
f_a(z)^{g_a}=\Phi_{\chi_a,g_a}(z),
\qquad
f_b(w)^{g_b}=\Phi_{\chi_b,g_b}(w),
\]
the last expression is \(\Sigma_{ab}\).
\end{proof}

\begin{proof}[Proof of Theorem~\ref{t58}]
Let \(\Sigma=(\Sigma_{dh})_{d,h\in[s]}\) be the matrix defined by
\eqref{eq:Q-cov}.  Lemma~\ref{lem:wick-moments} shows that all mixed
moments converge to the Wick moments of the centered Gaussian distribution
with covariance matrix \(\Sigma\).  Since every multivariate Gaussian
distribution, possibly degenerate, is moment-determinate, the multivariate
method of moments yields the claimed convergence in distribution.
\end{proof}

\begin{proof}[Proof of Theorem~\ref{t77}]
For \(g\in\mathbb Z_{>0}\) and a marked \(L\)-column
\(i^{(\epsilon)}\) with
\[
\epsilon i^{(\epsilon)}\longrightarrow \chi,
\]
set
\[
X_g^{(\epsilon)}(\chi)
:=
\int_{-\infty}^{\infty}
\left(
h^{\mathrm{col}}_{M^{(\epsilon)}}
\left(
i^{(\epsilon)},\frac{\eta}{\epsilon}
\right)
-
\mathbb E
h^{\mathrm{col}}_{M^{(\epsilon)}}
\left(
i^{(\epsilon)},\frac{\eta}{\epsilon}
\right)
\right)
e^{-n\beta g\eta}\,d\eta .
\]
By Lemma~\ref{l29} and by the zero-charge convention following
Lemma~\ref{le32},
\[
\int_{-\infty}^{\infty}
h_{M^{(\epsilon)}}(x,y)t^{gy}\,dy
=
\frac{2}{(g\log t)^2}
\gamma_g\!\left(\lambda^{(M^{(\epsilon)},i^{(\epsilon)})};t,t\right).
\]
After the change of variables
\[
\eta=\epsilon y
\]
and using the exact scaling
\[
t=e^{-n\beta\epsilon},
\]
the Laplace kernels agree:
\[
t^{gy}=e^{-n\beta g\eta}.
\]
Therefore, after centering,
\begin{equation}
X_g^{(\epsilon)}(\chi)
=
\frac{2}{n^2\beta^2g^2}
Q_g^{(\epsilon)}\!\left(\epsilon i^{(\epsilon)}\right).
\label{eq:height-Q-normalization}
\end{equation}

Apply this identity to the finitely many marked points
\[
(\chi_d,g_d),
\qquad d\in[s].
\]
Theorem~\ref{t58} gives convergence of the vector
\[
\left(
Q_{g_d}^{(\epsilon)}
\left(\epsilon i_d^{(\epsilon)}\right)
\right)_{d\in[s]}
\]
to a centered Gaussian vector with covariance
\[
\operatorname{Cov}
\left(
Q_{g_d}(\chi_d),
Q_{g_h}(\chi_h)
\right)
=
\frac{n^2\beta^2g_dg_h}{(2\pi i)^2}
\int_{C_d}\int_{C_h}
\Phi_{\chi_d,g_d}(z)
\Phi_{\chi_h,g_h}(w)
K_{LL}(z,w)\,dz\,dw.
\]
Multiplying this covariance by the two normalization factors in
\eqref{eq:height-Q-normalization} gives
\[
\begin{aligned}
\operatorname{Cov}
\left(
X_{g_d}(\chi_d),
X_{g_h}(\chi_h)
\right)
&=
\frac{2}{n^2\beta^2g_d^2}
\frac{2}{n^2\beta^2g_h^2}
\frac{n^2\beta^2g_dg_h}{(2\pi i)^2}
\\
&\qquad\times
\int_{C_d}\int_{C_h}
\Phi_{\chi_d,g_d}(z)
\Phi_{\chi_h,g_h}(w)
K_{LL}(z,w)\,dz\,dw
\\
&=
\frac{4}{n^2\beta^2g_dg_h(2\pi i)^2}
\int_{C_d}\int_{C_h}
\Phi_{\chi_d,g_d}(z)
\Phi_{\chi_h,g_h}(w)
K_{LL}(z,w)\,dz\,dw.
\end{aligned}
\]
This is exactly \eqref{eq:main-height-cov}. Hence the vector of
height Laplace observables converges to the centered Gaussian
vector stated in Theorem~\ref{t77}.

Finally, Lemma~\ref{lem:KLL-theta} gives
\[
K_{LL}(z,w)
=
\partial_z\partial_w
\log
\frac{\Theta_q(z/w)}
     {\Theta_q(u^2zw)}.
\]
Thus, in the spectral coordinate \(Z=uz\), the covariance is the
pullback of the annular prime-function derivative covariance on
the \(L\)-chart Laplace-test class. This proves
Theorem~\ref{t77}.
\end{proof}

\section{Limit Shape and Frozen Boundary}\label{sect:fb}

In this section, we first prove the integral formula for the Laplace transform of the
rescaled height function stated in Theorem~\ref{l61}.  We then upgrade the transform
convergence to a weak slope-measure limit shape by working in the natural moment variable
$x=e^{-n\beta\kappa}$.  This removes the former need for a separate Laplace-extension
hypothesis.  The subsequent frozen-boundary equation is stated in
Theorem~\ref{p412}.  

\subsection{Proof of the limit-shape formula}

\begin{proof}[Proof of Theorem \ref{l61}]
Fix \(k\in\mathbb Z_{>0}\) and set \(\alpha:=\beta k\).  Write
\[
\lambda_\epsilon
:=
\lambda^{(M^{(\epsilon)},i^{(\epsilon)})},
\qquad
\Gamma_\epsilon
:=
\gamma_k(\lambda_\epsilon;t,t).
\]
By the column-coordinate convention
\eqref{eq:column-height-convention} and
\(a_{i^{(\epsilon)}}^{(\epsilon)}=L\),
\[
h_{M^{(\epsilon)}}^{\mathrm{col}}
\bigl(i^{(\epsilon)},y\bigr)
=
h_{M^{(\epsilon)}}\left(
2i^{(\epsilon)}-\frac12,y
\right).
\]
Moreover, the canonical representative has zero charge by the convention
following Lemma~\ref{le32}.  Hence Lemma~\ref{l29}, followed by the change of
variables \(\kappa=\epsilon y\), gives
\begin{align*}
\mathcal{L}_{\beta k}^{(\epsilon)}(\chi)
&=
\epsilon^2
\int_{-\infty}^{\infty}
h_{M^{(\epsilon)}}^{\mathrm{col}}
\bigl(i^{(\epsilon)},y\bigr)
e^{-n\beta k\epsilon y}\,dy=
\epsilon^2
\int_{-\infty}^{\infty}
h_{M^{(\epsilon)}}^{\mathrm{col}}
\bigl(i^{(\epsilon)},y\bigr)t^{ky}\,dy\\
&=
\frac{2\epsilon^2}{(k\log t)^2}\,
\Gamma_\epsilon
=
\frac{2}{n^2\alpha^2}\,\Gamma_\epsilon,
\end{align*}
where we used the exact scaling
\(t=e^{-n\beta\epsilon}\) from
\eqref{eq:exact-t-scaling}.

Proposition~\ref{p57} therefore yields
\[
\lim_{\epsilon\to0}
\mathbb E_{\Pr^{(\epsilon)}}
\left[\mathcal{L}_{\beta k}^{(\epsilon)}(\chi)\right]
=
\frac{2}{n^2\alpha^2}
\frac{1}{2\pi\mathbf i}
\oint_{\mathcal C}
[S_\chi(w)]^{\beta k}\frac{dw}{w}
=
\frac{1}{n^2\alpha^2\pi\mathbf i}
\oint_{\mathcal C}
[S_\chi(w)]^\alpha\frac{dw}{w}.
\]

It remains to control the fluctuations.  Applying
Lemma~\ref{lem:wick-moments} with \(r=2\) and with the same marked block
repeated, using two distinct admissible contour copies, gives
\[
\mathbb E_{\Pr^{(\epsilon)}}\left[
\left(
\frac{\Gamma_\epsilon-
\mathbb E_{\Pr^{(\epsilon)}}\Gamma_\epsilon}{\epsilon}
\right)^2
\right]
\longrightarrow \Sigma
\]
for some finite constant \(\Sigma\).  Consequently,
\[
\operatorname{Var}_{\Pr^{(\epsilon)}}(\Gamma_\epsilon)
=O(\epsilon^2),
\]
and hence
\[
\label{rev:comment13}
\operatorname{Var}_{\Pr^{(\epsilon)}}\left(
\rev{\mathcal L_{\beta k}^{(\epsilon)}(\chi)}
\right)
=
\left(\frac{2}{n^2\alpha^2}\right)^2
O(\epsilon^2)
\longrightarrow0.
\]
Thus \(\rev{\mathcal L_{\beta k}^{(\epsilon)}(\chi)}\) converges in probability to the
limit in \eqref{lph}.
\end{proof}

\subsection{Limit shape in the natural scale}

The next lemmas upgrade Theorem~\ref{l61} to Theorem~\ref{thm:weak-limit-shape-beta}.  They
are the analogue, in the present column-coordinate setting, of the standard moment-determinacy
step which turns exponential moment asymptotics into a weak limit-shape theorem.

\begin{lemma}[Positive slope measures and slope bound]\label{lem:natural-positive-slope}
For every $\epsilon>0$ and every marked $L$-column used in Theorem~\ref{l61}, the function
$H_\epsilon(\kappa)$, defined in (\ref{dhe}), is nondecreasing in $\kappa$.  Its distributional derivative is a
positive locally finite measure.  Moreover, after linear interpolation on the microscopic
vertical mesh, the density is bounded by $2$; equivalently, for every interval $I\subset\mathbb R$,
\[
        dH_\epsilon(I)\le 2|I|+O(\epsilon),
\]
where the error is uniform on compact $I$.
\end{lemma}

\begin{proof}
In the $L$-chart, the height formula \eqref{hm1} counts occupied horizontal and diagonal
edges below the observation height and multiplies the count by $2$.  Passing from
$y$ to $\kappa=\epsilon y$ gives jumps of size $2\epsilon$ at mesh spacing $\epsilon$; hence
the distributional derivative is positive and has mesh-scale density at most $2$ after
linear interpolation.  The possible error comes only from the two endpoints of the interval.
\end{proof}

\begin{lemma}[Natural moments]\label{lem:natural-moments}
For every $k\in\mathbb Z_{>0}$,
\[
\int_0^\infty x^{k-1}\nu_\chi^{(\epsilon)}(dx)
=
n^2\beta^2k
\int_{-\infty}^{\infty}e^{-n\beta k\kappa}H_\epsilon(\kappa)\,d\kappa.
\]
Consequently the left-hand side converges in probability to the moment displayed in
\eqref{nu-moments}.
\end{lemma}

\begin{proof}
By the definition \eqref{dnu-eps} and the change of variables $x=e^{-n\beta\kappa}$,
\[
\int_0^\infty x^{k-1}\nu_\chi^{(\epsilon)}(dx)
=
n\beta\int_{-\infty}^{\infty}e^{-n\beta k\kappa}\,dH_\epsilon(\kappa).
\]
For each realization, the normalization in the definition of \(h_M\) implies
that \(H_\epsilon(\kappa)=0\) for all sufficiently negative \(\kappa\).
Moreover, \(H_\epsilon\) grows at most linearly as
\(\kappa\to+\infty\).  Hence
\[
e^{-n\beta k\kappa}H_\epsilon(\kappa)\longrightarrow0
\qquad\text{as }\kappa\to\pm\infty.
\]
Integration by parts therefore gives
\[
\int_{\mathbb R}e^{-n\beta k\kappa}\,dH_\epsilon(\kappa)
=
n\beta k
\int_{\mathbb R}e^{-n\beta k\kappa}H_\epsilon(\kappa)\,d\kappa.
\]
Theorem~\ref{l61}, with $\alpha=\beta k$, then gives \eqref{nu-moments}.
\end{proof}

\begin{lemma}[Compact moment determinacy and convergence]
\label{lem:natural-moment-determinacy}
Let \(\mathscr M\) be the space of positive finite Borel measures on
\([0,\infty)\), equipped with the topology of weak convergence against
bounded continuous functions.  Let \(\mu_\epsilon\) be
\(\mathscr M\)-valued random variables, and write
\[
\mathsf P_\epsilon:=\operatorname{Law}(\mu_\epsilon).
\]
Suppose that, for every \(k\ge0\),
\[
M_{\epsilon,k}
:=
\int_0^\infty x^k\,\mu_\epsilon(dx)
\xrightarrow{\mathbb P}
m_k,
\qquad
m_k\in[0,\infty),
\]
and that there exist constants \(C_0<\infty\) and \(C\ge0\) such that
\[
m_k\le C_0C^k,
\qquad k\ge0.
\]
Assume moreover that the laws \(\mathsf P_\epsilon\) are asymptotically
tight on \(\mathscr M\): for every \(\eta>0\), there exist
\(\epsilon_0>0\) and a compact set
\(\mathcal K_\eta\subset\mathscr M\) such that
\[
\sup_{0<\epsilon<\epsilon_0}
\mathsf P_\epsilon(\mathcal K_\eta^{\,c})
\le\eta.
\]

Then there exists a unique finite Borel measure \(\mu\) on
\([0,\infty)\) with moments \((m_k)_{k\ge0}\).  Moreover,
\[
\operatorname{supp}\mu\subseteq[0,C],
\]
Then \(\mu_\epsilon\) converges to \(\mu\) in probability in the weak
topology on \(\mathscr M\).
\end{lemma}

\begin{proof}
We first prove determinacy.  Let \(\rho\) be any finite measure with moments
\((m_k)_{k\ge0}\).  For \(R>C\) and \(k\ge1\),
\[
\rho([R,\infty))
\le
R^{-k}\int_0^\infty x^k\,\rho(dx)
=
R^{-k}m_k
\le
C_0\left(\frac{C}{R}\right)^k.
\]
Letting \(k\to\infty\) gives
\[
\operatorname{supp}\rho\subseteq[0,C].
\]
Thus any two finite measures with moments \((m_k)_{k\ge0}\) have the same
integrals against all polynomials on \([0,C]\).  By the Weierstrass
approximation theorem and the uniqueness part of the
Riesz--Markov--Kakutani representation theorem, they coincide.  Hence there
is at most one finite measure with the prescribed moments.

We next prove existence and convergence.  Let \(\epsilon_j\to0\) be arbitrary.
By asymptotic tightness and Prokhorov's theorem, after passing to a
subsequence,
\[
\mu_{\epsilon_j}\Longrightarrow\widetilde\mu
\]
for some random finite measure \(\widetilde\mu\).

Fix \(k\ge0\) and \(R\ge1\), and set
\[
f_{k,R}(x):=x^k\wedge R^k.
\]
Since \(f_{k,R}\) is bounded and continuous, the map
\[
\nu\longmapsto\int_0^\infty f_{k,R}(x)\,\nu(dx)
\]
is continuous on \(\mathscr M\).  Moreover,
\[
0
\le
M_{\epsilon,k}
-
\int_0^\infty f_{k,R}(x)\,\mu_\epsilon(dx)
\le
\frac1R M_{\epsilon,k+1}.
\]
By Slutsky's theorem,
\[
\bigl(
\mu_{\epsilon_j},
M_{\epsilon_j,k},
M_{\epsilon_j,k+1}
\bigr)
\Longrightarrow
\bigl(
\widetilde\mu,m_k,m_{k+1}
\bigr).
\]
The set
\[
\left\{
(\nu,a,b)\in\mathscr M\times\mathbb R^2:
0\le
a-\int_0^\infty f_{k,R}(x)\,\nu(dx)
\le\frac bR
\right\}
\]
is closed.  Since the prelimit triples belong to this set almost surely,
the Portmanteau theorem gives
\[
0
\le
m_k-\int_0^\infty f_{k,R}(x)\,\widetilde\mu(dx)
\le
\frac{m_{k+1}}R
\qquad\text{almost surely}.
\]
Letting \(R\to\infty\) through the integers and using monotone convergence,
we obtain
\[
\int_0^\infty x^k\,\widetilde\mu(dx)=m_k
\qquad\text{almost surely}.
\]
Since \(k\ge0\) is arbitrary, this holds simultaneously for all \(k\) on an
event of probability one.  By the determinacy proved above,
\(\widetilde\mu\) is almost surely equal to a single deterministic measure
\(\mu\).  This also proves the existence of \(\mu\).

Every subsequential distributional limit is therefore equal to the same
deterministic measure \(\mu\).  Hence
\[
\mu_\epsilon\Longrightarrow\mu
\]
as \(\mathscr M\)-valued random variables. Since \(\mu\) is deterministic, the Portmanteau theorem
implies that, for any metric \(d_{\mathrm w}\) inducing the weak topology,
\[
\mathbb P\!\left(
d_{\mathrm w}(\mu_\epsilon,\mu)>\delta
\right)\longrightarrow0,
\qquad \delta>0.
\]
Thus \(\mu_\epsilon\) converges to \(\mu\) in probability in the weak
topology.
\end{proof}

\begin{lemma}[Tightness of the measure-valued laws]
\label{lem:natural-tight-residue}
With \(\mathscr M\) as in
Lemma~\ref{lem:natural-moment-determinacy}, let
\[
\mathsf P_\epsilon^\chi
:=
\operatorname{Law}\bigl(\nu_\chi^{(\epsilon)}\bigr).
\]
Then the family \(\{\mathsf P_\epsilon^\chi\}_{\epsilon>0}\) is
asymptotically tight on \(\mathscr M\).  Equivalently, for every
\(\eta>0\), there exist \(\epsilon_0>0\) and a compact set
\(\mathcal K_\eta\subset\mathscr M\) such that
\[
\sup_{0<\epsilon<\epsilon_0}
\mathbb P\left(
\nu_\chi^{(\epsilon)}\notin\mathcal K_\eta
\right)
\le\eta.
\]
\end{lemma}

\begin{proof}
For \(j=0,1\), set
\[
M_{\epsilon,j}
:=
\int_0^\infty x^j\,\nu_\chi^{(\epsilon)}(dx).
\]
By Lemma~\ref{lem:natural-moments},
\[
M_{\epsilon,j}
\xrightarrow{\mathbb P}
m_j(\chi),
\qquad j=0,1.
\]
Consequently,
\[
Y_\epsilon
:=
\int_0^\infty(1+x)\,\nu_\chi^{(\epsilon)}(dx)
=
M_{\epsilon,0}+M_{\epsilon,1}
\xrightarrow{\mathbb P}
m_0(\chi)+m_1(\chi).
\]

Fix \(\eta>0\), and choose
\[
B>m_0(\chi)+m_1(\chi).
\]
Define
\[
\mathcal K_B
:=
\left\{
\mu\in\mathscr M:
\int_0^\infty(1+x)\,\mu(dx)\le B
\right\}.
\]
For every \(\mu\in\mathcal K_B\) and \(R>0\),
\[
\mu([0,\infty))\le B,
\qquad
\mu([R,\infty))
\le
\frac{B}{1+R}.
\]
Thus \(\mathcal K_B\) has uniformly bounded total mass and is uniformly
tight.  The finite-measure form of Prokhorov's theorem therefore implies
that \(\mathcal K_B\) is relatively compact in \(\mathscr M\).

The map
\[
\mu\longmapsto
\int_0^\infty(1+x)\,\mu(dx)
\]
is lower semicontinuous under weak convergence.  Hence
\(\mathcal K_B\) is closed and therefore compact.  Finally, for all
sufficiently small \(\epsilon>0\),
\[
\mathbb P\left(
\nu_\chi^{(\epsilon)}\notin\mathcal K_B
\right)
=
\mathbb P(Y_\epsilon>B)
\le\eta.
\]
This proves asymptotic tightness.
\end{proof}

\begin{lemma}[Recovery of the height profile]
\label{lem:natural-height-recovery}
Suppose that
\(\nu_\chi^{(\epsilon)}\) converges weakly in probability to a
deterministic finite measure \(\nu_\chi\).  Then the restriction of
\(\nu_\chi\) to \((0,\infty)\) is absolutely continuous.  Write
\[
\nu_\chi(dx)=f_\chi(x)\,dx,
\qquad x>0.
\]
Then
\[
0\le f_\chi(x)\le2
\qquad\text{for a.e. }x>0.
\]

Moreover, the function
\[
\mathcal H(\chi,\kappa)
:=
\frac1{n\beta}
\int_{(e^{-n\beta\kappa},\infty)}
\frac1x\,\nu_\chi(dx)
\]
is locally absolutely continuous and satisfies
\[
0\le
\partial_\kappa\mathcal H(\chi,\kappa)
\le2
\qquad\text{for a.e. }\kappa.
\]
If \(x=e^{-n\beta\kappa}\) is a Lebesgue point of \(f_\chi\), then
\[
\partial_\kappa\mathcal H(\chi,\kappa)
=
f_\chi(e^{-n\beta\kappa}).
\]
\end{lemma}

\begin{proof}
For \(x>0\), set
\[
\kappa(x):=-\frac{\log x}{n\beta}.
\]
Define a positive locally finite measure \(\sigma_\chi\) on \(\mathbb R\) by
\[
\int_{\mathbb R}\varphi(\kappa)\,\sigma_\chi(d\kappa)
:=
\frac1{n\beta}
\int_{(0,\infty)}
\varphi\bigl(\kappa(x)\bigr)\frac1x\,\nu_\chi(dx),
\qquad
\varphi\in C_c(\mathbb R).
\]

For \(\varphi\in C_c(\mathbb R)\), define
\[
g_\varphi(x):=
\begin{cases}
\dfrac1{n\beta x}\,
\varphi\!\left(-\dfrac{\log x}{n\beta}\right),
& x>0,\\[3mm]
0,&x=0.
\end{cases}
\]
Since \(\varphi\) has compact support, \(g_\varphi\) vanishes outside a
compact subinterval of \((0,\infty)\); in particular,
\(g_\varphi\in C_b([0,\infty))\).  By \eqref{dnu-eps},
\[
\int_{\mathbb R}\varphi(\kappa)\,dH_\epsilon(\kappa)
=
\int_0^\infty g_\varphi(x)\,
\nu_\chi^{(\epsilon)}(dx).
\]
The assumed weak convergence in probability of
\(\nu_\chi^{(\epsilon)}\) therefore gives
\[
\int_{\mathbb R}\varphi(\kappa)\,dH_\epsilon(\kappa)
\xrightarrow{\mathbb P}
\int_{\mathbb R}\varphi(\kappa)\,\sigma_\chi(d\kappa).
\]
Equivalently,
\[
dH_\epsilon\longrightarrow\sigma_\chi
\]
in probability in the vague topology on locally finite measures on
\(\mathbb R\).

Let \(I=(a,b)\) be a bounded interval such that
\(\sigma_\chi(\{a,b\})=0\).  Since \(I\) is a continuity set of
\(\sigma_\chi\), the continuity-set characterization of vague convergence
implies
\[
dH_\epsilon(I)\xrightarrow{\mathbb P}\sigma_\chi(I).
\]
On the other hand, Lemma~\ref{lem:natural-positive-slope} gives
\[
dH_\epsilon(I)\le 2|I|+O(\epsilon).
\]
Hence
\[
\sigma_\chi(I)\le2|I|.
\]
A locally finite measure has at most countably many atoms, so every bounded
open interval can be approximated from outside by intervals whose endpoints
are continuity points of \(\sigma_\chi\).  It follows that the same bound
holds for every bounded open interval.  By regularity,
\[
\sigma_\chi\le 2\,d\kappa.
\]
Consequently,
\[
\sigma_\chi(d\kappa)=s_\chi(\kappa)\,d\kappa
\]
for some measurable \(s_\chi\) satisfying
\[
0\le s_\chi(\kappa)\le2
\qquad\text{for a.e. }\kappa.
\]

Now let \(\psi\in C_c((0,\infty))\).  Applying the definition of
\(\sigma_\chi\) to
\[
\varphi(\kappa)
:=
n\beta e^{-n\beta\kappa}
\psi(e^{-n\beta\kappa})
\]
and then changing variables \(x=e^{-n\beta\kappa}\), we obtain
\[
\begin{aligned}
\int_0^\infty\psi(x)\,\nu_\chi(dx)
&=
\int_{\mathbb R}
n\beta e^{-n\beta\kappa}
\psi(e^{-n\beta\kappa})\,
\sigma_\chi(d\kappa)\\
&=
\int_0^\infty
\psi(x)\,
s_\chi\!\left(-\frac{\log x}{n\beta}\right)dx.
\end{aligned}
\]
Thus the restriction of \(\nu_\chi\) to \((0,\infty)\) is absolutely
continuous, with
\[
f_\chi(x)
=
s_\chi\!\left(-\frac{\log x}{n\beta}\right)
\qquad\text{for a.e. }x>0.
\]
In particular,
\[
0\le f_\chi(x)\le2
\qquad\text{for a.e. }x>0.
\]

Finally, by the definition of \(\sigma_\chi\),
\[
\begin{aligned}
\mathcal H(\chi,\kappa)
&=
\frac1{n\beta}
\int_{(e^{-n\beta\kappa},\infty)}
\frac1x\,\nu_\chi(dx)\\
&=
\sigma_\chi((-\infty,\kappa))
=
\int_{-\infty}^{\kappa}s_\chi(u)\,du.
\end{aligned}
\]
Hence \(\mathcal H(\chi,\cdot)\) is locally absolutely continuous and
\[
\partial_\kappa\mathcal H(\chi,\kappa)
=
s_\chi(\kappa)
\]
for almost every \(\kappa\), so in particular
\[
0\le\partial_\kappa\mathcal H(\chi,\kappa)\le2.
\]

Equivalently,
\[
\mathcal H(\chi,\kappa)
=
\frac1{n\beta}
\int_{e^{-n\beta\kappa}}^\infty
\frac{f_\chi(x)}{x}\,dx.
\]
If \(x_\kappa=e^{-n\beta\kappa}\) is a Lebesgue point of \(f_\chi\),
then it is also a Lebesgue point of \(f_\chi(x)/x\).  The fundamental
theorem for Lebesgue integrals and
\[
\frac{d}{d\kappa}x_\kappa=-n\beta x_\kappa
\]
therefore give
\[
\partial_\kappa\mathcal H(\chi,\kappa)
=
f_\chi(x_\kappa)
=
f_\chi(e^{-n\beta\kappa}),
\]
as claimed.
\end{proof}

\subsection{Interlacing assumptions on zeros and poles}\label{s53}

We now formulate conditions under which the equation
\[
\mathcal{G}_{\chi}(w)\prod_{k\ge1}\mathcal{F}_{u,v,k}(w)=s^n
\]
has at most one pair of nonreal conjugate roots. Informally, the assumptions below
say that the poles and zeros of the left-hand side are simple and strictly interlaced on
the real line. We use the zero--pole location sets
\(R_{\chi,i,j}\) and \(R_{i,j,k_1,k_2}\) introduced in
Definition~\ref{def:R-sets}.

Throughout the remainder of this section, for \(K\geq1\), set
\begin{align}
S_{\chi,K}(w)
:=
G_\chi(w)\prod_{r=1}^{K}F_{u,v,r}(w),
\qquad
C_{\chi,K}
:=
\lim_{w\to\infty}S_{\chi,K}(w).\label{dsck}
\end{align}
We reserve \(S_\chi\) for the corresponding infinite product.

\begin{assumption}[Zero--pole interlacing and separation]
\label{ap62}
Fix
\[
\chi\in(V_0,V_m)\setminus\{V_0,\ldots,V_m\}.
\]
For each \(K\ge1\), define the inward pole and numerator-zero sets by
\begin{align*}
D_{1,K}
&:=
\bigl(
\mathcal R_{\chi,1,1}\setminus\mathcal R_{\chi,1,2}
\bigr)
\cup
\bigl(
\mathcal R_{\chi,3,1}\setminus\mathcal R_{\chi,3,2}
\bigr)
\\
&\quad\cup
\bigcup_{k\in[K]}
\Bigl[
\bigl(
\mathcal R_{5,1,k,k}\setminus\mathcal R_{5,2,k,k}
\bigr)
\cup
\bigl(
\mathcal R_{7,1,k,k}\setminus\mathcal R_{7,2,k,k}
\bigr)
\\
&\hspace{31mm}\cup
\bigl(
\mathcal R_{5,1,k-1,k}\setminus\mathcal R_{5,2,k-1,k}
\bigr)
\cup
\bigl(
\mathcal R_{7,1,k-1,k}\setminus\mathcal R_{7,2,k-1,k}
\bigr)
\Bigr],
\\[1mm]
D_{2,K}
&:=
\bigl(
\mathcal R_{\chi,1,2}\setminus\mathcal R_{\chi,1,1}
\bigr)
\cup
\bigl(
\mathcal R_{\chi,3,2}\setminus\mathcal R_{\chi,3,1}
\bigr)
\\
&\quad\cup
\bigcup_{k\in[K]}
\Bigl[
\bigl(
\mathcal R_{5,2,k,k}\setminus\mathcal R_{5,1,k,k}
\bigr)
\cup
\bigl(
\mathcal R_{7,2,k,k}\setminus\mathcal R_{7,1,k,k}
\bigr)
\\
&\hspace{31mm}\cup
\bigl(
\mathcal R_{5,2,k-1,k}\setminus\mathcal R_{5,1,k-1,k}
\bigr)
\cup
\bigl(
\mathcal R_{7,2,k-1,k}\setminus\mathcal R_{7,1,k-1,k}
\bigr)
\Bigr].
\end{align*}
Define the outward pole and numerator-zero sets by
\begin{align*}
D_{3,K}
&:=
\bigl(
\mathcal R_{\chi,2,1}\setminus\mathcal R_{\chi,2,2}
\bigr)
\cup
\bigl(
\mathcal R_{\chi,4,1}\setminus\mathcal R_{\chi,4,2}
\bigr)
\\
&\quad\cup
\bigcup_{k\in[K]}
\Bigl[
\bigl(
\mathcal R_{6,1,k,k}\setminus\mathcal R_{6,2,k,k}
\bigr)
\cup
\bigl(
\mathcal R_{6,1,k,k-1}\setminus\mathcal R_{6,2,k,k-1}
\bigr)
\\
&\hspace{31mm}\cup
\bigl(
\mathcal R_{8,1,k,k}\setminus\mathcal R_{8,2,k,k}
\bigr)
\cup
\bigl(
\mathcal R_{8,1,k,k-1}\setminus\mathcal R_{8,2,k,k-1}
\bigr)
\Bigr],
\\[1mm]
D_{4,K}
&:=
\bigl(
\mathcal R_{\chi,2,2}\setminus\mathcal R_{\chi,2,1}
\bigr)
\cup
\bigl(
\mathcal R_{\chi,4,2}\setminus\mathcal R_{\chi,4,1}
\bigr)
\\
&\quad\cup
\bigcup_{k\in[K]}
\Bigl[
\bigl(
\mathcal R_{6,2,k,k}\setminus\mathcal R_{6,1,k,k}
\bigr)
\cup
\bigl(
\mathcal R_{6,2,k,k-1}\setminus\mathcal R_{6,1,k,k-1}
\bigr)
\\
&\hspace{31mm}\cup
\bigl(
\mathcal R_{8,2,k,k}\setminus\mathcal R_{8,1,k,k}
\bigr)
\cup
\bigl(
\mathcal R_{8,2,k,k-1}\setminus\mathcal R_{8,1,k,k-1}
\bigr)
\Bigr].
\end{align*}

Assume the following.

\begin{enumerate}
\item
For every \(i\in[4]\) and \(\ell\in[2]\), distinct active labels
\((r,p)\in[n]\times[m]\) give distinct points of
\(\mathcal R_{\chi,i,\ell}\).  Likewise, for every
\[
i\in\{5,6,7,8\},
\qquad
\ell\in[2],
\qquad
k_1,k_2\in\mathbb Z_{\ge0},
\qquad
k_1+k_2\ge1,
\]
distinct active labels give distinct points of
\(\mathcal R_{i,\ell,k_1,k_2}\).

\item
The inward pole and zero sets have the same cardinality:
\[
|D_{1,K}|=|D_{2,K}|=:N_K.
\]
If \(N_K>0\), write
\[
D_{1,K}
=
\{p_{1,K}<\cdots<p_{N_K,K}\},
\qquad
D_{2,K}
=
\{z_{1,K}<\cdots<z_{N_K,K}\}.
\]
Then
\[
p_{1,K}<z_{1,K}<p_{2,K}<z_{2,K}
<\cdots<p_{N_K,K}<z_{N_K,K}.
\]
When \(N_K=0\), this condition is void.

\item
Set
\[
\mathbb R_-:=(-\infty,0),
\qquad
\mathbb R_+:=(0,\infty),
\]
and, for \(\sigma\in\{-,+\}\), define
\[
D_{j,K}^{\sigma}
:=
D_{j,K}\cap\mathbb R_\sigma,
\qquad j\in\{3,4\}.
\]
On each half-axis, the outward pole and zero sets have the same
cardinality:
\[
|D_{3,K}^{\sigma}|
=
|D_{4,K}^{\sigma}|
=:M_K^\sigma.
\]
If \(M_K^\sigma>0\), write
\[
D_{3,K}^{\sigma}
=
\{p_{1,K}^{\sigma}<\cdots<
  p_{M_K^\sigma,K}^{\sigma}\},
\qquad
D_{4,K}^{\sigma}
=
\{z_{1,K}^{\sigma}<\cdots<
  z_{M_K^\sigma,K}^{\sigma}\}.
\]
Then
\[
z_{1,K}^{\sigma}<p_{1,K}^{\sigma}
<z_{2,K}^{\sigma}<p_{2,K}^{\sigma}
<\cdots
<z_{M_K^\sigma,K}^{\sigma}
<p_{M_K^\sigma,K}^{\sigma}.
\]
When \(M_K^\sigma=0\), this condition is void.

\item
Define the inward and outward sets on the two half-axes by
\[
I_K^\sigma
:=
(D_{1,K}\cup D_{2,K})\cap\mathbb R_\sigma,
\qquad
O_K^\sigma
:=
(D_{3,K}\cup D_{4,K})\cap\mathbb R_\sigma,
\qquad
\sigma\in\{-,+\}.
\]
There exist constants
\[
a_-<b_-<0<b_+<a_+,
\]
which may depend on \(\chi\) but not on \(K\), such that, for every
\(K\ge1\),
\[
O_K^-\subset(-\infty,a_-],
\qquad
I_K^-\subset[b_-,0),
\]
and
\[
I_K^+\subset(0,b_+],
\qquad
O_K^+\subset[a_+,\infty).
\]
An inclusion involving an empty set is understood to be vacuous.
\end{enumerate}
\end{assumption}

\begin{revision}
Figure~\ref{fig:zp} illustrates the separating integration contour
$\mathcal C_\chi$ used in Proposition~\ref{p57} and constructed in
Corollary~\ref{cor:automatic-branch}. It can be chosen independently of $K$ so that
\[
\begin{aligned}
\{0\}\cup D_{1,K}\cup D_{2,K}
&\subset\operatorname{Int}(\mathcal C_\chi),\\
D_{3,K}\cup D_{4,K}
&\subset\operatorname{Ext}(\mathcal C_\chi),
\qquad K\ge1.
\end{aligned}
\]
Thus all inward poles and zeros lie inside, while all outward poles and zeros
lie outside. Taking the unions over $K$ gives the same inclusions for the
infinite product; the possible accumulation of the inward families at $0$
remains inside the contour. The contour itself meets no zero or pole.
It lies in the complex $w$-plane and is distinct from the frozen boundary
in the $(\chi,\kappa)$-plane.
\end{revision}

\begin{figure}[tbp]
\centering
\begingroup\color{revisioncolor}
\begin{tikzpicture}[x=.99cm,y=1cm,>=Latex,
  every node/.style={font=\small,text=revisioncolor},
  every path/.style={draw=revisioncolor}]
  \draw[->,thin] (-7.25,0)--(7.35,0) node[right] {$\mathbb R$};
  \draw[line width=.8pt] (0,0) ellipse [x radius=3.65,y radius=1.12];
  \draw[->,line width=.8pt]
    ({3.65*cos(37)},{1.12*sin(37)})
    arc[start angle=37,end angle=73,x radius=3.65,y radius=1.12];
  \node[above] at (0,1.17) {$\mathcal C_\chi$};
  \foreach \x in {-6.2,-1.15,2.45,4.9}
    \fill[revisioncolor] (\x,0) circle (2pt);
  \foreach \x in {-4.9,-2.45,1.15,6.2}
    \draw[fill=white,line width=.7pt] (\x,0) circle (2.4pt);
  \node[below=5pt] at (-6.2,0) {$z^-_{1,K}$};
  \node[below=5pt] at (-4.9,0) {$p^-_{1,K}$};
  \node[below=5pt] at (-2.45,0) {$p_{1,K}$};
  \node[below=5pt] at (-1.15,0) {$z_{1,K}$};
  \node[below=5pt] at (1.15,0) {$p_{2,K}$};
  \node[below=5pt] at (2.45,0) {$z_{2,K}$};
  \node[below=5pt] at (4.9,0) {$z^+_{1,K}$};
  \node[below=5pt] at (6.2,0) {$p^+_{1,K}$};
  \draw[thin] (0,-.075)--(0,.075);
  \node[above=4pt] at (0,0) {$0$};
  \node[align=center] at (-5.55,-1.67)
    {$D^-_{3,K}\cup D^-_{4,K}$\\outside};
  \node[align=center] at (0,-1.67)
    {$D_{1,K}\cup D_{2,K}$\\inside};
  \node[align=center] at (5.55,-1.67)
    {$D^+_{3,K}\cup D^+_{4,K}$\\outside};
  \fill[revisioncolor] (-1.62,-2.6) circle (2pt);
  \node[right=4pt] at (-1.62,-2.6) {zero};
  \draw[fill=white,line width=.7pt] (.5,-2.6) circle (2.4pt);
  \node[right=4pt] at (.5,-2.6) {pole};
\end{tikzpicture}
\captionsetup{labelfont={color=revisioncolor},textfont={color=revisioncolor}}
\caption{A schematic finite-truncation configuration illustrating the ordering
and separation in Assumption~\ref{ap62}. Filled circles are zeros and open
circles are poles. The labels $p_{j,K},z_{j,K}$ denote inward poles and zeros;
$p^{\pm}_{j,K},z^{\pm}_{j,K}$ denote outward poles and zeros on the two
half-axes. The positively oriented contour $\mathcal C_\chi$ encloses the
origin and the inward families, and excludes the outward families. The origin
is marked by a tick; it is not designated as a zero or pole. Positions are
schematic.}
\label{fig:zp}\label{rev:comment14}
\endgroup
\end{figure}
\FloatBarrier

In the next two assumptions we give explicit conditions on the graph parameters which
guarantee Assumption~\ref{ap62}.

\begin{assumption}\label{ap64}
Let $i,j$ be integers such that
\[
\epsilon i\in(V_{p_1-1},V_{p_1}),
\qquad
\epsilon j\in(V_{p_2-1},V_{p_2}),
\]
and let $i_*,j_*\in[n]$ satisfy
\[
(i-i_*)\equiv 0\pmod n,
\qquad
(j-j_*)\equiv 0\pmod n.
\]
Assume the following.
\begin{enumerate}
\item If $a_i^{(\epsilon)}=a_j^{(\epsilon)}$, $b_i^{(\epsilon)}=b_j^{(\epsilon)}$,
$i_*\neq j_*$, and $p_1=p_2$, then $\tau_{i_*}\neq \tau_{j_*}$.

\item If $a_i^{(\epsilon)}=a_j^{(\epsilon)}$, $b_i^{(\epsilon)}=-$, $b_j^{(\epsilon)}=+$,
and $p_1\ge p_2$, then
\[
\tau_{i_*}^{-1}\tau_{j_*}<e^{V_{p_2-1}-V_{p_1}}.
\]

\item If $a_i^{(\epsilon)}=a_j^{(\epsilon)}$, $b_i^{(\epsilon)}=b_j^{(\epsilon)}$,
and $\tau_{i_*}>\tau_{j_*}$, then
\[
\tau_{i_*}^{-1}\tau_{j_*}<e^{V_{p_2-1}-V_{p_1}}.
\]
\end{enumerate}
\end{assumption}

\begin{assumption}\label{ap65}
Assume the following.
\begin{enumerate}
\item For any $(j_1,p_1),(j_2,p_2)$ such that
$\mathbf 1_{\mathcal E_{j_1,p_1,>,1}}=\mathbf 1_{\mathcal E_{j_2,p_2,>,1}}=1$, one has
\[
\max\{u,v\}<e^{V_{p_2-1}-V_{p_1}}\tau_{j_2}^{-1}\tau_{j_1}.
\]

\item For any $(j_1,p_1),(j_2,p_2)$ such that
$\mathbf 1_{\mathcal E_{j_1,p_1,<,1}}=\mathbf 1_{\mathcal E_{j_2,p_2,<,1}}=1$, one has
\[
\max\{u,v\}<e^{V_{p_2-1}-V_{p_1}}\tau_{j_2}^{-1}\tau_{j_1}.
\]

\item For any $(j_1,p_1),(j_2,p_2)$ such that
$\mathbf 1_{\mathcal E_{j_1,p_1,>,0}}=\mathbf 1_{\mathcal E_{j_2,p_2,>,0}}=1$, one has
\[
\max\{u,v\}<e^{V_{p_2}-V_{p_1}}\tau_{j_2}^{-1}\tau_{j_1}.
\]

\item For any $(j_1,p_1),(j_2,p_2)$ such that
$\mathbf 1_{\mathcal E_{j_1,p_1,<,0}}=\mathbf 1_{\mathcal E_{j_2,p_2,<,0}}=1$, one has
\[
\max\{u,v\}<e^{V_{p_2-1}-V_{p_1-1}}\tau_{j_2}^{-1}\tau_{j_1}.
\]
\end{enumerate}
\end{assumption}

\begin{lemma}\label{l64}
Assumptions~\ref{ap64} and~\ref{ap65} imply
Assumption~\ref{ap62}.
\end{lemma}

\begin{proof}
Fix \(K\ge1\), we first regard the zeros and poles of $S_{\chi, K}$ as multisets and then cancel common
numerator and denominator points, as in the set differences defining
\(D_{i,K}\).

A \emph{family} is either one of the four unreflected pairs
\[
(\mathcal R_{\chi,i,1},\mathcal R_{\chi,i,2}),
\qquad i\in[4],
\]
or a reflected pair
\[
(\mathcal R_{i,1,k_1,k_2},\mathcal R_{i,2,k_1,k_2}),
\qquad i\in\{5,6,7,8\},
\]
at one of the scales occurring in the \(K\)-truncation.  In every family,
the first set is the raw pole set and the second is the raw numerator-zero
set.  Each factor contributes one pole and one zero.  A family is allowed
to be empty; in that case both lists are empty and the family contributes
nothing to the ordered concatenations below.

Fix a family and one of its two sides \(\ell\in\{1,2\}\).
By Definition~\ref{def:R-sets}, the points on this side are indexed by
the active pairs \((r,p)\in[n]\times[m]\).
Assumption~\ref{ap64}(1),(3) implies that distinct active pairs give
distinct locations in each unreflected set
\(\mathcal R_{\chi,i,\ell}\).
Each reflected set \(\mathcal R_{i,\ell,k_1,k_2}\) is obtained from the
corresponding base locations by multiplication by a fixed positive
\(u,v\)-scale, so the same conclusion holds for the reflected families.
Hence every raw pole list and every raw numerator-zero list is simple.
This proves Assumption~\ref{ap62}(1).

We next determine the order of the families.  For subsets
\(A,B\subset\mathbb R\), write \(A\prec B\) if
\(a<b\) for every \(a\in A\) and \(b\in B\).  Set
\[
\mathscr R_{\chi,i}
:=
\mathcal R_{\chi,i,1}\cup\mathcal R_{\chi,i,2},
\qquad i\in[4],
\]
and
\[
\mathscr R_{i,K}
:=
\begin{cases}
\displaystyle
\bigcup_{k=1}^{K}\bigcup_{j=1}^{2}
\left(
\mathcal R_{i,j,k,k}\cup\mathcal R_{i,j,k-1,k}
\right),
& i\in\{5,7\},\\[4mm]
\displaystyle
\bigcup_{k=1}^{K}\bigcup_{j=1}^{2}
\left(
\mathcal R_{i,j,k,k}\cup\mathcal R_{i,j,k,k-1}
\right),
& i\in\{6,8\}.
\end{cases}
\]

For each \(i\in[4]\) and each label \((j,p)\) active in the
corresponding family, the associated pole and numerator zero are ordered
as follows:
\[
\begin{array}{c|c}
i & \text{left-to-right order}\\ \hline
1,3 & \mathrm p<\mathrm z,\\
2,4 & \mathrm z<\mathrm p.
\end{array}
\]
This follows immediately from Definition~\ref{def:R-sets} and the
inequalities defining an active label.  The reflected families \(5,7\)
inherit the first ordering, while \(6,8\) inherit the second, because their
locations are obtained by multiplying both endpoints by the same positive
\(u,v\)-scale.

It remains to establish the alternating order within each fixed family.
Fix either an unreflected pair
\[
(\mathcal R_{\chi,i,1},\mathcal R_{\chi,i,2}),
\qquad i\in[4],
\]
or a reflected pair
\[
(\mathcal R_{i,1,k_1,k_2},\mathcal R_{i,2,k_1,k_2}),
\qquad i\in\{5,6,7,8\},
\]
with \((k_1,k_2)\) fixed.  For each residue \(j\), the explicit endpoint
formulas and \(V_0<\cdots<V_m\) order the intervals indexed by the active
macroscopic blocks \(p\).  Consecutive intervals may share only an endpoint,
which is removed upon cancellation.  Assumption~\ref{ap64}(1),(3) separates
the chains corresponding to distinct residues.  Positive \(u,v\)-rescaling
preserves this order for the reflected families.

Consequently, every nonempty reduced family contains an equal positive
number of poles and numerator zeros, and its ordered lists have the form
\begin{equation}\label{eq:inward-family-order}
p_1<z_1<p_2<z_2<\cdots<p_q<z_q,
\qquad i\in\{1,3,5,7\},
\end{equation}
or
\begin{equation}\label{eq:outward-family-order}
z_1<p_1<z_2<p_2<\cdots<z_q<p_q,
\qquad i\in\{2,4,6,8\}.
\end{equation}
For an empty reduced family, both cardinalities are zero and these
statements are understood to be void.

Assumption~\ref{ap64}(2) separates the two unreflected families on each
half-axis, while Assumption~\ref{ap65} places the inward reflected families
between the unreflected families and \(0\), places the outward reflected
families beyond the unreflected families, and orders their successive
copies.  Hence, after empty blocks are omitted,
\begin{equation}\label{eq:global-family-order}
\mathscr R_{8,K}
\prec
\mathscr R_{\chi,4}
\prec
\mathscr R_{\chi,3}
\prec
\mathscr R_{7,K}
\prec
\{0\}
\prec
\mathscr R_{5,K}
\prec
\mathscr R_{\chi,1}
\prec
\mathscr R_{\chi,2}
\prec
\mathscr R_{6,K}.
\end{equation}

The sets \(D_{1,K}\) and \(D_{2,K}\) are the pole and zero lists of the
inward families
\[
3,\quad 7,\quad 5,\quad 1,
\]
in the order displayed in \eqref{eq:global-family-order}.  Every nonempty
such family starts with a pole and ends with a zero by
\eqref{eq:inward-family-order}, whereas an empty family contributes no
points.  Concatenating the lists therefore gives
\[
|D_{1,K}|=|D_{2,K}|=:N_K.
\]
If \(N_K>0\), writing
\[
D_{1,K}
=
\{p_{1,K}<\cdots<p_{N_K,K}\},
\qquad
D_{2,K}
=
\{z_{1,K}<\cdots<z_{N_K,K}\},
\]
we obtain
\[
p_{1,K}<z_{1,K}<p_{2,K}<z_{2,K}
<\cdots<p_{N_K,K}<z_{N_K,K}.
\]
When \(N_K=0\), the assertion is void.  This proves
Assumption~\ref{ap62}(2).

Likewise, \(D_{3,K}\) and \(D_{4,K}\) are the pole and zero lists of the
outward families
\[
8,\quad 4,\quad 2,\quad 6.
\]
On the negative half-axis the relevant families are \(8,4\), and on the
positive half-axis they are \(2,6\).  Every nonempty outward family starts
with a zero and ends with a pole by
\eqref{eq:outward-family-order}, while an empty family contributes no
points.  Consequently, for each \(\sigma\in\{-,+\}\),
\[
|D_{3,K}^{\sigma}|=|D_{4,K}^{\sigma}|=:M_K^\sigma.
\]
If \(M_K^\sigma>0\), writing
\[
D_{3,K}^{\sigma}
=
\{p_{1,K}^{\sigma}<\cdots<
  p_{M_K^\sigma,K}^{\sigma}\},
\qquad
D_{4,K}^{\sigma}
=
\{z_{1,K}^{\sigma}<\cdots<
  z_{M_K^\sigma,K}^{\sigma}\},
\]
we obtain
\[
z_{1,K}^{\sigma}<p_{1,K}^{\sigma}
<z_{2,K}^{\sigma}<p_{2,K}^{\sigma}
<\cdots
<z_{M_K^\sigma,K}^{\sigma}
<p_{M_K^\sigma,K}^{\sigma}.
\]
When \(M_K^\sigma=0\), the assertion is void.  This proves
Assumption~\ref{ap62}(3).

Finally, \eqref{eq:global-family-order} separates the outward and inward
families on each half-axis.  Increasing \(K\) adds new type-\(5\) and
type-\(7\) copies only closer to \(0\), while the new type-\(6\) and
type-\(8\) copies lie only farther from \(0\).  Hence the two separating
gaps may be chosen independently of \(K\).  Thus there exist constants
\[
a_-<b_-<0<b_+<a_+
\]
such that, for every \(K\ge1\),
\[
O_K^-\subset(-\infty,a_-],
\qquad
I_K^-\subset[b_-,0),
\]
and
\[
I_K^+\subset(0,b_+],
\qquad
O_K^+\subset[a_+,\infty).
\]
If one of these sets is empty, the corresponding inclusion is vacuous.
This proves Assumption~\ref{ap62}(4).
\end{proof}

\begin{lemma}\label{lem:annulus-implies-LL}
Let \(\mathcal C_z\) and \(\mathcal C_w\) be disjoint one-point contours
contained in
\[
\mathcal A_{u,v}:=\{w\in\mathbb C:v<|w|<u^{-1}\}.
\]
Then \((\mathcal C_z,\mathcal C_w)\) is \(LL\)-admissible in the sense of
Definition~\ref{def:KLL}.
\end{lemma}

\begin{proof}
Recall from Definition~\ref{def:KLL} that
\[
c_r=(u^2v^2)^r,
\qquad
a_r=v^2(u^2v^2)^{r-1},
\qquad
b_r=u^2(u^2v^2)^{r-1}.
\]
Fix \(z\in\mathcal C_z\), \(w\in\mathcal C_w\), and \(r\ge1\).
Since the two contours are disjoint, \(z\ne w\).

Because
\[
v<|z|,|w|<u^{-1},
\]
we have
\[
|c_rw|
<
\frac{c_r}{u}
=
v(uv)^{2r-1}
<
v
<
|z|.
\]
Hence \(z\ne c_rw\).  Interchanging \(z\) and \(w\) gives
\(w\ne c_rz\).

Moreover,
\[
|zw|>v^2\ge v^2(u^2v^2)^{r-1}=a_r,
\]
so \(zw\ne a_r\).  Finally,
\[
|b_rzw|
<
b_ru^{-2}
=
(u^2v^2)^{r-1}
\le1,
\]
and the first inequality is strict; hence
\[
|b_rzw|<1,
\]
so \(b_rzw\ne1\).

Thus none of the interaction-pole equations
\[
z=w,\qquad
z=c_rw,\qquad
w=c_rz,\qquad
zw=a_r,\qquad
b_rzw=1
\]
can occur on
\(\mathcal C_z\times\mathcal C_w\).  Since
\(\mathcal C_z\) and \(\mathcal C_w\) are already one-point contours, they
satisfy the remaining contour requirements in
Definition~\ref{def:KLL}.  Therefore
\((\mathcal C_z,\mathcal C_w)\) is \(LL\)-admissible.
\end{proof}

\begin{corollary}[Automatic branch and pairwise \(LL\)-admissibility]
\label{cor:automatic-branch}
Assume the hypotheses of Lemma~\ref{l64}.  For \(j\in[4]\), set
\[
D_{j,\infty}:=\bigcup_{K\ge1}D_{j,K}.
\]
Then there exist a positively oriented simple closed contour
\(\mathcal C_\chi\) and a connected annular neighbourhood
\(U_\chi\Subset\mathbb C^*\) of \(\mathcal C_\chi\) such that
\[
\{0\}\cup D_{1,\infty}\cup D_{2,\infty}
\subset\operatorname{Int}(\mathcal C_\chi),
\qquad
D_{3,\infty}\cup D_{4,\infty}
\subset\operatorname{Ext}(\mathcal C_\chi),
\]
and \(S_\chi\) is holomorphic and non-vanishing on \(U_\chi\) and
admits a holomorphic logarithm there.

Consequently, every finite nested contour family contained in \(U_\chi\),
together with all contour deformations performed inside \(U_\chi\), is
branch-admissible in the sense of
Definition~\ref{def:branch-admissible}.

If, in addition, Assumption~\ref{ass:admissible} holds, then, for every
finite collection of marked locations and every prescribed finite number
of nested contour copies at each location, the copies may be chosen
pairwise \(LL\)-admissibly.
\end{corollary}

\begin{proof}
The global ordering in the proof of Lemma~\ref{l64}, together with the
monotonicity of the reflected copies, gives a contour
\(\mathcal C_\chi\), independent of \(K\), separating the inward and
outward zero--pole families.

The zeros and poles of \(S_{\chi,K}\) inside \(\mathcal C_\chi\) are
\(D_{2,K}\) and \(D_{1,K}\), respectively.  By
Assumption~\ref{ap62}(2),
\[
|D_{1,K}|=|D_{2,K}|.
\]
Hence the argument principle gives
\[
\frac{1}{2\pi\mathbf i}
\oint_{\mathcal C_\chi}
\frac{S_{\chi,K}'(w)}{S_{\chi,K}(w)}\,dw
=0.
\]

Choose a connected annular neighbourhood \(U_\chi\) of
\(\mathcal C_\chi\) whose closure is disjoint from all zero and pole
locations.  The explicit formulas for \(\mathcal F_{u,v,r}\), together
with \(u,v\in(0,1)\), give normal convergence of
\[
\prod_{r\ge1}\mathcal F_{u,v,r}
\]
on \(U_\chi\).  Thus \(S_{\chi,K}\to S_\chi\) locally uniformly there,
and the limit is holomorphic and non-vanishing.  Passing to the limit gives
\[
\frac{1}{2\pi\mathbf i}
\oint_{\mathcal C_\chi}
\frac{S_\chi'(w)}{S_\chi(w)}\,dw
=0.
\]
Since \(U_\chi\) is annular, every closed curve in \(U_\chi\) is homologous
to an integer multiple of \(\mathcal C_\chi\).  Therefore \(S_\chi\) has
zero winding along every closed curve in \(U_\chi\), and hence admits a
holomorphic logarithm there.  This proves the branch-admissibility
assertion.

It remains to prove the final assertion.  Order the marked locations as
\[
\chi_1\le\cdots\le\chi_s.
\]
We claim that the corresponding separating Jordan domains may be chosen
so that
\[
\overline{\Omega_{\chi_s}}
\subset
\Omega_{\chi_{s-1}}\subset \overline{\Omega_{\chi_{s-1}}}
\subset\cdots\subset
\Omega_{\chi_1}.
\]

Indeed, fix \(d<h\).  Let \((j_-,p_-)\) be active in a
\(>\)-family at \(\chi_h\), and let \((j_+,p_+)\) be active in the
corresponding \(<\)-family at \(\chi_d\).  Since
\[
V_{p_-}>\chi_h\ge\chi_d>V_{p_+-1},
\]
we have \(p_-\ge p_+\).  Assumption~\ref{ap64}(2) therefore gives
\[
e^{V_{p_-}}\tau_{j_-}^{-1}
<
e^{V_{p_+-1}}\tau_{j_+}^{-1}.
\]
Hence, on the positive half-axis, every inward interval associated with
\(\chi_h\) lies strictly before every outward interval associated with
\(\chi_d\).  The corresponding negative-axis statement follows after
multiplication by \(-1\).  Assumption~\ref{ap65} places the reflected
inward families closer to \(0\) and the reflected outward families farther
from \(0\), so the same cross-separation holds for the complete families.
The separating domains can therefore be chosen with the stated reverse
nesting.  Repeated copies at the same marked location are taken as
sufficiently close nested parallel copies.

We next show that these contours may be chosen inside
\[
A_{u,v}:=\{w\in\mathbb C:v<|w|<u^{-1}\}.
\]
The \(k=1\) one-body bounds in
Assumption~\ref{ass:admissible} give
\[
u x_i^{(\epsilon)}\le1-\delta
\quad\text{if }b_i^{(\epsilon)}=-,
\qquad
v x_i^{(\epsilon)}\le1-\delta
\quad\text{if }b_i^{(\epsilon)}=+.
\]
Passing to the scaling limit shows that every unreflected inward
zero--pole location has modulus at most
\[
\frac{1-\delta}{u}<u^{-1},
\]
whereas every unreflected outward location has modulus at least
\[
\frac{v}{1-\delta}>v.
\]
The reflected inward locations are obtained by multiplying by factors
strictly smaller than \(1\), while the reflected outward locations are
obtained by multiplying by factors strictly larger than \(1\).  Hence the
same bounds hold for all reflected locations.

Together with the inward--outward separation in
Assumption~\ref{ap62}(4), these bounds show that, on each half-axis, the
separating gap meets the radial interval \((v,u^{-1})\).  We may therefore
choose all the strictly nested separating contours inside \(A_{u,v}\);
after shrinking their annular neighbourhoods if necessary, all required
nested copies remain there.

The contours are thus pairwise disjoint one-point contours contained in
\(A_{u,v}\).  Lemma~\ref{lem:annulus-implies-LL} implies that every two
distinct contour copies are \(LL\)-admissible.
\end{proof}

\subsection{Stieltjes transform in the natural variable and the slope}

Let \(\nu_\chi\) be the limiting slope measure of
Theorem~\ref{thm:weak-limit-shape-beta}.  Recall the functions
\(S_\chi\) and \(T_\chi\) defined in \eqref{dsc} and
\eqref{eq:Tchi}, respectively.
By \eqref{nu-moments}, the Stieltjes transform
\[
        \mathrm{St}_{\nu_\chi}(\zeta):=\int_0^\infty \frac{\nu_\chi(dx)}{\zeta-x}
\]
has, for $|\zeta|$ large, the expansion
\begin{align}
\mathrm{St}_{\nu_\chi}(\zeta)
&=
\sum_{k\ge1}\frac{1}{\zeta^k}\int_0^\infty x^{k-1}\nu_\chi(dx)\notag\\
&=
\frac{1}{\pi\mathbf i}
\oint_{\mathcal C}\sum_{k\ge1}\frac{T_\chi(w)^k}{k\zeta^k}\frac{dw}{w}
\notag\\
&=
-\frac{1}{\pi\mathbf i}
\oint_{\mathcal C}\log\left(1-\frac{T_\chi(w)}{\zeta}\right)\frac{dw}{w}.
\label{ctx}
\end{align}
Since $\nu_\chi$ is compactly supported by Theorem~\ref{thm:weak-limit-shape-beta}, the
expansion above is valid for $|\zeta|$ larger than the support radius.  Both sides are
analytic on the connected component of $\mathbb C\setminus\operatorname{supp}(\nu_\chi)$
containing infinity.  The identity on the large-$\zeta$ region therefore extends to that
component, with the logarithm \(L_\chi\) supplied by
Corollary~\ref{cor:automatic-branch}.  Stieltjes inversion gives
\begin{equation}\label{dsm2}
\frac{1}{\pi}
\lim_{\delta\downarrow0}
\Im\operatorname{St}_{\nu_\chi}(x+\mathbf i\delta)
=
-\frac{d\nu_\chi}{dx}(x)
\end{equation}
at every Lebesgue point \(x\) of the density.  Hence, by
\eqref{slope-from-nu-density},
\[
\partial_\kappa\mathcal H(\chi,\kappa)
=
-\frac{1}{\pi}
\lim_{\delta\downarrow0}
\Im\operatorname{St}_{\nu_\chi}
\bigl(e^{-n\beta\kappa}+\mathbf i\delta\bigr)
\]
whenever \(e^{-n\beta\kappa}\) is a Lebesgue point of the density.

To analyze the upper boundary limit of \eqref{ctx} at a real point
\(x>0\), set
\[
\zeta=x+\mathbf i\delta,
\qquad \delta>0,
\]
and let \(\delta\downarrow0\).  For fixed \(\delta>0\), the points \(w\) at which the logarithm in
\eqref{ctx} becomes singular are precisely the zeros of the quantity
inside the logarithm:
\[
1-\frac{T_\chi(w)}{x+\mathbf i\delta}=0.
\]
Equivalently,
\[
T_\chi(w)=x+\mathbf i\delta.
\]
Letting \(\delta\downarrow0\) leads to the real root equation
\[
T_\chi(w)=x,
\qquad
x=e^{-n\beta\kappa}.
\]
On the positive real branch, this is equivalent to
\begin{equation}\label{eq:natural-S-equivalent}
S_\chi(w)=e^{-n\kappa}.
\end{equation}
The remaining root analysis is conditional on the interlacing and no-escape hypotheses.  All root-count assertions for the infinite product are obtained through the
finite truncations below, followed by the no-escape and Hurwitz passage.

\subsection{Uniqueness of the nonreal root pair}

We first study the truncated equation.

\begin{lemma}\label{l65}
Suppose Assumption~\ref{ap62} holds.  Fix
\(s\in\mathbb R\) and \(K\in\mathbb Z_{>0}\), and assume that
\[
S_{\chi,K}(w)
\not\equiv s^n.
\]
Then the equation
\begin{equation}\label{ceqK}
S_{\chi,K}(w)
=
s^n
\end{equation}
has at most one pair of nonreal conjugate roots, counted with
multiplicity.
\end{lemma}

\begin{proof}
Set
$c:=s^n$. After cancelling common numerator and denominator factors, write
\begin{align}
S_{\chi,K}(w)
=
C_{\chi,K}
\frac{\prod_{z\in Z_K}(w-z)}
     {\prod_{p\in P_K}(w-p)},\label{dck}
\end{align}
where
\[
Z_K:=D_{2,K}\cup D_{4,K},
\qquad
P_K:=D_{1,K}\cup D_{3,K}.
\]
By Assumption~\ref{ap62}, all points of \(Z_K\) and \(P_K\) are
real and simple, and
\[
|Z_K|
=
|P_K|
=
d_K,
\qquad
d_K:=N_K+M_K^-+M_K^+.
\]
Thus \eqref{ceqK} is equivalent to
\begin{equation}
H_{K,c}(w)
:=
C_{\chi,K}\prod_{z\in Z_K}(w-z)
-
c\prod_{p\in P_K}(w-p)
=
0.
\label{eq:cleared-root-polynomial}
\end{equation}
No pole of \(S_{\chi,K}\) is a zero of \(H_{K,c}\), since the
fraction above is reduced.

Partition the pole set into the three ordered blocks
\[
P_K^-
:=
D_{3,K}^-,
\qquad
P_K^0
:=
D_{1,K},
\qquad
P_K^+
:=
D_{3,K}^+,
\]
with corresponding numerator-zero blocks
\[
Z_K^-:=D_{4,K}^-,
\qquad
Z_K^0:=D_{2,K},
\qquad
Z_K^+:=D_{4,K}^+.
\]
Assumption~\ref{ap62}(4) places these blocks, with empty blocks omitted,
in the order
\[
P_K^-\cup Z_K^-
\prec
P_K^0\cup Z_K^0
\prec
P_K^+\cup Z_K^+.
\]

Within the inward block, Assumption~\ref{ap62}(2) gives
\[
p_1<z_1<p_2<z_2<\cdots<p_{N_K}<z_{N_K}.
\]
Within either outward block, Assumption~\ref{ap62}(3) gives
\[
z_1<p_1<z_2<p_2<\cdots<z_{M_K^\sigma}<p_{M_K^\sigma},
\qquad
\sigma\in\{-,+\}.
\]
Consequently, between any two consecutive poles belonging to the same
nonempty block there is exactly one simple numerator zero and no other
pole.  Hence the two one-sided limits of \(S_{\chi,K}\) at these poles
have opposite signs:
\[
\lim_{w\downarrow p_i}S_{\chi,K}(w)
=
-\lim_{w\uparrow p_{i+1}}S_{\chi,K}(w)
\in\{+\infty,-\infty\}.
\]
By continuity, \(S_{\chi,K}\) maps each such interval onto
\(\mathbb R\).  Therefore \eqref{ceqK} has at least one real root in
each interval between consecutive poles of the same block.

Let \(r_K\le3\) be the number of nonempty pole blocks among
\[
P_K^-,
\qquad
P_K^0,
\qquad
P_K^+.
\]
If a block contains \(m\) poles, the preceding argument gives one
distinct real root in each of the \(m-1\) intervals between its
consecutive poles.  Summing over the nonempty blocks, we obtain at least
\[
d_K-r_K
\]
distinct real roots of \(H_{K,c}\).

Since
\[
\deg H_{K,c}\le d_K,
\]
the total number of nonreal roots of \(H_{K,c}\), counted with
multiplicity, is at most
\[
d_K-(d_K-r_K)=r_K\le3.
\]
The polynomial \(H_{K,c}\) has real coefficients, so its nonreal roots
occur in complex-conjugate pairs, with equal multiplicities.  Their total
multiplicity is therefore even and hence at most \(2\).  Thus
\eqref{ceqK} has at most one pair of nonreal conjugate roots.
\end{proof}

The next lemma is the finite-$K$ sign criterion used to determine whether the unmatched
real root lies on the negative axis.

\begin{lemma}[Root-labelled Stieltjes representation for \(\beta=1\)]
\label{lem:root-labelled-stieltjes}
Assume \(\beta=1\) and Assumption~\ref{ap62}.  Let
\(\mathcal C_\chi\) be the one-point contour supplied by
Corollary~\ref{cor:automatic-branch}, and let $S_{\chi,K}(w)$ be defined as in (\ref{dsck}).
Since $\beta=1$,  \(T_\chi=S_\chi\).

For \(|\zeta|\) sufficiently large and each \(\xi\in D_{1,K}\), let
\(w_{\xi,K}(\zeta)\) be the unique root of
\[
S_{\chi,K}(w)=\zeta
\]
near \(\xi\), labelled by
\[
w_{\xi,K}(\zeta)\longrightarrow\xi
\qquad\text{as }\zeta\to\infty.
\]
Choose the logarithms so that
\[
\log\left(1-\frac{S_{\chi,K}(0)}{\zeta}\right)\longrightarrow0,
\qquad
\log w_{\xi,K}(\zeta)-\log\xi\longrightarrow0
\]
as \(\zeta\to\infty\).  Then
\begin{align}
&-\frac{1}{\pi\mathbf i}
\oint_{\mathcal C_\chi}
\log\left(
1-\frac{S_{\chi,K}(w)}{\zeta}
\right)\frac{dw}{w}
\notag\\
&\qquad=
-2\log\left(
1-\frac{S_{\chi,K}(0)}{\zeta}
\right)
+
2\sum_{\xi\in D_{1,K}}
\left(
\log w_{\xi,K}(\zeta)-\log\xi
\right).
\label{eq:finite-root-labelled-stieltjes}
\end{align}
Consequently,
\begin{equation}\label{eq511}
\operatorname{St}_{\nu_\chi}(\zeta)
=
-2\log\left(
1-\frac{S_\chi(0)}{\zeta}
\right)
+
2\lim_{K\to\infty}
\sum_{\xi\in D_{1,K}}
\left(
\log w_{\xi,K}(\zeta)-\log\xi
\right).
\end{equation}
The identities are initially understood for \(|\zeta|\) sufficiently
large.  They remain valid under simultaneous analytic continuation from
\(\zeta=\infty\), with the root labels and logarithms continued along the
same path.
\end{lemma}

\begin{proof}
Fix \(K\), and put
\[
F_{K,\zeta}(w)
:=
1-\frac{S_{\chi,K}(w)}{\zeta}.
\]
For \(|\zeta|\) sufficiently large,
\(F_{K,\zeta}\) is uniformly close to \(1\) on
\(\mathcal C_\chi\).  Hence its winding number around \(0\) along
\(\mathcal C_\chi\) is zero.  Its poles inside
\(\mathcal C_\chi\) are precisely the points of \(D_{1,K}\), so the
argument principle shows that \(F_{K,\zeta}\) has exactly
\(|D_{1,K}|\) zeros inside \(\mathcal C_\chi\), counted with
multiplicity.

Every \(\xi\in D_{1,K}\) is a simple pole of \(S_{\chi,K}\).  Therefore,
for \(|\zeta|\) sufficiently large, the equation
\(S_{\chi,K}(w)=\zeta\) has a unique simple root
\(w_{\xi,K}(\zeta)\) near \(\xi\), and
\[
w_{\xi,K}(\zeta)\longrightarrow\xi.
\]
These roots account for all zeros of \(F_{K,\zeta}\) inside
\(\mathcal C_\chi\).

Choose a small positively oriented circle \(\mathcal C_0\) around \(0\).
For each \(\xi\in D_{1,K}\), choose a positively oriented contour
\(\mathcal C_\xi\) enclosing precisely the pole \(\xi\) and the
corresponding root \(w_{\xi,K}(\zeta)\), but not \(0\) or any other
zero or pole.  Joining each \(w_{\xi,K}(\zeta)\) to \(\xi\) by a cut
inside \(\mathcal C_\xi\), we may deform \(\mathcal C_\chi\) to
\[
\mathcal C_0
\cup
\bigcup_{\xi\in D_{1,K}}\mathcal C_\xi
\]
without crossing a zero, pole, or logarithmic cut of
\(F_{K,\zeta}\).

Let
\[
G_{K,\zeta}(w):=\log F_{K,\zeta}(w).
\]
Since \(G_{K,\zeta}\) is holomorphic near \(0\), Cauchy's formula gives
\[
-\frac{1}{\pi\mathbf i}
\oint_{\mathcal C_0}
G_{K,\zeta}(w)\frac{dw}{w}
=
-2G_{K,\zeta}(0)
=
-2\log\left(
1-\frac{S_{\chi,K}(0)}{\zeta}
\right).
\]

Fix \(\xi\in D_{1,K}\).  Since \(0\) lies outside
\(\mathcal C_\xi\), a single-valued holomorphic logarithm
\(\log w\) may be chosen near \(\mathcal C_\xi\).  Moreover,
\(F_{K,\zeta}\) has one zero and one pole inside
\(\mathcal C_\xi\), so its winding number along
\(\mathcal C_\xi\) is zero and \(G_{K,\zeta}\) is single-valued near
that contour.  Integration by parts gives
\begin{align*}
-\frac{1}{\pi\mathbf i}
\oint_{\mathcal C_\xi}
G_{K,\zeta}(w)\frac{dw}{w}
&=
\frac{1}{\pi\mathbf i}
\oint_{\mathcal C_\xi}
\log w\,G_{K,\zeta}'(w)\,dw.
\end{align*}
Since
\[
G_{K,\zeta}'(w)
=
\frac{S_{\chi,K}'(w)}
     {S_{\chi,K}(w)-\zeta},
\]
this logarithmic derivative has simple poles at
\(w_{\xi,K}(\zeta)\) and \(\xi\), with residues \(1\) and \(-1\),
respectively.  The residue theorem therefore yields
\[
-\frac{1}{\pi\mathbf i}
\oint_{\mathcal C_\xi}
G_{K,\zeta}(w)\frac{dw}{w}
=
2\left(
\log w_{\xi,K}(\zeta)-\log\xi
\right).
\]
Summing over \(\xi\in D_{1,K}\) proves
\eqref{eq:finite-root-labelled-stieltjes}.

Finally, normal convergence of the reflected product gives
\[
S_{\chi,K}\longrightarrow S_\chi
\]
uniformly on \(\mathcal C_\chi\), while the same geometric estimates at
the origin give
\[
S_{\chi,K}(0)\longrightarrow S_\chi(0).
\]
Hence, for \(|\zeta|\) sufficiently large,
\[
-\frac{1}{\pi\mathbf i}
\oint_{\mathcal C_\chi}
\log\left(
1-\frac{S_{\chi,K}(w)}{\zeta}
\right)\frac{dw}{w}
\longrightarrow
-\frac{1}{\pi\mathbf i}
\oint_{\mathcal C_\chi}
\log\left(
1-\frac{S_\chi(w)}{\zeta}
\right)\frac{dw}{w}.
\]
By \eqref{ctx}, the last expression equals
\(\operatorname{St}_{\nu_\chi}(\zeta)\).  Taking \(K\to\infty\) in
\eqref{eq:finite-root-labelled-stieltjes} proves
\eqref{eq511}.  The continuation statement follows
from uniqueness of analytic continuation from \(\zeta=\infty\).
\end{proof}

\begin{lemma}[Finite argument balance]
\label{lem:finite-argument-balance}
Assume \(\beta=1\) and Assumption~\ref{ap62}. Fix \(K\geq1\)
such that \(D_{1,K}\neq\varnothing\), and write
\[
D_{1,K}
=
\{p_{1,K}<\cdots<p_{N_K,K}\}.
\]
Let \(x>0\) satisfy
\[
S_{\chi,K}(0)\neq x,
\]
and suppose that the equation
\[
S_{\chi,K}(w)=x
\]
has a unique nonreal conjugate pair of roots, of total multiplicity two.
Denote this pair by
\[
w_{+,K},\ \overline{w_{+,K}},
\qquad
\operatorname{Im}w_{+,K}>0.
\]

For \(\delta>0\), let
\[
\mathcal W_{1,K}(\delta)
:=
\left\{
w_{\xi,K}(x+\mathbf i\delta):
\xi\in D_{1,K}
\right\}
\]
be the unordered multiset obtained from the pole-labelled roots of
Lemma~\ref{lem:root-labelled-stieltjes}.

Then as \(\delta\downarrow0\),
\[
\mathcal W_{1,K}(\delta)
\longrightarrow
\{w_{+,K},u_{1,K},\ldots,u_{N_K-1,K}\}
\]
as multisets, where, for each \(1\leq j<N_K\),
\(u_{j,K}\) is the unique root of \(S_{\chi,K}(w)=x\) in
\[
(p_{j,K},p_{j+1,K}),
\]
and this root is simple.

The negative real roots satisfy
\[
\#\{1\leq j<N_K:u_{j,K}<0\}
=
\#\bigl(D_{1,K}\cap(-\infty,0)\bigr)
+
\mathbf 1_{\{S_{\chi,K}(0)>x\}}
-1.
\]

For \(t\in\mathbb R\setminus\{0\}\), define the upper boundary
argument by
\[
\operatorname{Arg}_{+}t
:=
\lim_{\delta\downarrow0}
\arg(t+\mathbf i\delta)
=
\pi\mathbf 1_{\{t<0\}},
\]
where \(\arg\in(-\pi,\pi)\) is the principal argument. Then
\[
\mathfrak m_K
:=
\mathbf 1_{\{S_{\chi,K}(0)>x\}}
-\frac{1}{\pi}
\left(
\sum_{j=1}^{N_K-1}\operatorname{Arg}_{+}u_{j,K}
-
\sum_{\xi\in D_{1,K}}\operatorname{Arg}_{+}\xi
\right)
=1,
\]
with the empty sum understood as zero.
\end{lemma}

\begin{proof}
Write
\[
N:=N_K=|D_{1,K}|=|D_{2,K}|,\qquad
P_K:=D_{1,K}\cup D_{3,K},
\qquad
Z_K:=D_{2,K}\cup D_{4,K}.
\]
After cancellation,
\[
S_{\chi,K}(w)
=
C_{\chi,K}
\frac{\displaystyle\prod_{z\in Z_K}(w-z)}
     {\displaystyle\prod_{p\in P_K}(w-p)},
\qquad C_K>0.
\]

\medskip
\noindent
\emph{Identification of the pole-labelled roots.}
For \(p\in P_K\), let
\[
a_p:=\operatorname*{Res}_{w=p}S_{\chi,K}(w).
\]
Since \(C_K>0\),
\[
\operatorname{sgn}a_p
=
(-1)^{
\#(Z_K\cap(p,\infty))
+
\#((P_K\setminus\{p\})\cap(p,\infty))
}.
\]
Each zero--pole block other than the block containing \(p\)
contributes an even number to the exponent. Within the inward
block,
\[
p_{1,K}<z_{1,K}<\cdots<p_{N,K}<z_{N,K},
\]
so the exponent is odd; within either outward block it is even.
Consequently,
\[
a_p<0\quad\text{for }p\in D_{1,K},
\qquad
a_p>0\quad\text{for }p\in D_{3,K}.
\tag{1}
\label{eq:residue-signs-argument-balance}
\]

For \(\zeta=\mathbf iR\), \(R\to\infty\), the root issuing from
the pole \(p\) satisfies
\[
w_p(\mathbf iR)
=
p+\frac{a_p}{\mathbf iR}+O(R^{-2}).
\]
It therefore lies in the upper half-plane when \(p\in D_{1,K}\)
and in the lower half-plane when \(p\in D_{3,K}\).

For every \(\zeta\in\mathbb C_+:=\{w\in \mathbb C:\Im w>0\}\), the cleared equation
\[
C_{\chi,K}\prod_{z\in Z_K}(w-z)
-
\zeta\prod_{p\in P_K}(w-p)
=0
\]
has degree \(|P_K|=|Z_K|\), because
\(C_{\chi,K}-\zeta\neq0\). Moreover, none of its roots can lie on the
real axis, since \(S_{\chi,K}\) is real-valued there away from its poles.
Thus, as \(\zeta\) varies in the connected set \(\mathbb C_+\),
the number of roots in either half-plane is constant. It follows
from the preceding large-\(R\) analysis that
\[
\#\{w\in\mathbb C_+:S_{\chi,K}(w)=\zeta\}=N,
\]
counting multiplicity, and that these roots are precisely those
continued from the poles in \(D_{1,K}\). In particular,
\[
\mathcal W_{1,K}(\delta)
=
\{w\in\mathbb C_+:S_{\chi,K}(w)=x+\mathbf i\delta\},
\qquad \delta>0,
\tag{2}
\label{eq:upper-root-identification}
\]
as multisets.

\medskip
\noindent
\emph{The upper boundary roots.}
Put
\[
F_{K}(w):=S_{\chi,K}(w)-x.
\]
We first record an oriented real-root count. Let
\(J=(a,b)\subset\mathbb R\setminus P_K\) satisfy
\[
F_K(a+)<0<F_K(b-).
\]
For sufficiently small \(\delta>0\), let \(n_+(J)\), respectively
\(n_-(J)\), denote the number of roots of
\[
F_K(w)=\mathbf i\delta
\]
which converge, as \(\delta\downarrow0\), to zeros of \(F_K\) in
\(J\) from the upper, respectively lower, half-plane.

Let \(y_1<\cdots<y_\ell\) be the distinct zeros of \(F_K\) in
\(J\), and let \(\sigma_j\in\{-1,1\}\) be the sign of \(F_K\) on
the component immediately to the right of \(y_j\), with
\(\sigma_0\) denoting its sign immediately to the right of \(a\).
Then
\[
\sigma_0=-1,
\qquad
\sigma_\ell=1.
\]
If \(y_j\) has multiplicity \(m_j\), write
\[
F_K(w)
=
c_j(w-y_j)^{m_j}
+
O\bigl((w-y_j)^{m_j+1}\bigr),
\qquad c_j\in\mathbb R^*=\mathbb R\setminus\{0\}.
\]
After the rescaling \(w-y_j=\delta^{1/m_j}v\), Rouché's theorem
shows that the local roots are governed by
\[
c_jv^{m_j}=\mathbf i.
\]
Hence the difference between the numbers approaching \(y_j\) from
the upper and lower half-planes is
\begin{align}\label{dpm}
n_+(y_j)-n_-(y_j)
=
\begin{cases}
0&\mathrm{If}\ m_j\ \mathrm{is\ even}\\
\mathrm{sgn}(c_j)&\mathrm{If}\ m_j\ \mathrm{is\ odd}
\end{cases}
\end{align}
The right hand side of (\ref{dpm}) is exactly $\frac{\sigma_j-\sigma_{j-1}}{2}$; hence we have
\[
n_+(y_j)-n_-(y_j)
=
\frac{\sigma_j-\sigma_{j-1}}{2}.
\]
Summing over \(j\) gives the telescoping identity
\[
n_+(J)-n_-(J)
=
\frac12\sum_{j=1}^{\ell}
(\sigma_j-\sigma_{j-1})
=
\frac{\sigma_\ell-\sigma_0}{2}
=1.
\tag{3}
\label{eq:oriented-root-count}
\]

For \(1\leq j<N\), set
\[
J_j:=(p_{j,K},p_{j+1,K}).
\]
The separation in Assumption~\ref{ap62} ensures that \(J_j\)
contains no pole other than its endpoints. By
\eqref{eq:residue-signs-argument-balance},
\[
\lim_{w\downarrow p_{j,K}}F_K(w)=-\infty,
\qquad
\lim_{w\uparrow p_{j+1,K}}F_K(w)=+\infty.
\]
Therefore
\[
n_+(J_j)-n_-(J_j)=1,
\qquad 1\leq j<N,
\]
and in particular \(n_+(J_j)\geq1\).

The root \(w_{+,K}\) is simple, since the unique nonreal conjugate
pair has total multiplicity two. Hence, for all sufficiently small
\(\delta>0\), exactly one root of
\[
S_{\chi,K}(w)=x+\mathbf i\delta
\]
lies near \(w_{+,K}\), and this root lies in the upper half-plane.
Together with the roots approaching the intervals \(J_j\), we have
therefore accounted for at least
\[
1+\sum_{j=1}^{N-1}n_+(J_j)\geq N
\]
upper-half-plane roots. By
\eqref{eq:upper-root-identification}, there are exactly \(N\) such
roots. Consequently,
\[
n_+(J_j)=1,
\qquad
n_-(J_j)=0,
\qquad
1\leq j<N.
\tag{4}
\label{eq:interval-root-count}
\]

The total multiplicity of the real zeros of \(F_K\) in \(J_j\)
equals \(n_+(J_j)+n_-(J_j)\). Thus
\eqref{eq:interval-root-count} implies that \(J_j\) contains a
unique zero \(u_{j,K}\), and that this zero is simple. It also
shows that no other upper-half-plane root can have a finite
boundary limit. Hence
\[
\mathcal W_{1,K}(\delta)
\longrightarrow
\{w_{+,K},u_{1,K},\ldots,u_{N-1,K}\}
\]
as \(\delta\downarrow0\), in the sense of multisets.

\medskip
\noindent
\emph{The negative-root count.}
Set
\[
r_K:=\#\{j:p_{j,K}<0\},
\qquad
q_K:=\#\{j:u_{j,K}<0\}.
\]
We claim that
\[
q_K
=
r_K+\mathbf 1_{\{S_K(0)>x\}}-1.
\tag{5}
\label{eq:negative-root-balance}
\]

Suppose first that \(r_K=0\). Then every interval \(J_j\) is
contained in \((0,\infty)\), so \(q_K=0\). If \(F_K(0)<0\), recall also that
\[
\lim_{w\uparrow p_{1,K}}F_K(w)=+\infty.
\]
Applying \eqref{eq:oriented-root-count} to
\((0,p_{1,K})\) would produce an additional upper-boundary root,
contradicting the exhaustion of the \(N\) roots established above.
Thus \(F_K(0)>0\), and
\[
q_K=0=r_K+1-1.
\]

Suppose next that \(r_K=N\). Then all \(N-1\) interval roots are
negative, and hence \(q_K=N-1\). If \(F_K(0)>0\), recall also that 
\[
\lim_{w\downarrow p_{N,K}}F_K(w)=-\infty.
\]
Applying \eqref{eq:oriented-root-count} to
\((p_{N,K},0)\) would again give an additional upper-boundary
root. Therefore \(F_K(0)<0\), and
\[
q_K=N-1=r_K-1.
\]

Finally, assume \(0<r_K<N\). The intervals \(J_j\) with
\(j<r_K\) contribute \(r_K-1\) negative roots, while those with
\(j>r_K\) contribute only positive roots. The remaining root
\(u_{r_K,K}\) lies in
\[
(p_{r_K,K},p_{r_K+1,K}),
\]
which contains \(0\). Since this root is unique and
\[
F_K(p_{r_K,K}+)<0,
\qquad
F_K(p_{r_K+1,K}-)>0,
\]
it lies to the left of \(0\) precisely when \(F_K(0)>0\).
Consequently,
\[
q_K
=
r_K-1+\mathbf 1_{\{F_K(0)>0\}},
\]
which proves \eqref{eq:negative-root-balance} in all cases.

Since \(F_K(0)>0\) is equivalent to \(S_{\chi,K}(0)>x\), and since
\[
\sum_{j=1}^{N-1}\operatorname{Arg}_+u_{j,K}
=
\pi q_K,
\qquad
\sum_{\xi\in D_{1,K}}\operatorname{Arg}_+\xi
=
\pi r_K,
\]
we obtain
\[
\begin{aligned}
\mathfrak m_K
&=
\mathbf 1_{\{S_{\chi,K}(0)>x\}}
-\frac1\pi
\left(
\sum_{j=1}^{N-1}\operatorname{Arg}_+u_{j,K}
-
\sum_{\xi\in D_{1,K}}\operatorname{Arg}_+\xi
\right)\\
&=
\mathbf 1_{\{S_{\chi,K}(0)>x\}}-q_K+r_K\\
&=1,
\end{aligned}
\]
as claimed.
\end{proof}

We can now express the slope in terms of the argument of the upper-half-plane root.

\begin{proposition}[Truncated-root slope formula]
\label{p67}
Assume \(\beta=1\) and Assumption~\ref{ap62}. Let \(H\) be the
limit shape of Theorem~\ref{thm:weak-limit-shape-beta}, let
\((\chi,\kappa)\) be a regular slope point, and set
\[
x:=e^{-n\kappa}.
\]
Suppose that
\[
S_\chi(0)\neq x,
\]
and that, for all sufficiently large \(K\),
\[
D_{1,K}\neq\varnothing,
\qquad
C_{\chi,K}\neq x.
\]

If, for all sufficiently large \(K\), the equation
\[
S_{\chi,K}(w)=x
\]
has a nonreal conjugate pair, denote its upper-half-plane member
by \(w_{+,K}\). Then this pair is unique and simple, and
\begin{equation}
\frac{\partial H(\chi,\kappa)}{\partial\kappa}
=
2-\frac{2}{\pi}
\lim_{K\to\infty}\arg w_{+,K}.
\label{eq:conditional-truncated-slope}
\end{equation}

If, in addition, \((\chi,\kappa)\) is liquid, then the truncated
equation has such a pair for all sufficiently large \(K\), and
\[
\arg w_{+,K}
\longrightarrow
\pi\left(
1-\frac12
\frac{\partial H(\chi,\kappa)}{\partial\kappa}
\right)
\in(0,\pi).
\]
\end{proposition}

\begin{proof}
Write
\[
h:=
\frac{\partial H(\chi,\kappa)}{\partial\kappa},
\qquad
I:=
\mathbf 1_{\{S_\chi(0)>x\}}.
\]
For \(\xi\in D_{1,K}\), let \(w_{\xi,K}(\zeta)\), together with
the corresponding logarithm, be continued from \(\zeta=\infty\)
as in Lemma~\ref{lem:root-labelled-stieltjes}, and set
\[
A_K
:=
\operatorname{Im}
\sum_{\xi\in D_{1,K}}
\left(
\log w_{\xi,K}(x+\mathbf i0)-\log\xi
\right).
\]
Taking the upper boundary value in \eqref{eq511} and using
\eqref{dsm2} together with \eqref{slope-from-nu-density} gives
\begin{equation}
h
=
2I-\frac{2}{\pi}\lim_{K\to\infty}A_K.
\label{eq:conditional-slope-before-balance}
\end{equation}

Moreover, the convergence at the origin established in the proof
of Lemma~\ref{lem:root-labelled-stieltjes} gives
\[
S_{\chi,K}(0)\longrightarrow S_\chi(0).
\]
Since \(S_\chi(0)\neq x\), for all sufficiently large \(K\),
\begin{equation}
I_K
:=
\mathbf 1_{\{S_{\chi,K}(0)>x\}}
=
I,
\qquad
S_{\chi,K}(0)\neq x.
\label{eq:conditional-indicator-stability}
\end{equation}

Assume first that, for all sufficiently large \(K\), the equation
\[
S_{\chi,K}(w)=x
\]
has a nonreal conjugate pair. By Lemma~\ref{l65}, this pair is
unique and has total multiplicity two; hence both of its members
are simple. Write it as
\[
w_{+,K},\ \overline{w_{+,K}},
\qquad
\operatorname{Im}w_{+,K}>0,
\]
and put
\[
\theta_K:=\arg w_{+,K}\in(0,\pi).
\]

For all sufficiently large \(K\),
Lemma~\ref{lem:finite-argument-balance} gives
\[
A_K
=
\theta_K
+
\sum_{j=1}^{N_K-1}\operatorname{Arg}_+u_{j,K}
-
\sum_{\xi\in D_{1,K}}\operatorname{Arg}_+\xi
\]
and
\[
I_K
-\frac1\pi
\left(
\sum_{j=1}^{N_K-1}\operatorname{Arg}_+u_{j,K}
-
\sum_{\xi\in D_{1,K}}\operatorname{Arg}_+\xi
\right)
=1.
\]
Combining these identities with
\eqref{eq:conditional-indicator-stability}, we obtain
\[
\begin{aligned}
2I-\frac{2}{\pi}A_K
&=
2I_K-\frac{2}{\pi}\theta_K
-\frac{2}{\pi}
\left(
\sum_{j=1}^{N_K-1}\operatorname{Arg}_+u_{j,K}
-
\sum_{\xi\in D_{1,K}}\operatorname{Arg}_+\xi
\right)\\
&=
2-\frac{2}{\pi}\theta_K.
\end{aligned}
\]
Since the left-hand side converges to \(h\) by
\eqref{eq:conditional-slope-before-balance}, it follows that
\[
\theta_K
\longrightarrow
\pi\left(1-\frac{h}{2}\right)
\in[0,\pi].
\]
Equivalently,
\[
\frac{\partial H(\chi,\kappa)}{\partial\kappa}
=
2-\frac{2}{\pi}
\lim_{K\to\infty}\arg w_{+,K}.
\]

It remains to prove the final assertion. Suppose that
\((\chi,\kappa)\) is liquid, so that
\[
h\in(0,2).
\]
If \(S_{\chi,K}(w)=x\) had no nonreal roots along an infinite
subsequence, then, because \(C_{\chi,K}\neq x\), no root would lie at
infinity, and all pole-labelled upper boundary roots would be real
along that subsequence. Consequently,
\[
A_K\in\pi\mathbb Z
\]
and hence
\[
2I-\frac{2}{\pi}A_K\in2\mathbb Z
\]
along the same subsequence. This contradicts
\eqref{eq:conditional-slope-before-balance}, because its
left-hand side converges to \(h\in(0,2)\).

Thus the truncated equation has a nonreal conjugate pair for all
sufficiently large \(K\). Applying the first part of the proof
gives
\[
\arg w_{+,K}
\longrightarrow
\pi\left(
1-\frac12
\frac{\partial H(\chi,\kappa)}{\partial\kappa}
\right)
\in(0,\pi),
\]
as claimed.
\end{proof}

We now pass from the truncated equation to the infinite product.

\begin{equation}\label{ceqnl}
\mathcal{G}_\chi(w)\prod_{k\ge1}\mathcal{F}_{u,v,k}(w)=s^n.
\end{equation}

\begin{lemma}[Local uniform compactness of truncated spectral roots]
\label{lem:local-spectral-root-compactness}
Assume Assumptions~\ref{ap64}--\ref{ap65}. Fix
\[
\chi_0\in
(V_0,V_m)\setminus\{V_0,\ldots,V_m\},
\qquad
x_0>0.
\]
Suppose that, for some \(K_*\geq1\),
\begin{equation}
S_{\chi_0,K}(0)\neq x_0,
\qquad
C_{\chi_0,K}\neq x_0,
\qquad
K\geq K_*.
\label{eq:local-endpoint-nonexceptional}
\end{equation}

Then there exist open intervals
\[
I\ni\chi_0,
\qquad
X\ni x_0,
\qquad
X\Subset(0,\infty),
\]
constants
\[
0<r<R<\infty,
\]
and \(K_0\geq K_*\) such that, for every
\[
\chi\in I,\qquad x\in X,\qquad K\geq K_0,
\]
every solution \(w\in\mathbb C_+\) of
\[
S_{\chi,K}(w)=x
\]
satisfies
\begin{equation}
r\leq |w|\leq R.
\label{eq:local-uniform-radial-bound}
\end{equation}
Here
\[
\mathbb C_+
:=
\{w\in\mathbb C:\operatorname{Im}w>0\}.
\]
\end{lemma}

\begin{proof}
Set
\[
q:=(uv)^2\in(0,1).
\]
Choose compact intervals \(I_0\) and \(X_0\) such that, for some
\(p\in[m]\),
\[
\chi_0\in\operatorname{int}(I_0)
\Subset(V_{p-1},V_p),
\qquad
x_0\in\operatorname{int}(X_0)
\Subset(0,\infty).
\]
We shall shrink \(I_0\) and \(X_0\) finitely many times, and at
the end set
\[
I:=\operatorname{int}(I_0),
\qquad
X:=\operatorname{int}(X_0).
\]

By Lemma~\ref{l64}, Assumptions~\ref{ap64}--\ref{ap65} imply
Assumption~\ref{ap62}. Since all inequalities in
Assumptions~\ref{ap64}--\ref{ap65} are strict, after shrinking
\(I_0\) the active zero--pole labels, the cancellation pattern in
the reduced rational function \(S_{\chi,K}\), and the relative
order of the reduced zeros and poles are independent of
\(\chi\in I_0\). The unreflected zero--pole locations depend
continuously on \(\chi\), whereas the reflected zero--pole
locations are independent of \(\chi\) and are obtained from
finitely many base locations by the fixed dilations
\[
q^k,\qquad v^2q^{k-1},\qquad
q^{-k},\qquad u^{-2}q^{-(k-1)}.
\]
All constants below may therefore be chosen uniformly for
\(\chi\in I_0\).

\medskip
\noindent
\emph{Endpoint separation.}
Evaluation at \(0\) and at infinity is unchanged under positive
dilation of the argument. Hence there exist constants
\(a_0,a_\infty>0\), independent of \(K\), and positive continuous
functions \(c_0,c_\infty\) on \(I_0\) such that
\begin{equation}
S_{\chi,K}(0)=c_0(\chi)a_0^K,
\qquad
C_{\chi,K}=c_\infty(\chi)a_\infty^K.
\label{eq:local-endpoint-geometric-form}
\end{equation}
Condition~\eqref{eq:local-endpoint-nonexceptional}, continuity,
and the elementary alternatives for a positive geometric sequence
imply, after shrinking \(I_0\) and \(X_0\), that there exist
\(\delta>0\) and \(K_0\geq K_*\) such that
\begin{equation}
\left|
\frac{x}{S_{\chi,K}(0)}-1
\right|
\geq\delta,
\qquad
\left|
\frac{x}{C_{\chi,K}}-1
\right|
\geq\delta
\label{eq:local-endpoint-uniform-separation}
\end{equation}
for every
\[
\chi\in I_0,\qquad x\in X_0,\qquad K\geq K_0.
\]
Indeed, if \(a_0\neq1\), then \(c_0(\chi)a_0^K\) converges
uniformly on \(I_0\) either to \(0\) or to infinity; if \(a_0=1\),
the first inequality follows from
\[
c_0(\chi_0)\neq x_0
\]
and continuity. The argument at infinity is identical.

\medskip
\noindent
\emph{Interval families and endpoint scales.}
For fixed \(\chi\in I_0\) and \(K\geq1\), cancel all common
numerator and denominator factors of \(S_{\chi,K}\). Using the
ordered pole and zero lists of Assumption~\ref{ap62}(2)--(3),
and suppressing their dependence on \(\chi\), define
\[
\mathcal I_{\chi,K}
:=
\bigl\{
[p_{j,K},z_{j,K}]
:
1\leq j\leq N_K
\bigr\},
\]
and
\[
\mathcal O_{\chi,K}
:=
\bigcup_{\sigma\in\{-,+\}}
\bigl\{
[z_{j,K}^{\sigma},p_{j,K}^{\sigma}]
:
1\leq j\leq M_K^\sigma
\bigr\}.
\]
By Assumption~\ref{ap62}(2)--(3), the intervals in each family
are pairwise disjoint, and the reduced factorization is
\begin{equation}
S_{\chi,K}(w)
=
C_{\chi,K}
\prod_{[a,b]\in\mathcal I_{\chi,K}}
\frac{w-b}{w-a}
\prod_{[a,b]\in\mathcal O_{\chi,K}}
\frac{w-a}{w-b}.
\label{eq:local-interval-factorization}
\end{equation}
If both interval families are empty, then
\(S_{\chi,K}\equiv C_{\chi,K}\), and
\eqref{eq:local-endpoint-uniform-separation} excludes a solution.
We henceforth consider the nonconstant case.

The explicit reflected zero--pole locations imply that, after
shrinking \(I_0\) if necessary, there exist finite interval
families
\[
\mathcal I_{\chi}^{(0)},
\qquad
\mathcal O_{\chi}^{(0)},
\]
depending continuously on \(\chi\in I_0\), and finite interval
families
\[
\mathcal I^{(1)},\quad
\mathcal I^{(2)},
\qquad
\mathcal O^{(1)},\quad
\mathcal O^{(2)},
\]
independent of \(\chi\), such that
\begin{equation}
\mathcal I_{\chi,K}
=
\mathcal I_{\chi}^{(0)}
\sqcup
\bigcup_{k=1}^{K}
\left(
q^k\mathcal I^{(1)}
\sqcup
v^2q^{k-1}\mathcal I^{(2)}
\right),
\label{eq:local-inward-interval-decomposition}
\end{equation}
and
\begin{equation}
\mathcal O_{\chi,K}
=
\mathcal O_{\chi}^{(0)}
\sqcup
\bigcup_{k=1}^{K}
\left(
q^{-k}\mathcal O^{(1)}
\sqcup
u^{-2}q^{-(k-1)}\mathcal O^{(2)}
\right).
\label{eq:local-outward-interval-decomposition}
\end{equation}
Here, for \(a>0\) and for an interval family \(\mathcal A\),
\[
a\mathcal A:=\{aL:L\in\mathcal A\}.
\]
Each of the six interval families above is finite and consists of
nondegenerate real intervals; any of them may be empty, in which
case it contributes no intervals to the corresponding disjoint
union, product, or sum.

Define the reflected inward and reflected outward interval
families by
\[
\widehat{\mathcal I}^{\mathrm{ref}}
:=
\mathcal I^{(1)}
\cup
\frac{v^2}{q}\mathcal I^{(2)},
\qquad
\widehat{\mathcal O}^{\mathrm{ref}}
:=
\mathcal O^{(1)}
\cup
u^{-2}q\,\mathcal O^{(2)}.
\]
Then the reflected inward intervals are precisely
\[
\bigcup_{k=1}^{K}
q^k\widehat{\mathcal I}^{\mathrm{ref}},
\]
and the reflected outward intervals are precisely
\[
\bigcup_{k=1}^{K}
q^{-k}\widehat{\mathcal O}^{\mathrm{ref}}.
\]

For an interval family \(\mathcal A\), write
\[
\partial\mathcal A
:=
\{a,b:[a,b]\in\mathcal A\}
\]
for the set of its endpoints. Let \(\mathcal E_{\chi,K}\) be the
set of reduced zeros and poles of \(S_{\chi,K}\), and put
\[
\underline\rho_{\chi,K}
:=
\min_{\lambda\in\mathcal E_{\chi,K}}|\lambda|,
\qquad
\overline\rho_{\chi,K}
:=
\max_{\lambda\in\mathcal E_{\chi,K}}|\lambda|.
\]

If
\[
\widehat{\mathcal I}^{\mathrm{ref}}\neq\varnothing,
\]
define
\[
\underline\rho^{\mathrm{ref}}
:=
\min_{\lambda\in\partial
\widehat{\mathcal I}^{\mathrm{ref}}}
|\lambda|>0,
\qquad
\underline\rho_K^{\mathrm{ref}}
:=
q^K\underline\rho^{\mathrm{ref}}.
\]
Assumption~\ref{ap65} places the reflected inward families between
the unreflected inward families and \(0\), and orders their
successive copies. Hence, after increasing \(K_0\) if necessary,
\begin{equation}
\underline\rho_{\chi,K}
=
\underline\rho_K^{\mathrm{ref}},
\qquad
\chi\in I_0,\quad K\geq K_0.
\label{eq:local-min-reflected}
\end{equation}
In particular, in this case the minimum modulus is independent of
\(\chi\). If
\[
\widehat{\mathcal I}^{\mathrm{ref}}=\varnothing,
\]
then the endpoints closest to \(0\) come only from the unreflected
families; after shrinking \(I_0\), there exists
\(\underline\rho_0>0\) such that
\begin{equation}
\underline\rho_{\chi,K}\geq \underline\rho_0,
\qquad
\chi\in I_0,\quad K\geq K_0.
\label{eq:local-min-unreflected}
\end{equation}

Similarly, if
\[
\widehat{\mathcal O}^{\mathrm{ref}}\neq\varnothing,
\]
define
\[
\overline\rho^{\mathrm{ref}}
:=
\max_{\lambda\in\partial
\widehat{\mathcal O}^{\mathrm{ref}}}
|\lambda|<\infty,
\qquad
\overline\rho_K^{\mathrm{ref}}
:=
q^{-K}\overline\rho^{\mathrm{ref}}.
\]
Assumption~\ref{ap65} places the reflected outward families beyond
the unreflected outward families, and orders their successive
copies. Hence, after increasing \(K_0\) if necessary,
\begin{equation}
\overline\rho_{\chi,K}
=
\overline\rho_K^{\mathrm{ref}},
\qquad
\chi\in I_0,\quad K\geq K_0.
\label{eq:local-max-reflected}
\end{equation}
In particular, in this case the maximum modulus is independent of
\(\chi\). If
\[
\widehat{\mathcal O}^{\mathrm{ref}}=\varnothing,
\]
then the endpoints farthest from \(0\) come only from the
unreflected families; after shrinking \(I_0\), there exists
\(\overline\rho_0<\infty\) such that
\begin{equation}
\overline\rho_{\chi,K}\leq \overline\rho_0,
\qquad
\chi\in I_0,\quad K\geq K_0.
\label{eq:local-max-unreflected}
\end{equation}

For a real interval \(L=[a,b]\), \(a<b\), write
\[
|L|:=b-a
\]
for its Euclidean length. For
\(\lambda\in\mathcal E_{\chi,K}\), the notation \(|\lambda|\)
denotes the usual absolute value.

The interval decompositions
\eqref{eq:local-inward-interval-decomposition} and
\eqref{eq:local-outward-interval-decomposition} give constants
\(A,B<\infty\), independent of \(\chi\in I_0\) and \(K\), such
that
\begin{equation}
\sum_{L\in\mathcal I_{\chi,K}}|L|\leq A,
\qquad
\sum_{L\in\mathcal O_{\chi,K}}
\int_L\frac{|dt|}{t^2}\leq A,
\label{eq:local-geometric-tail-bounds}
\end{equation}
and
\begin{equation}
\sum_{\lambda\in\mathcal E_{\chi,K}}
\frac1{|\lambda|}
\leq
\frac{B}{\underline\rho_{\chi,K}},
\qquad
\sum_{\lambda\in\mathcal E_{\chi,K}}
|\lambda|
\leq
B\overline\rho_{\chi,K}.
\label{eq:local-endpoint-sums}
\end{equation}
Indeed, under a dilation \(L\mapsto aL\),
\[
|aL|=a|L|,
\qquad
\int_{aL}\frac{|dt|}{t^2}
=
a^{-1}\int_L\frac{|dt|}{t^2},
\]
and the series generated by
\eqref{eq:local-inward-interval-decomposition}--
\eqref{eq:local-outward-interval-decomposition} are geometric.

For all sufficiently small \(\varepsilon>0\),
\eqref{eq:local-endpoint-sums} gives
\begin{align}
\sup_{\substack{\chi\in I_0\\
|w|\leq\varepsilon\underline\rho_{\chi,K}}}
\left|
\log
\frac{S_{\chi,K}(w)}{S_{\chi,K}(0)}
\right|
&\leq C\varepsilon,
\label{eq:local-origin-approximation}\\
\sup_{\substack{\chi\in I_0\\
|w|\geq\varepsilon^{-1}\overline\rho_{\chi,K}}}
\left|
\log
\frac{S_{\chi,K}(w)}{C_{\chi,K}}
\right|
&\leq C\varepsilon.
\label{eq:local-infinity-approximation}
\end{align}
Near \(0\), this follows by summing
\[
\log
\left[
\frac{(w-z)/(w-p)}{z/p}
\right]
=
\log(1-w/z)-\log(1-w/p),
\]
while near infinity it follows by summing
\[
\log\frac{w-z}{w-p}
=
\log(1-z/w)-\log(1-p/w).
\]

Choose \(\varepsilon>0\) so small that
\[
e^{C\varepsilon}-1<\delta.
\]
Then
\eqref{eq:local-endpoint-uniform-separation}--
\eqref{eq:local-infinity-approximation} imply
\begin{equation}
S_{\chi,K}(w)\neq x
\quad\text{if}\quad
|w|\leq
\varepsilon\underline\rho_{\chi,K}
\quad\text{or}\quad
|w|\geq
\varepsilon^{-1}\overline\rho_{\chi,K},
\label{eq:local-extreme-regions}
\end{equation}
for every
\[
\chi\in I_0,\qquad x\in X_0,\qquad K\geq K_0.
\]

\medskip
\noindent
\emph{Phase balance.}
Let \(w\in\mathbb C_+\), and write
\[
\rho:=|w|,
\qquad
y:=\operatorname{Im}w>0.
\]
For a real interval \(L=[a,b]\), \(a<b\), set
\[
\omega_L(w)
:=
\int_L
\frac{y}{|w-t|^2}\,dt.
\]
Equivalently,
\[
\omega_L(w)
=
\arg\frac{w-b}{w-a},
\]
where \(\arg\) denotes the branch with values in
\((-\pi,\pi)\). The ratio belongs to \(\mathbb C_+\), so this
value lies in \((0,\pi)\).

Define
\[
\Phi_{\chi,K}(w)
:=
\sum_{L\in\mathcal I_{\chi,K}}\omega_L(w),
\qquad
\Psi_{\chi,K}(w)
:=
\sum_{L\in\mathcal O_{\chi,K}}\omega_L(w).
\]
Since the intervals in each family are pairwise disjoint,
\[
0\leq\Phi_{\chi,K}(w)<\pi,
\qquad
0\leq\Psi_{\chi,K}(w)<\pi.
\]
Equation~\eqref{eq:local-interval-factorization} gives
\[
\arg S_{\chi,K}(w)
\equiv
\Phi_{\chi,K}(w)-\Psi_{\chi,K}(w)
\pmod{2\pi}.
\]
Consequently, every upper-half-plane solution of
\[
S_{\chi,K}(w)=x>0
\]
satisfies the exact identity
\begin{equation}
\Phi_{\chi,K}(w)=\Psi_{\chi,K}(w).
\label{eq:local-phase-balance}
\end{equation}

\medskip
\noindent
\emph{Lower radial bound.}
The outward intervals are uniformly separated from \(0\). Hence,
for \(\rho\) sufficiently small,
\eqref{eq:local-geometric-tail-bounds} gives
\begin{equation}
\Psi_{\chi,K}(w)
\leq
4y
\sum_{L\in\mathcal O_{\chi,K}}
\int_L\frac{|dt|}{t^2}
\leq A_0y,
\label{eq:local-small-outward-bound}
\end{equation}
where \(A_0\) is independent of \(\chi\) and \(K\).

Suppose first that
\[
\widehat{\mathcal I}^{\mathrm{ref}}\neq\varnothing.
\]
Since the reflected inward intervals are the dilates
\[
q^kL,
\qquad
L\in\widehat{\mathcal I}^{\mathrm{ref}},
\qquad
1\leq k\leq K,
\]
and since \(\widehat{\mathcal I}^{\mathrm{ref}}\) is a finite
family of nondegenerate intervals separated from \(0\), there
exist constants
\[
\rho_0>0,
\qquad
c_1,c_2,c_3>0,
\]
independent of \(\chi\in I_0\) and \(K\), such that, whenever
\[
\varepsilon\underline\rho_{\chi,K}\leq \rho\leq\rho_0,
\]
one can find \(L\in\mathcal I_{\chi,K}\) satisfying
\begin{equation}
c_1\rho
\leq
\min_{t\in L}|t|
\leq
\max_{t\in L}|t|
\leq
c_2\rho,
\qquad
|L|\geq c_3\rho.
\label{eq:local-comparable-inward-interval}
\end{equation}
Indeed, by \eqref{eq:local-min-reflected}, the lower endpoint
\(\varepsilon\underline\rho_{\chi,K}\) is a fixed multiple of the
smallest reflected inward scale. Choosing the reflected level whose
scale is immediately below \(\rho\), and using the finiteness,
nondegeneracy, and separation from \(0\) of
\(\widehat{\mathcal I}^{\mathrm{ref}}\), gives the constants
uniformly.

It follows that
\begin{equation}
\Phi_{\chi,K}(w)
\geq
\int_L\frac{y}{|w-t|^2}\,dt
\geq
\frac{c_3}{(1+c_2)^2}\frac{y}{\rho}
=:b_0\frac{y}{\rho}.
\label{eq:local-small-inward-lower-bound}
\end{equation}
For a root, combining
\eqref{eq:local-phase-balance},
\eqref{eq:local-small-outward-bound}, and
\eqref{eq:local-small-inward-lower-bound} yields
\[
b_0\frac{y}{\rho}\leq A_0y.
\]
Since \(y>0\),
\[
\rho\geq\frac{b_0}{A_0}.
\]
Together with \eqref{eq:local-extreme-regions}, this excludes a
fixed neighborhood of \(0\).

If instead
\[
\widehat{\mathcal I}^{\mathrm{ref}}=\varnothing,
\]
then \eqref{eq:local-min-unreflected} and
\eqref{eq:local-extreme-regions} directly exclude a fixed
neighborhood of \(0\).

Thus there exists \(r_*>0\) such that every upper-half-plane
solution satisfies
\begin{equation}
|w|\geq r_*.
\label{eq:local-root-lower-bound}
\end{equation}

\medskip
\noindent
\emph{Upper radial bound.}
The inward intervals are uniformly bounded, and their total
length is uniformly bounded by
\eqref{eq:local-geometric-tail-bounds}. Hence, for \(\rho\)
sufficiently large,
\begin{equation}
\Phi_{\chi,K}(w)
\leq
A_\infty\frac{y}{\rho^2},
\label{eq:local-large-inward-bound}
\end{equation}
where \(A_\infty\) is independent of \(\chi\) and \(K\).

Suppose next that
\[
\widehat{\mathcal O}^{\mathrm{ref}}\neq\varnothing.
\]
Since the reflected outward intervals are the dilates
\[
q^{-k}L,
\qquad
L\in\widehat{\mathcal O}^{\mathrm{ref}},
\qquad
1\leq k\leq K,
\]
and since \(\widehat{\mathcal O}^{\mathrm{ref}}\) is a finite
family of nondegenerate intervals separated from \(0\), there
exist constants
\[
R_0>0,
\qquad
c_1',c_2',c_3'>0,
\]
independent of \(\chi\in I_0\) and \(K\), such that, whenever
\[
R_0\leq \rho\leq
\varepsilon^{-1}\overline\rho_{\chi,K},
\]
one can find \(L\in\mathcal O_{\chi,K}\) satisfying
\[
c_1'\rho
\leq
\min_{t\in L}|t|
\leq
\max_{t\in L}|t|
\leq
c_2'\rho,
\qquad
|L|\geq c_3'\rho.
\]
Indeed, by \eqref{eq:local-max-reflected}, the upper endpoint
\(\varepsilon^{-1}\overline\rho_{\chi,K}\) is a fixed multiple of
the largest reflected outward scale. Choosing the reflected level
whose scale is immediately above \(\rho\), and using the finiteness,
nondegeneracy, and separation from \(0\) of
\(\widehat{\mathcal O}^{\mathrm{ref}}\), gives the constants
uniformly.

Consequently,
\begin{equation}
\Psi_{\chi,K}(w)
\geq
\frac{c_3'}{(1+c_2')^2}\frac{y}{\rho}
=:b_\infty\frac{y}{\rho}.
\label{eq:local-large-outward-lower-bound}
\end{equation}
For a root,
\eqref{eq:local-phase-balance},
\eqref{eq:local-large-inward-bound}, and
\eqref{eq:local-large-outward-lower-bound} imply
\[
b_\infty\frac{y}{\rho}
\leq
A_\infty\frac{y}{\rho^2}.
\]
Cancelling \(y>0\), we get
\[
\rho\leq\frac{A_\infty}{b_\infty}.
\]
Together with \eqref{eq:local-extreme-regions}, this excludes the
complement of a fixed disk.

If instead
\[
\widehat{\mathcal O}^{\mathrm{ref}}=\varnothing,
\]
then \eqref{eq:local-max-unreflected} and
\eqref{eq:local-extreme-regions} directly exclude the complement
of a fixed disk.

Thus there exists \(R_*<\infty\) such that every upper-half-plane
solution satisfies
\begin{equation}
|w|\leq R_*.
\label{eq:local-root-upper-bound}
\end{equation}

Taking
\[
r:=r_*,
\qquad
R:=R_*,
\]
and setting
\[
I:=\operatorname{int}(I_0),
\qquad
X:=\operatorname{int}(X_0),
\]
we obtain
\[
r\leq |w|\leq R
\]
for every
\[
\chi\in I,\qquad x\in X,\qquad K\geq K_0
\]
and every \(w\in\mathbb C_+\) satisfying
\[
S_{\chi,K}(w)=x.
\]
This proves \eqref{eq:local-uniform-radial-bound}.
\end{proof}

\begin{lemma}[Infinite-product root criterion]
\label{l610}
Assume \(\beta=1\) and Assumptions~\ref{ap64}--\ref{ap65}.
Let \(H\) be the limit shape supplied by
Theorem~\ref{thm:weak-limit-shape-beta}, let
\((\chi,\kappa)\) be a regular slope point, and set
\[
x:=e^{-n\kappa}.
\]
Suppose that
\[
S_\chi(0)\neq x
\]
and that, for all sufficiently large \(K\),
\[
D_{1,K}\neq\varnothing,
\qquad
C_{\chi,K}\neq x.
\]

Then the equation
\[
S_\chi(w)=x
\]
has at most one root in
\[
\mathbb C_+
:=
\{w\in\mathbb C:\operatorname{Im}w>0\},
\]
counted with multiplicity. In particular, any such root is simple.

Moreover, the following are equivalent:
\begin{enumerate}
\item \((\chi,\kappa)\) is liquid;
\item the equation \(S_\chi(w)=x\) has a root
      \(w_+\in\mathbb C_+\).
\end{enumerate}

When these conditions hold, for all sufficiently large \(K\), the
truncated equation
\[
S_{\chi,K}(w)=x
\]
has a unique nonreal conjugate pair
\[
w_{+,K},\ \overline{w_{+,K}},
\qquad
\operatorname{Im}w_{+,K}>0,
\]
both roots are simple, and
\begin{equation}
w_{+,K}\longrightarrow w_+.
\label{eq:finite-to-infinite-upper-root}
\end{equation}
Furthermore,
\begin{equation}
\frac{\partial H(\chi,\kappa)}{\partial\kappa}
=
2-\frac{2}{\pi}\arg w_+,
\qquad
\arg w_+\in(0,\pi).
\label{eq:infinite-root-slope}
\end{equation}
Consequently, a regular slope point is frozen if and only if
\(S_\chi(w)=x\) has no nonreal roots.
\end{lemma}

\begin{proof}
Put
\[
F_K(w):=S_{\chi,K}(w)-x,
\qquad
F(w):=S_\chi(w)-x.
\]
By normal convergence,
\[
F_K\longrightarrow F
\]
locally uniformly on \(\mathbb C\setminus\mathbb R\).

If \(F\) had at least two zeros in \(\mathbb C_+\), counted with
multiplicity, Hurwitz's theorem applied in disjoint disks around
these zeros would imply that \(F_K\) has at least two zeros in
\(\mathbb C_+\) for all sufficiently large \(K\), contradicting
Lemma~\ref{l65}. Hence \(F\) has at most one zero in
\(\mathbb C_+\), counted with multiplicity; in particular, such a
zero is simple.

Suppose first that
\[
F(w_+)=0
\qquad\text{for some }w_+\in\mathbb C_+.
\]
Hurwitz's theorem gives, for all sufficiently large \(K\), a
unique zero \(w_{+,K}\in\mathbb C_+\) of \(F_K\) near \(w_+\), and
\[
w_{+,K}\longrightarrow w_+.
\]
Its conjugate is also a root, and Lemma~\ref{l65} shows that this
is the unique nonreal conjugate pair. Proposition~\ref{p67}
therefore yields
\[
\frac{\partial H(\chi,\kappa)}{\partial\kappa}
=
2-\frac{2}{\pi}\arg w_+.
\]
Since \(0<\arg w_+<\pi\), the slope belongs to \((0,2)\), so
\((\chi,\kappa)\) is liquid.

Conversely, suppose that \((\chi,\kappa)\) is liquid.
Proposition~\ref{p67} gives a unique simple conjugate pair
\(w_{+,K},\overline{w_{+,K}}\) for all sufficiently large \(K\),
with
\[
\arg w_{+,K}\longrightarrow
\theta:=
\pi\left(
1-\frac12
\frac{\partial H(\chi,\kappa)}{\partial\kappa}
\right)
\in(0,\pi).
\]
Choose
\[
0<\eta<
\frac12\min\{\theta,\pi-\theta\}.
\]
Then
\[
\eta\leq\arg w_{+,K}\leq\pi-\eta
\]
for all sufficiently large \(K\).

Apply
Lemma~\ref{lem:local-spectral-root-compactness}
with
\[
(\chi_0,x_0)=(\chi,x).
\]
There exist \(0<r<R<\infty\) such that
\[
r\leq|w_{+,K}|\leq R
\]
for all sufficiently large \(K\). Hence
\[
w_{+,K}\in
Q:=
\left\{
w\in\mathbb C:
r\leq|w|\leq R,\quad
\eta\leq\arg w\leq\pi-\eta
\right\}
\Subset\mathbb C_+.
\]

Every subsequence therefore has a further subsequence converging
to some \(w_*\in\mathbb C_+\). Local uniform convergence gives
\[
S_\chi(w_*)=x.
\]
By uniqueness, every such limit is the same root \(w_+\).
Consequently,
\[
w_{+,K}\longrightarrow w_+.
\]

Thus a regular slope point is liquid if and only if
\(S_\chi(w)=x\) has a root in \(\mathbb C_+\), and in that case
\[
\frac{\partial H(\chi,\kappa)}{\partial\kappa}
=
2-\frac{2}{\pi}\arg w_+.
\]
The frozen assertion follows by taking the complementary regular
slope regime.
\end{proof}

\begin{proof}[Proof of Theorem \ref{p412}]
For \(p=(\chi,\kappa)\), set
\[
F_p(w):=S_\chi(w)-e^{-n\kappa}.
\]

Since \(p_0\) lies on the frozen boundary, there exist sequences
of liquid and frozen regular slope points
\[
p_j^{\mathrm L}
=
(\chi_j^{\mathrm L},\kappa_j^{\mathrm L})
\longrightarrow p_0,
\qquad
p_j^{\mathrm F}
=
(\chi_j^{\mathrm F},\kappa_j^{\mathrm F})
\longrightarrow p_0.
\]
After discarding finitely many terms, both sequences lie in the
neighborhoods appearing in the hypotheses and in
Lemma~\ref{lem:local-spectral-root-compactness}.

By Lemma~\ref{l610}, for every \(j\) the equation
\[
F_{p_j^{\mathrm L}}(w)=0
\]
has a unique root
\[
w_j\in\mathbb C_+.
\]
If \(w_{j,K}\) denotes the upper-half-plane root of the
corresponding truncated equation, the same lemma gives
\[
w_{j,K}\longrightarrow w_j
\qquad\text{as }K\to\infty.
\]

Let \(0<r<R<\infty\) be the constants supplied by
Lemma~\ref{lem:local-spectral-root-compactness}. For every fixed
\(j\), that lemma gives
\[
r\leq |w_{j,K}|\leq R
\]
for all sufficiently large \(K\). Passing to the limit
\(K\to\infty\), we obtain
\begin{equation}
r\leq |w_j|\leq R,
\qquad j\geq1.
\label{eq:boundary-liquid-root-radial-bound}
\end{equation}
Hence, after passing to a subsequence,
\[
w_j\longrightarrow w_0
\]
for some
\[
w_0\in\mathbb C,
\qquad
r\leq |w_0|\leq R.
\]
In particular,
\[
w_0\neq0.
\]

We next show that \(w_0\) is not a pole of \(S_{\chi_0}\).
Only finitely many pole families of \(S_\chi\) meet the annulus
\[
A:=
\left\{
w\in\mathbb C:
\frac r2\leq |w|\leq2R
\right\},
\]
because the reflected poles accumulate only at \(0\) and at
infinity. After shrinking the parameter neighborhood $U_0$, these poles
may be written as finitely many continuous branches
\[
p_1(\chi),\ldots,p_M(\chi),
\]
with constant multiplicities.

If \(w_0=p_\ell(\chi_0)\) for some \(\ell\), then in a
neighborhood of this pole one has
\[
S_\chi(w)
=
\frac{H_\ell(\chi,w)}
     {w-p_\ell(\chi)},
\]
where \(H_\ell\) is jointly continuous, holomorphic in \(w\), and
bounded away from zero. Since
\[
\chi_j^{\mathrm L}\longrightarrow\chi_0,
\qquad
w_j\longrightarrow w_0,
\]
we would have
\[
|w_j-p_\ell(\chi_j^{\mathrm L})|
\longrightarrow0
\]
and hence
\[
|S_{\chi_j^{\mathrm L}}(w_j)|
\longrightarrow\infty.
\]
This contradicts
\[
S_{\chi_j^{\mathrm L}}(w_j)
=
e^{-n\kappa_j^{\mathrm L}}
\longrightarrow e^{-n\kappa_0}.
\]
Therefore \(w_0\) is not a pole of \(S_{\chi_0}\).

Choose a disk \(B_0\) centered at \(w_0\) which contains no pole
of \(S_\chi\) for \(\chi\) sufficiently close to \(\chi_0\).
On this disk the reflected product converges normally, and
\(G_\chi(w)\) depends jointly continuously on \((\chi,w)\).
Consequently,
\[
(\chi,w)\longmapsto S_\chi(w)
\]
is jointly continuous near \((\chi_0,w_0)\). It follows that
\[
\begin{aligned}
S_{\chi_0}(w_0)
&=
\lim_{j\to\infty}
S_{\chi_j^{\mathrm L}}(w_j)\\
&=
\lim_{j\to\infty}
e^{-n\kappa_j^{\mathrm L}}
=
e^{-n\kappa_0}.
\end{aligned}
\tag{2.20}
\label{eq:frozen-boundary-root}
\]

We now prove that \(w_0\) is real. Since \(w_j\in\mathbb C_+\),
\[
\operatorname{Im}w_0\geq0.
\]
Suppose that \(\operatorname{Im}w_0>0\). Choose a closed disk
\[
B\Subset\mathbb C_+
\]
centered at \(w_0\), containing no pole of $S_{\chi}$ when $\chi$ is close to $\chi_0$, such that
\[
F_{p_0}(w)\neq0,
\qquad w\in\partial B.
\]
The joint continuity established above gives
\[
F_p\longrightarrow F_{p_0}
\]
uniformly on \(\partial B\) as \(p\to p_0\). Rouché's theorem
therefore implies that \(F_p\) has a zero in \(B\) for every
\(p\) sufficiently close to \(p_0\).

In particular, \(F_{p_j^{\mathrm F}}\) has a zero in
\(\mathbb C_+\) for all sufficiently large \(j\). By
Lemma~\ref{l610}, this would imply that \(p_j^{\mathrm F}\) is
liquid, contradicting its choice. Hence
\[
w_0\in\mathbb R\setminus\{0\}.
\]

It remains to prove that \(w_0\) is a multiple root. Suppose, to
the contrary, that
\[
S_{\chi_0}'(w_0)\neq0.
\]
Then \(w_0\) is a simple zero of \(F_{p_0}\). Choose a
conjugation-invariant closed disk \(B\), centered at \(w_0\),
such that \(w_0\) is the only zero of \(F_{p_0}\) in \(B\), the
disk contains no pole of $S_{\chi}$ when $\chi$ is close to $\chi_0$, and
\[
F_{p_0}(w)\neq0,
\qquad w\in\partial B.
\]
Rouché's theorem implies that, for all sufficiently large \(j\),
the function \(F_{p_j^{\mathrm L}}\) has exactly one zero in
\(B\), counted with multiplicity.

On the other hand, the spectral product has real zero--pole data,
and hence
\[
S_\chi(\overline w)
=
\overline{S_\chi(w)}.
\]
Thus both
\[
w_j
\qquad\text{and}\qquad
\overline{w_j}
\]
are zeros of \(F_{p_j^{\mathrm L}}\). Since
\[
w_j\longrightarrow w_0
\]
and \(B\) is conjugation invariant, both zeros lie in \(B\) for
all sufficiently large \(j\). They are distinct because
\(w_j\in\mathbb C_+\). This contradicts the preceding zero count.

Therefore
\[
S_{\chi_0}'(w_0)=0.
\]
Together with \eqref{eq:frozen-boundary-root}, this gives
\[
S_{\chi_0}(w_0)=e^{-n\kappa_0},
\qquad
\frac{S_{\chi_0}'(w_0)}
     {S_{\chi_0}(w_0)}
=0,
\]
because \(S_{\chi_0}(w_0)=e^{-n\kappa_0}>0\).
This proves \eqref{fb}.
\end{proof}


\section{A sampler for finite truncations}
\label{sect:sampler}

This section constructs an exact sampler for the finite reflected truncations of the doubly
free-boundary rail-yard measure.  The sampler assigns partitions
\[
   \lambda^{(i,j)}\in\mathbb Y,
   \qquad (i,j)\in\mathbb Z^2,
\]
\rev{to the vertices of a finite lattice region.
At the end of the construction, we restore the original order of the
column labels and update the partitions using the formulas given below.
We output the resulting ordered sequence of partitions.
This sequence represents a dimer configuration on the original rail-yard graph.}

The construction is best read as a folded and reflected version of the growth-diagram
sampler for Schur processes.  In the SchurSample algorithm of
\cite{bbbccv14}, the deterministic input is an
interlacing word, together with a list of parameters.  The word is not the random sequence
of partitions.  Rather, it prescribes the horizontal, vertical, and dual interlacing relations
between consecutive partitions and determines the Young-diagram-shaped growth domain.
Empty partitions are placed on the coordinate axes, each box is filled from its three
southwest neighbours by one local Cauchy or dual-Cauchy bijection, and the sampled
Schur-process sequence
\[
   \Lambda=(\lambda^{(0)},\lambda^{(1)},\ldots,\lambda^{(T)})
\]
is read from the outer boundary of the filled diagram.  The proof of exactness is an
induction over boxes, and the boxes may be filled in any order compatible with their
partial order.

In the present rail-yard setting, the role of the Schur-process interlacing word is played by
the fixed data \((a_i,b_i)_{i=l}^r\), equivalently by the string of interlacing relations in
Lemma~\ref{lm24}.  The new feature is that the two free endpoints repeatedly reflect this
word.  At truncation level \(K\), only the first \(K\) reflected layers are kept, so there is
again a finite growth diagram.  A local stochastic square samples one Cauchy or dual-Cauchy
factor.  A local deterministic square only commutes two already present coordinate lines
and contributes no normalization.  The one-body updates at the folded free boundary are
diagonal updates: a line meets its reflected copy, so the two incoming partitions in the
local rule coincide.  

For each fixed truncation level, the output law is exactly the corresponding truncated
free-boundary measure.  Lemmas~\ref{lem:terminal-boundary-weight-convergence} and
\ref{lem:terminal-mass-convergence} identify the limit of these truncated laws, and
Theorem~\ref{thm:truncated-sampler} gives total-variation convergence to the full doubly
free-boundary law.
\label{rev:comment15}

\subsection{Elementary partition recursions}

We first record the four local recursions used by the sampler. We write \(\operatorname{Geom}(\xi)\) for the random variable with
\[
   \mathbb P(\operatorname{Geom}(\xi)=g)=(1-\xi)\xi^g,
   \qquad g\in\mathbb Z_{\ge0},
\]
and \(\operatorname{Bern}(p)\) for the Bernoulli random variable with success probability
\(p\).

\begin{definition}[The \(HH\) recursion]
\label{def:HH}
Let \(\lambda,\mu,\kappa\in\mathbb Y\) satisfy
\[
   \kappa\prec\lambda,
   \qquad
   \kappa\prec\mu .
\]
For \(\xi\in(0,1)\), define \(HH_\xi(\lambda,\mu,\kappa)=\nu\) as follows.  Sample
\[
   G\sim\operatorname{Geom}(\xi),
\]
set
\[
   \nu_1=\max\{\lambda_1,\mu_1\}+G,
\]
and, for \(i\ge2\), set
\[
   \nu_i=\max\{\lambda_i,\mu_i\}+\min\{\lambda_{i-1},\mu_{i-1}\}-\kappa_{i-1}.
\]
Then
\[
   \lambda\prec\nu,
   \qquad
   \mu\prec\nu,
   \qquad
   |\nu|+|\kappa|=|\lambda|+|\mu|+G .
\]
For fixed \((\lambda,\mu)\), the map \((\kappa,G)\mapsto\nu\) is bijective.
\end{definition}

\begin{definition}[The \(HV\) recursion]
\label{def:HV}
Let \(\lambda,\mu,\kappa\in\mathbb Y\) satisfy
\[
   \kappa\prec'\lambda,
   \qquad
   \kappa\prec\mu .
\]
For \(\xi>0\), define \(HV_\xi(\lambda,\mu,\kappa)=\nu\) as follows.  Sample
\[
   B_0\sim \operatorname{Bern}\!\left(\frac{\xi}{1+\xi}\right).
\]
Use the conventions \(\lambda_0=\mu_0=+\infty\) and \(\lambda_i=\mu_i=0\) for all
sufficiently large \(i\).  Set \(B=B_0\).  For
\[
   1\le i\le \max\{\ell(\lambda),\ell(\mu)\}+1,
\]
define
\[
   \nu_i=
   \begin{cases}
   \max\{\lambda_i,\mu_i\}+B,
      & \lambda_i\le\mu_i<\lambda_{i-1},\\
   \max\{\lambda_i,\mu_i\},
      & \text{otherwise}.
   \end{cases}
\]
After assigning \(\nu_i\), if
\[
   \mu_{i+1}<\lambda_i\le\mu_i,
\]
update
\[
   B:=\min\{\lambda_i,\mu_i\}-\kappa_i .
\]
Then
\[
   \mu\prec'\nu,
   \qquad
   \lambda\prec\nu,
   \qquad
   |\nu|+|\kappa|=|\lambda|+|\mu|+B_0 .
\]
For fixed \((\lambda,\mu)\), the map \((\kappa,B_0)\mapsto\nu\) is bijective.
\end{definition}

\begin{definition}[The \(AA\) recursion]
\label{def:AA}
Let \(\lambda,\mu,\kappa\in\mathbb Y\) satisfy
\[
   \mu\prec\kappa\prec\lambda .
\]
Define \(AA(\lambda,\mu,\kappa)=\nu\) by
\[
   \nu_i=\min\{\lambda_i,\mu_{i-1}\}+\max\{\lambda_{i+1},\mu_i\}-\kappa_i,
   \qquad i\ge1,
\]
with \(\mu_0=+\infty\).  Then
\[
   \mu\prec\nu\prec\lambda,
   \qquad
   |\nu|+|\kappa|=|\lambda|+|\mu| .
\]
For fixed \((\lambda,\mu)\), the map \(\kappa\mapsto\nu\) is bijective.
\end{definition}

\begin{revision}
\begin{definition}[The \(AB\) recursion]
\label{def:AB}
Let \(\lambda,\mu,\kappa\in\mathbb Y\) satisfy
\[
   \mu\prec'\kappa\prec\lambda.
\]
Set \(\mu_0=+\infty\), and extend all partitions by zero parts. Put
\[
 L_i:=\max\{\lambda_{i+1},\mu_i\},
 \qquad U_i:=\min\{\lambda_i,\mu_{i-1}\},
 \qquad i\ge1.
\]
Define the finite sets of free input and output rows by
\[
 \mathcal R:=\{i\ge1:\lambda_{i+1}\le\mu_i<\lambda_i\},
 \qquad
 \mathcal S:=\{i\ge1:\mu_i<\lambda_i\le\mu_{i-1}\}.
\]
These sets have equal cardinality. Write
\[
 \mathcal R=\{r_1<\cdots<r_m\},\qquad
 \mathcal S=\{s_1<\cdots<s_m\},
 \qquad
 \delta_j:=\kappa_{r_j}-L_{r_j}\in\{0,1\}.
\]
Define \(AB(\lambda,\mu,\kappa)=\nu\) by transferring the free bits
in increasing row order:
\[
 \nu_i=
 \begin{cases}
 U_i-\delta_j,&i=s_j\text{ for some }j\in[m],\\
 U_i,&i\notin\mathcal S.
 \end{cases}
\]
If \(m=0\), this means \(\nu_i=U_i\) for every \(i\). Then
\[
 \mu\prec\nu\prec'\lambda,
 \qquad |\nu|+|\kappa|=|\lambda|+|\mu|.
\]
For fixed \((\lambda,\mu)\), this is a bijection from the admissible
\(\kappa\)'s to the admissible \(\nu\)'s.
\end{definition}
\end{revision}

\revheading{Verification of the local assertions.}\label{head:local-assertions}\label{rev:comment16}
For \(HH_\xi\), the assumptions \(\kappa\prec\lambda\) and \(\kappa\prec\mu\) imply
\[
   \max\{\lambda_i,\mu_i\}\le \kappa_{i-1}
   \le \min\{\lambda_{i-1},\mu_{i-1}\},
   \qquad i\ge2 .
\]
Hence \(\lambda_i,\mu_i\le\nu_i\le\lambda_{i-1},\mu_{i-1}\), proving
\(\lambda\prec\nu\) and \(\mu\prec\nu\).  The size identity follows by telescoping:
\[
\begin{aligned}
 |\nu|
 &=G+\sum_{i\ge1}\max\{\lambda_i,\mu_i\}
      +\sum_{i\ge1}\min\{\lambda_i,\mu_i\}-|\kappa|  \\
 &=G+|\lambda|+|\mu|-|\kappa| .
\end{aligned}
\]
The inverse is explicit:
\[
   G=\nu_1-\max\{\lambda_1,\mu_1\},
   \qquad
   \kappa_i=\max\{\lambda_{i+1},\mu_{i+1}\}
      +\min\{\lambda_i,\mu_i\}-\nu_{i+1} .
\]
For \(AA\), put
\[
   A_i=\min\{\lambda_i,\mu_{i-1}\},
   \qquad
   B_i=\max\{\lambda_{i+1},\mu_i\}.
\]
The hypotheses give \(B_i\le\kappa_i\le A_i\), hence \(B_i\le\nu_i\le A_i\).  This is
exactly \(\mu\prec\nu\prec\lambda\).  Moreover
\[
   \sum_{i\ge1}A_i+\sum_{i\ge1}B_i=|\lambda|+|\mu|,
\]
because \(A_1=\lambda_1\) and, for \(i\ge2\),
\(A_i+B_{i-1}=\lambda_i+\mu_{i-1}\).  Thus
\(|\nu|+|\kappa|=|\lambda|+|\mu|\), and the inverse is
\[
   \kappa_i=A_i+B_i-\nu_i .
\]
\begin{revision}
For \(HV\), the interlacing relations, size identity, and bijectivity
follow from the dual-Cauchy bijection in
\cite[Section~3.1]{bbbccv14}.

For \(AB\), the input and output interlacing conditions are respectively
\[
 L_i\le\kappa_i\le\min\{\lambda_i,\mu_i+1\},
 \qquad
 \max\{\mu_i,\lambda_i-1\}\le\nu_i\le U_i.
\]
The output intervals are nonempty: the input conditions imply
\(\mu_i\le\lambda_i\) and
\(\lambda_i\le\mu_{i-1}+1\) for \(i\ge2\), while
\(\mu_0=+\infty\).
Choices in different rows are independent, since these bounds ensure
\[
 \kappa_i\ge\lambda_{i+1}\ge\kappa_{i+1},
 \qquad
 \nu_i\ge\mu_i\ge\nu_{i+1}.
\]
Each row has either one allowed value or two consecutive allowed values.
The rows with two choices are precisely \(\mathcal R\) for the input
and \(\mathcal S\) for the output.

These two sets have the same size. Indeed, shifting the summation
index gives
\[
 |\mathcal R|-|\mathcal S|
 =
 \sum_{i\ge1}
 \left(
 \mathbf1_{\{\mu_{i-1}<\lambda_i\}}
 -
 \mathbf1_{\{\mu_i<\lambda_{i+1}\}}
 \right)
 =0.
\]
The rule pairs the input row \(r_j\) with the output row \(s_j\).
It is invertible: recover
\(\delta_j=U_{s_j}-\nu_{s_j}\), then set
\(\kappa_{r_j}=L_{r_j}+\delta_j\) and
\(\kappa_i=L_i\) for \(i\notin\mathcal R\).
Thus it is a bijection between the admissible inputs and outputs.

Finally, each \(\delta_j\) is added once to \(\kappa\) and subtracted
once from \(\nu\). Hence
\[
\begin{aligned}
 |\kappa|+|\nu|
 &=\sum_{i\ge1}(L_i+U_i)\\
 &=\lambda_1+\sum_{i\ge1}(\lambda_{i+1}+\mu_i)\\
 &=|\lambda|+|\mu|,
\end{aligned}
\]
where the second equality uses
\(U_1=\lambda_1\) and
\(L_i+U_{i+1}=\lambda_{i+1}+\mu_i\).
\end{revision}

For any rule \(R\in\{HH_\xi,HV_\xi,AA,AB\}\), define its conjugate rule by
\[
   R^\vee(\lambda,\mu,\kappa)=\bigl(R(\lambda',\mu',\kappa')\bigr)' .
\]
For \(a,b\in\{L,R\}\), set
\[
S_{a,b,\xi}:=
\begin{cases}
HH_\xi,       & a=L,\ b=L,\\
HH^\vee_\xi, & a=R,\ b=R,\\
HV_\xi,       & a=L,\ b=R,\\
HV^\vee_\xi, & a=R,\ b=L,
\end{cases}
\qquad
C_{a,b}:=
\begin{cases}
AA,       & a=L,\ b=L,\\
AA^\vee, & a=R,\ b=R,\\
AB,       & a=L,\ b=R,\\
AB^\vee, & a=R,\ b=L.
\end{cases}
\]
The symbol \(S\) denotes a stochastic update, while \(C\) denotes a deterministic
commutation.

Write \(\alpha\prec_L\beta\) for \(\alpha\prec\beta\), and
\(\alpha\prec_R\beta\) for \(\alpha\prec'\beta\).
The rules \(HH\) and \(HV\), together with their conjugates, are
the local sampling rules: \(HH\) uses one geometric random variable
and realizes an ordinary Cauchy factor, while \(HV\) uses one
Bernoulli random variable and realizes a dual Cauchy factor.
The rules \(AA\) and \(AB\), together with their conjugates,
exchange the order of two adjacent interlacing steps by updating
the intermediate partition while keeping the endpoint partitions fixed.
These updates are deterministic and require no additional random variables.

\begin{revision}
For a stochastic update, the three input partitions and a sampled
random variable determine the output \(\nu\). The labels are arranged as
\[
\begin{matrix}
\lambda&\nu\\
\kappa&\mu
\end{matrix},
\qquad \nu=S_{a,b,\xi}(\lambda,\mu,\kappa).
\]
\end{revision}

\subsection{Coordinate domain and initialization}

\begin{revision}
Let
\[
 I_\pm:=\{i\in[l..r]:b_i=\pm\},
 \qquad N_\pm:=|I_\pm|,
 \qquad N:=N_-+N_+.
\]
List the column labels in increasing order:
\[
 I_\pm=\{i^\pm_1<\cdots<i^\pm_{N_\pm}\},
 \qquad
 \psi_\pm(k):=i^\pm_k,\quad k\in[N_\pm].
\]
We use \([0]=\varnothing\) and skip empty index ranges.
Write \(\lambda^{(i,j)}\) for the partition stored at the lattice
vertex \((i,j)\).

\revheading{Initialization.}
Fix \(K\ge1\). Sample a partition \(\Lambda^{(K)}\) with law
\begin{equation}
\label{eq:LambdaK}
 \mathbb P\{\Lambda^{(K)}=\lambda\}
 =
 \left[\prod_{h=1}^K\bigl(1-(uv)^h\bigr)\right]
 (uv)^{|\lambda|}\mathbf1_{\{\lambda_1\le K\}}.
\end{equation}
To sample this law, independently draw the number of parts of size
\(k\) from \(\operatorname{Geom}((uv)^k)\), for \(1\le k\le K\).

Assign this same partition to the initial vertices:
\[
\begin{aligned}
 \lambda^{(i,-NK)}&=\Lambda^{(K)},
 & i&\in[-N_+-NK..-NK],\\
 \lambda^{(-NK,j)}&=\Lambda^{(K)},
 & j&\in[-N_--NK..-NK].
\end{aligned}
\]

\revheading{Output.}
Define
\[
 D=(z_0,z_1,\ldots,z_N),\qquad z_0=(-N_+,0),
\]
by
\[
 z_m-z_{m-1}=
 \begin{cases}
 (1,0),&b_{l+m-1}=+,\\
 (0,-1),&b_{l+m-1}=-,
 \end{cases}
 \qquad 1\le m\le N.
\]
Thus \(z_N=(0,-N_-)\).
Assign the final ordered partitions to these vertices and output
\[
 (\lambda^{z_0},\lambda^{z_1},\ldots,\lambda^{z_N}).
\]

\revheading{Sampling parameters.}
The stochastic updates in
Section~\ref{sec:sampler-coordinate-recursion} use the following
parameters, defined for \(h\ge1\):
\begin{equation}
\label{eq:coordinate-parameters}
\begin{aligned}
\xi^{-\mathrm{in}}_{h}(i)
 &:=u^{2h}v^{2h-1}x_i,
&
\xi^{-\mathrm{out}}_{h}(i)
 &:=u^{2h-1}v^{2h-2}x_i,
&&i\in I_-,\\
\xi^{+\mathrm{in}}_{h}(i)
 &:=u^{2h-1}v^{2h}x_i,
&
\xi^{+\mathrm{out}}_{h}(i)
 &:=u^{2h-2}v^{2h-1}x_i,
&&i\in I_+,\\
\xi^{--}_{h,0}(i,j)
 &:=u^{4h-2}v^{4h-4}x_ix_j,
&
\xi^{--}_{h,1}(i,j)
 &:=u^{4h}v^{4h-2}x_ix_j,
&&i,j\in I_-,\\
\xi^{++}_{h,0}(i,j)
 &:=u^{4h-4}v^{4h-2}x_ix_j,
&
\xi^{++}_{h,1}(i,j)
 &:=u^{4h-2}v^{4h}x_ix_j,
&&i,j\in I_+,\\
\xi^{+-}_{h,0}(i,j)
 &:=u^{4h-4}v^{4h-4}x_ix_j,
&
\xi^{+-}_{h,1}(i,j)
 &:=u^{4h-2}v^{4h-2}x_ix_j,
&&i\in I_+,\ j\in I_-.
\end{aligned}
\end{equation}
Under Assumption~\ref{ass:admissible}, every parameter used in
\(HH\) or \(HH^\vee\) lies in \((0,1)\), and every local weight \(\xi\) used in \(HV\) or \(HV^\vee\) is positive,
and the corresponding Bernoulli probability
\(\xi/(1+\xi)\) lies in \((0,1)\).
\end{revision}

\FloatBarrier
\rev{Figure~\ref{fig:sampler-domain} shows the initial partition
locations, their locations after the first update, and the output
locations in a two-column example.}

\begin{figure}[!htbp]
\centering
\begingroup\color{revisioncolor}
\begin{minipage}[c]{.52\textwidth}
\centering
\begin{tikzpicture}[x=1.22cm,y=1.22cm,draw=revisioncolor,text=revisioncolor,
 every node/.style={font=\footnotesize},line cap=round,line join=round]
\draw[->,thin] (-3.55,-3.5)--(.55,-3.5) node[right] {$i$};
\draw[->,thin] (-3.5,-3.55)--(-3.5,.55) node[above] {$j$};
\foreach \x in {-3,-2,-1,0}{
 \draw (\x,-3.55)--(\x,-3.45);
 \node[below] at (\x,-3.55) {$\x$};
}
\foreach \y in {-3,-2,-1,0}{
 \draw (-3.55,\y)--(-3.45,\y);
 \node[left] at (-3.55,\y) {$\y$};
}
\draw[thin,draw=revisioncolor!35] (-2,-2) rectangle (-1,-1);
\draw[thin,draw=revisioncolor!35] (-1,-1) rectangle (0,0);
\draw[dashed,line width=.8pt] (-3,-2)--(-2,-2)--(-2,-3);
\draw[line width=.9pt] (-2,-1)--(-1,-1)--(-1,-2);
\draw[line width=1.8pt] (-1,0)--(0,0)--(0,-1);
\foreach \p in {(-3,-2),(-2,-2),(-2,-3),(-2,-1),(-1,-1),(-1,-2),(-1,0),(0,0),(0,-1)}
 \fill[revisioncolor] \p circle (1.5pt);
\draw[->,dotted,line width=.8pt] (-1.84,-1.84)--(-1.16,-1.16);
\draw[->,dotted,line width=.8pt] (-.84,-.84)--(-.16,-.16);
\node[above left] at (-1,0) {$z_0$};
\node[above right] at (0,0) {$z_1$};
\node[right] at (0,-1) {$z_2$};
\node[below left] at (-2,-2) {$F_{-2}$};
\node[above left] at (-1,-1) {$F_{-1}$};
\node[right] at (.08,-.45) {$D=F_0$};
\end{tikzpicture}
\end{minipage}\hfill
\begin{minipage}[c]{.43\textwidth}
\small
\(K=1,\ N_+=N_-=1,\ N=2\),\par
\((a_1,b_1)=(L,+)\),\par
\((a_2,b_2)=(L,-)\).
\par\medskip
\noindent\textbf{Initial partitions:}\par
All three vertices of \(F_{-2}\) carry the same sampled partition
\(\Lambda^{(1)}\).
\par\medskip
\noindent\textbf{After the first update:}\par
The updated partitions are placed at the three vertices of \(F_{-1}\).
\par\medskip
\noindent\textbf{Output partitions:}\par
Read the three final partitions at
\(z_0=(-1,0),\ z_1=(0,0),\)\par
\(z_2=(0,-1)\), in this order.
\end{minipage}
\captionsetup{labelfont={color=revisioncolor},textfont={color=revisioncolor}}
\caption{Partition locations in a two-column example with \(K=1\),
\(b_1=+\), and \(b_2=-\). Each marked vertex carries a partition.
The dashed path shows the initial locations, the thin solid path shows the
locations after the first update, and the thick path \(D\) shows where the
final partitions are read. Here \(F_s\) denotes the path
\((s-1,s)\to(s,s)\to(s,s-1)\), and \(D=F_0\) in this example.
The dotted arrows indicate the order of the displayed paths.
The output is \((\lambda^{z_0},\lambda^{z_1},\lambda^{z_2})\).
The formulas for updating the partitions are given in
Section~\ref{sec:sampler-coordinate-recursion}.}
\label{fig:sampler-domain}
\endgroup
\end{figure}
\FloatBarrier

\subsection{The coordinate recursion}\label{sec:sampler-coordinate-recursion}

\leavevmode\par
\begin{revision}
\revheading{Update order.}
For \(h=K,K-1,\ldots,1\), process \(q=0,1,\ldots,N-1\)
in this order and set
\[
 s=-Nh+q+1.
\]
Let \(F_s\) be the path
\[
 (s-N_+,s)\longrightarrow(s,s)\longrightarrow(s,s-N_-).
\]
The initialized vertices form \(F_{-NK}\).

Use the minus updates when \(q<N_-\).
For \(q\ge N_-\), use the plus updates when \(h>1\), and
the final plus updates when \(h=1\).
Except for the final plus updates, each step updates the
partitions on \(F_{s-1}\) to those on \(F_s\).

Process each integer interval \([A..B]\) in increasing order,
and explicitly listed indices in the displayed order.
Each \(S\)-update uses a fresh random variable, independent of
the initialization and all previous random variables.
Store its output immediately.
During each group of \(C\)-updates in the minus or \(h>1\) plus
stages, retain the old \(\lambda\)-values and compute the temporary values
\(\widetilde\lambda\) in the stated order.
Write these temporary values back to \(\lambda\)
after completing the group.
\end{revision}

\rev{Figure~\ref{fig:sampler-recursion} shows the three local updates.}

\begin{figure}[tbp]
\centering
\begingroup\color{revisioncolor}
\begin{minipage}[t]{.46\textwidth}
\centering
\textbf{(a) An ordinary stochastic square}\par\medskip
\begin{tikzpicture}[x=2.3cm,y=1.7cm,draw=revisioncolor,text=revisioncolor,
 every node/.style={font=\small},line cap=round]
\draw[thin] (0,0)--(1,0)--(1,1)--(0,1)--cycle;
\node[below left] at (0,0) {$\kappa$};
\node[above left] at (0,1) {$\lambda$};
\node[below right] at (1,0) {$\mu$};
\node[above right] at (1,1) {$\nu$};
\draw[->,line width=.8pt] (.14,1)--(.85,1);
\draw[->,line width=.8pt] (1,.14)--(1,.85);
\node at (.5,.5) {$S_{a,b,\xi}$};
\foreach \p in {(0,0),(1,0),(0,1),(1,1)}\fill[revisioncolor] \p circle (1.5pt);
\end{tikzpicture}
\par\smallskip
\(\nu=S_{a,b,\xi}(\lambda,\mu,\kappa)\).
\end{minipage}\hfill
\begin{minipage}[t]{.46\textwidth}
\centering
\textbf{(b) A folded boundary square}\par\medskip
\begin{tikzpicture}[x=2.3cm,y=1.7cm,draw=revisioncolor,text=revisioncolor,
 every node/.style={font=\small},line cap=round]
\draw[thin] (0,0)--(1,0)--(1,1)--(0,1)--cycle;
\draw[dashed] (.08,.08)--(.92,.92);
\node[below left] at (0,0) {$\kappa$};
\node[above left] at (0,1) {$\lambda$};
\node[below right] at (1,0) {$\lambda$};
\node[above right] at (1,1) {$\nu$};
\draw[->,line width=.8pt] (.14,1)--(.85,1);
\draw[->,line width=.8pt] (1,.14)--(1,.85);
\foreach \p in {(0,0),(0,1),(1,1)}\fill[revisioncolor] \p circle (1.5pt);
\draw[fill=white] (1,0) circle (1.8pt);
\end{tikzpicture}
\par\smallskip
\(\nu=S_{a,a,\xi}(\lambda,\lambda,\kappa)\).
\end{minipage}
\par\medskip
\begin{tikzpicture}[draw=revisioncolor,text=revisioncolor,
 every node/.style={font=\small},line cap=round]
\node at (-3.6,0) {\textbf{(c) A deterministic commutation}};
\draw[thin] (0,0)--(4.1,0);
\node[above] at (0,0) {$\mu$};
\node[above] at (2.05,0) {$\kappa\ \longmapsto\ \nu$};
\node[above] at (4.1,0) {$\lambda$};
\foreach \x in {0,2.05,4.1}\fill[revisioncolor] (\x,0) circle (1.5pt);
\node[below] at (2.05,-.12) {$\nu=C_{a,b}(\lambda,\mu,\kappa)$};
\end{tikzpicture}
\captionsetup{labelfont={color=revisioncolor},textfont={color=revisioncolor}}
\caption{Local partition updates.
In (a), the three known partitions \(\lambda,\mu,\kappa\) and a
sampled random variable determine
\(\nu=S_{a,b,\xi}(\lambda,\mu,\kappa)\).
In (b), the dashed line marks the boundary fold and the hollow
vertex is the reflected copy, so the two incoming partitions are equal.
In (c), \(C_{a,b}\) replaces the middle partition \(\kappa\) by
\(\nu\), keeping the endpoint partitions \(\lambda,\mu\) fixed.}
\label{fig:sampler-recursion}
\endgroup
\end{figure}

\revheading{Minus stages.}
Assume \(q<N_-\), and set
\[
   i_0:=\psi_-(N_- - q)=i^-_{N_- -q}.
\]
\revheading{1. \(C\)-updates on \(x=s-1\).}

Put
\begin{revision}
\[
   \widetilde\lambda^{(s-1,-N_-+s+q)}
   :=\lambda^{(s-1,-N_-+s+q)} .
\]
\end{revision}
For
\begin{revision}
\[
   i=-N_-+s-1+q,
     -N_-+s-2+q,
     \ldots,
     -N_-+s.
\]
\end{revision}
let \(j=i+N_- -s\), and replace
\[
   \lambda^{(s-1,i)}
   \quad\text{by}\quad
   \widetilde\lambda^{(s-1,i)}
   :=C_{a_{i_0},a_{\psi_-(N_- -j)}}
   \bigl(
      \widetilde\lambda^{(s-1,i+1)},
      \lambda^{(s-1,i-1)},
      \lambda^{(s-1,i)}
   \bigr).
\]
Write these temporary values back to \(\lambda\).

\revheading{2. \(S\)-updates on \(x=s\).}
\rev{At the free boundary, the reflected input equals the neighbouring
partition.}

Assign the first point on the new \(x=s\) side:
\[
   \lambda^{(s,-N_-+s)}
   :=S_{a_{i_0},a_{i_0},\xi^{-\mathrm{in}}_h(i_0)}
   \bigl(
      \lambda^{(s-1,-N_-+s)},
      \lambda^{(s-1,-N_-+s)},
      \lambda^{(s-1,-N_-+s-1)}
   \bigr).
\]
\rev{The first range uses \(\xi^{--}_{h,0}\) for already processed
minus labels; the second uses \(\xi^{--}_{h,1}\) for the remaining
minus labels.}

For
\[
   t\in[-N_-+s+1..-N_-+s+q],
\]
set \(j(t):=\psi_-(-t+s+1)\), and sample
\begin{revision}
\[
   \lambda^{(s,t)}
   :=S_{a_{i_0},a_{j(t)},\xi^{--}_{h,0}(j(t),i_0)}
   \bigl(
      \lambda^{(s-1,t)},
      \lambda^{(s,t-1)},
      \lambda^{(s-1,t-1)}
   \bigr).
\]
\end{revision}
For
\[
   t\in[-N_-+s+q+1..s-1],
\]
set \(j(t):=\psi_-(-t+s)\), and sample
\begin{revision}
\[
   \lambda^{(s,t)}
   :=S_{a_{i_0},a_{j(t)},\xi^{--}_{h,1}(j(t),i_0)}
   \bigl(
      \lambda^{(s-1,t)},
      \lambda^{(s,t-1)},
      \lambda^{(s-1,t-1)}
   \bigr).
\]
\end{revision}

\revheading{3. \(C\)-updates on \(y=s-1\).}

Put
\[
   \widetilde\lambda^{(s,s-1)}:=\lambda^{(s,s-1)} .
\]
For
\begin{revision}
\[
   i=s-1,s-2,\ldots,-N_++s,
\]
\end{revision}
let \(j=i+N_+-s+1\), and replace
\[
   \lambda^{(i,s-1)}
   \quad\text{by}\quad
   \widetilde\lambda^{(i,s-1)}
   :=C_{a_{i_0},a_{\psi_+(j)}}
   \bigl(
      \widetilde\lambda^{(i+1,s-1)},
      \lambda^{(i-1,s-1)},
      \lambda^{(i,s-1)}
   \bigr).
\]
Write these temporary values back to \(\lambda\).

\revheading{4. \(S\)-updates on \(y=s\).}
Assign the first point on the new \(y=s\) side:
\[
   \lambda^{(-N_++s,s)}
   :=S_{a_{i_0},a_{i_0},\xi^{-\mathrm{out}}_h(i_0)}
   \bigl(
      \lambda^{(-N_++s,s-1)},
      \lambda^{(-N_++s,s-1)},
      \lambda^{(-N_++s-1,s-1)}
   \bigr).
\]
For
\[
   t\in[-N_++s+1..s],
\]
set \(j(t):=\psi_+(t+N_+-s)\), and sample
\[
   \lambda^{(t,s)}
   :=S_{a_{j(t)},a_{i_0},\xi^{+-}_{h,1}(j(t),i_0)}
   \bigl(
      \lambda^{(t-1,s)},
      \lambda^{(t,s-1)},
      \lambda^{(t-1,s-1)}
   \bigr).
\]

\revheading{5. \(C\)-updates on \(x=s\).}
Put
\[
   \widetilde\lambda^{(s,s)}:=\lambda^{(s,s)} .
\]
For
\[
   i=s-1,s-2,\ldots,-N_-+s+q+1,
\]
let \(j=s-i\), and replace
\[
   \lambda^{(s,i)}
   \quad\text{by}\quad
   \widetilde\lambda^{(s,i)}
   :=C_{a_{i_0},a_{\psi_-(j)}}
   \bigl(
      \widetilde\lambda^{(s,i+1)},
      \lambda^{(s,i-1)},
      \lambda^{(s,i)}
   \bigr).
\]
Write these temporary values back to \(\lambda\).

\revheading{Plus stages for \(h>1\).}\label{head:plus-stages}\label{rev:comment17}
Assume \(q\ge N_-\) and \(h>1\).  Set
\[
   i_0:=\psi_+(N-q)=i^+_{N-q}.
\]
\revheading{1. \(C\)-updates on \(y=s-1\).}
Put
\begin{revision}
\[
   \widetilde\lambda^{(s+N_- -q-1,s-1)}
   :=\lambda^{(s+N_- -q-1,s-1)} .
\]
\end{revision}
For
\begin{revision}
\[
   i=s+N_- -q-2,
     s+N_- -q-3,
     \ldots,
     -N_++s.
\]
\end{revision}
let \(j=i+N_+-s+1\), and replace
\[
   \lambda^{(i,s-1)}
   \quad\text{by}\quad
   \widetilde\lambda^{(i,s-1)}
   :=C_{a_{i_0},a_{\psi_+(j)}}
   \bigl(
      \widetilde\lambda^{(i+1,s-1)},
      \lambda^{(i-1,s-1)},
      \lambda^{(i,s-1)}
   \bigr).
\]
Write these temporary values back to \(\lambda\).

\revheading{2. \(S\)-updates on \(y=s\).}
Assign
\[
   \lambda^{(-N_++s,s)}
   :=S_{a_{i_0},a_{i_0},\xi^{+\mathrm{in}}_h(i_0)}
   \bigl(
      \lambda^{(-N_++s,s-1)},
      \lambda^{(-N_++s,s-1)},
      \lambda^{(-N_++s-1,s-1)}
   \bigr).
\]
For
\[
   t\in[-N_++s+1..s+N_- -q-1],
\]
set \(j(t):=\psi_+(t-s+N_+)\), and sample
\begin{revision}
\[
   \lambda^{(t,s)}
   :=S_{a_{j(t)},a_{i_0},\xi^{++}_{h,1}(j(t),i_0)}
   \bigl(
      \lambda^{(t-1,s)},
      \lambda^{(t,s-1)},
      \lambda^{(t-1,s-1)}
   \bigr).
\]
\end{revision}
For
\[
   t\in[s+N_- -q..s-1],
\]
set \(j(t):=\psi_+(t-s+N_+ +1)\), and sample
\begin{revision}
\[
   \lambda^{(t,s)}
   :=S_{a_{j(t)},a_{i_0},\xi^{++}_{h,0}(j(t),i_0)}
   \bigl(
      \lambda^{(t-1,s)},
      \lambda^{(t,s-1)},
      \lambda^{(t-1,s-1)}
   \bigr).
\]
\end{revision}

\revheading{3. \(C\)-updates on \(x=s-1\).}
Put
\[
   \widetilde\lambda^{(s-1,s)}:=\lambda^{(s-1,s)} .
\]
For
\[
   i=s-1,s-2,\ldots,-N_-+s,
\]
let \(j=s-i\), and replace
\[
   \lambda^{(s-1,i)}
   \quad\text{by}\quad
   \widetilde\lambda^{(s-1,i)}
   :=C_{a_{i_0},a_{\psi_-(j)}}
   \bigl(
      \widetilde\lambda^{(s-1,i+1)},
      \lambda^{(s-1,i-1)},
      \lambda^{(s-1,i)}
   \bigr).
\]
Write these temporary values back to \(\lambda\).

\revheading{4. \(S\)-updates on \(x=s\).}
Assign
\[
   \lambda^{(s,-N_-+s)}
   :=S_{a_{i_0},a_{i_0},\xi^{+\mathrm{out}}_h(i_0)}
   \bigl(
      \lambda^{(s-1,-N_-+s)},
      \lambda^{(s-1,-N_-+s)},
      \lambda^{(s-1,-N_-+s-1)}
   \bigr).
\]
For
\[
   t\in[-N_-+s+1..s],
\]
set \(j(t):=\psi_-(s+1-t)\), and sample
\[
   \lambda^{(s,t)}
   :=S_{a_{j(t)},a_{i_0},\xi^{+-}_{h,0}(i_0,j(t))}
   \bigl(
      \lambda^{(s,t-1)},
      \lambda^{(s-1,t)},
      \lambda^{(s-1,t-1)}
   \bigr).
\]

\revheading{5. \(C\)-updates on \(y=s\).}
Put
\[
   \widetilde\lambda^{(s,s)}:=\lambda^{(s,s)} .
\]
For
\[
   i=s-1,s-2,\ldots,s-q+N_-,
\]
let \(j=N_+-s+i+1\), and replace
\[
   \lambda^{(i,s)}
   \quad\text{by}\quad
   \widetilde\lambda^{(i,s)}
   :=C_{a_{i_0},a_{\psi_+(j)}}
   \bigl(
      \widetilde\lambda^{(i+1,s)},
      \lambda^{(i-1,s)},
      \lambda^{(i,s)}
   \bigr).
\]
Write these temporary values back to \(\lambda\).

\revheading{Final plus stages.}\label{head:final-plus-stages}\label{rev:comment18}
\begin{revision}
For \(h=1\), after the minus updates, read the partitions on
\(F_{-N_+}\) from its left endpoint to its bottom endpoint as
\((\eta_0,\ldots,\eta_N)\).
The ordered list initially consists of all plus labels in increasing
order, followed by all minus labels in increasing order.
The plus weights are \((uv)^2x_i\), and the minus weights are \(x_i\).

Each entry \((i,\varepsilon,w)\) stores its original column label,
current sign, and current weight. Its strip type remains \(a_i\).
There is one partition before the first entry, one between each pair
of adjacent entries, and one after the last entry.
For adjacent entries \((i,\varepsilon,w_i),(j,\delta,w_j)\), write
\((\alpha,\beta,\gamma)\) for the partitions before, between, and
after them. Exchange the two entries with their weights attached,
and replace only the middle partition using
\begin{equation}\label{eq:sampler-word-exchanges}
\begin{array}{c|c}
(\varepsilon,\delta)&\beta_{\mathrm{new}}\\\hline
(+,+)&C_{a_j,a_i}(\gamma,\alpha,\beta)\\
(-,-)&C_{a_i,a_j}(\alpha,\gamma,\beta)\\
(-,+)&S_{a_j,a_i,w_iw_j}(\alpha,\gamma,\beta).
\end{array}
\end{equation}
The first two rows are deterministic; the last uses a fresh random
variable. No exchange of a \((+,-)\) pair is required.

At the left endpoint, change \((i,+,w)\) to \((i,-,w/u^2)\)
and update
\begin{equation}\label{eq:sampler-word-left-fold}
 \eta_0^{\mathrm{new}}
 =S_{a_i,a_i,w/u}(\eta_1,\eta_1,\eta_0).
\end{equation}
At the right endpoint, change \((i,-,w)\) to \((i,+,w/v^2)\)
and update
\begin{equation}\label{eq:sampler-word-right-fold}
 \eta_N^{\mathrm{new}}
 =S_{a_i,a_i,w/v}(\eta_{N-1},\eta_{N-1},\eta_N).
\end{equation}
All other partitions are unchanged in these endpoint updates.

For \(r=N_+,N_+-1,\ldots,1\), set \(i_0=i^+_r\) and perform
these five operations:
\begin{enumerate}
\item Move \(i_0\) left across the preceding unprocessed plus entries,
using the \((+,+)\) row of \eqref{eq:sampler-word-exchanges}, until it
is the leftmost entry.
\item Apply \eqref{eq:sampler-word-left-fold}. Its sign becomes minus
and its weight becomes \(v^2x_{i_0}\); the boundary parameter is
\(uv^2x_{i_0}\).
\item Move the current minus entry to the right endpoint.
Use the \((-,+)\) stochastic row when the next entry is plus, and
the \((-,-)\) deterministic row when it is minus.
\item Apply \eqref{eq:sampler-word-right-fold}. Its sign becomes plus
and its weight becomes \(x_{i_0}\); the boundary parameter is
\(vx_{i_0}\).
\item Move the current plus entry left while the preceding entry has
original label greater than \(i_0\). Use the \((-,+)\) stochastic row
when that entry is minus, and the \((+,+)\) deterministic row when
it is plus. Stop when the preceding label is smaller than \(i_0\),
or when there is no preceding entry.
\end{enumerate}

The final list has the original column order, signs, and weights.
Assign its ordered partitions to the output path by
\[
 \lambda^{z_m}:=\eta_m,\qquad 0\le m\le N,
\]
and output \((\eta_0,\ldots,\eta_N)\).
If \(I_+=\varnothing\), skip the final plus updates; the list already
has the original column order. If \(I_-=\varnothing\), use the same
procedure with no mixed crossings.
\end{revision}

\begin{revision}
\revheading{Verification of the update rules.}
\begin{lemma}[Frontier invariant]\label{lem:sampler-frontier-invariant}
Put \(\rho=uv\).  Before the final plus stages, traverse a completed coordinate
frontier east from \((-N_++s,s)\) to \((s,s)\), then south to
\((s,-N_-+s)\).  Its ordered edges have the signed labels
\[
 (i^+_1,+),\ldots,(i^+_{N_+},+),
 (i^-_1,-),\ldots,(i^-_{N_-},-).
\]
Across a plus edge the partition increases, and across a minus edge it
decreases, by a horizontal strip for type \(L\) and a vertical strip for
type \(R\).  At the beginning of layer \(h\), attach the weight
\(w_i=\rho^{2h}x_i\) to every edge.  Processing a label changes its weight
to \(\rho^{2h-2}x_i\), while every other edge keeps its current weight.
Thus, within a layer, the already processed labels have weight
\(\rho^{2h-2}x_i\), and the remaining labels have weight
\(\rho^{2h}x_i\).
\end{lemma}
\begin{proof}
Initially all edge strips are empty, so the assertion holds.  A minus
stage first moves its current minus letter to the right endpoint by
same-sign commutations and reflects it to a plus letter.  It then moves
left through the remaining minus letters by stochastic updates and
through the plus letters by deterministic commutations.  After reflection
at the left endpoint, it moves right through the plus letters by
stochastic updates and returns to its original minus rank by deterministic
commutations.  The intermediate plus weight is
\(u^{2h}v^{2h-2}x_{i_0}\).

A plus stage uses the reverse sequence of directions: move left,
reflect at the left endpoint, move right through the remaining letters,
reflect at the right endpoint, and move left to its original plus rank.
Its intermediate minus weight is \(u^{2h-2}v^{2h}x_{i_0}\).
The local interlacing relations are preserved by each update, and the
two reflections divide the current weight by \(u^2v^2\).  Multiplying
the weights of the two letters at each stochastic crossing gives
exactly the parameters in \eqref{eq:coordinate-parameters}.  In particular, a plus line crossing an
unprocessed plus line uses \(\xi^{++}_{h,1}\), whereas crossing an
already processed plus line uses \(\xi^{++}_{h,0}\).
\end{proof}

\revheading{Final insertion.}
To check that this rule is unambiguous, let \(W_r\) be the list before
processing \(i^+_r\).  Its prefix consists of the unprocessed entries
\(i^+_1,\ldots,i^+_r\), all plus with weights \((uv)^2x_i\).
Its tail consists of every minus label and the already processed plus
labels, in increasing original label order, each with its original sign
and weight \(x_i\).  Removing the current entry leaves all other
entries in their former relative order.  The last operation therefore
inserts it at the unique position giving \(W_{r-1}\).  This proves the
invariant inductively.

In operation 3, an unprocessed plus crossing has parameter
\(\xi^{++}_{1,1}\), and an already processed plus crossing has
parameter \(\xi^{++}_{1,0}\).  Operation 5 creates an unreflected
plus--minus interaction with parameter \(x_{i_0}x_j\) exactly when
\(j\in I_-\) and \(i_0<j\).  Thus every such crossing is permitted
by Assumption~\ref{ass:admissible}; a pair with \(j<i_0\) creates no
unreflected stochastic update.
\end{revision}

\subsection{Exactness and convergence}
\begin{revision}
Put \(q_0=uv\in(0,1)\). Let \(\mathcal B\) be the countable set of admissible
partition sequences on the original word. For
\(\eta=(\eta_0,\ldots,\eta_N)\in\mathcal B\), write
\[
\begin{split}
\Phi_i(\alpha,\beta)&:=
\begin{cases}
s_{\alpha/\beta}(x_i),&(a_i,b_i)=(L,-),\\
s_{\beta/\alpha}(x_i),&(a_i,b_i)=(L,+),\\
s_{\alpha'/\beta'}(x_i),&(a_i,b_i)=(R,-),\\
s_{\beta'/\alpha'}(x_i),&(a_i,b_i)=(R,+),
\end{cases}\\
W^\partial(\eta)&:=u^{|\eta_0|}v^{|\eta_N|}
 \prod_{m=1}^N\Phi_{l+m-1}(\eta_{m-1},\eta_m).
\end{split}
\]
Thus \(Z_{f,f}=\sum_{\eta\in\mathcal B}W^\partial(\eta)\), and
\(\operatorname{Pr}_{f,f}(\eta)=W^\partial(\eta)/Z_{f,f}\).

\begin{lemma}[Local normalizations]
\label{lem:coordinate-local-normalizations}
For fixed incoming partitions, the stochastic seed weights and their sums are
\[
\begin{array}{c|c|c}
\text{recursion}&\text{seed weight}&\text{normalizer}\\\hline
HH,\ HH^\vee&\xi^m,\quad m\ge0&(1-\xi)^{-1}\\
HV,\ HV^\vee&\xi^m,\quad m\in\{0,1\}&1+\xi\\
AA,\ AA^\vee,\ AB,\ AB^\vee&1&1.
\end{array}
\]
\end{lemma}
\begin{proof}
The stochastic rules are injective in the seed for fixed incoming partitions,
and sum the geometric series or the two Bernoulli weights. The deterministic
rules have one output. Their stronger bijectivity assertions, with the middle
incoming partition allowed to vary, were verified in Section~6.1.
\end{proof}

The local weight identities also identify the terminal measure. For a
stochastic exchange, use edge variables \(x,y\), endpoint partitions
\(\lambda,\mu\), and \(\nu=S_{a,b,xy}(\lambda,\mu,\kappa)\).
The seed \(m=|\nu|+|\kappa|-|\lambda|-|\mu|\) satisfies
\begin{equation}\label{eq:sampler-local-weight}
x^{|\nu|-|\lambda|}y^{|\nu|-|\mu|}
=(xy)^m y^{|\lambda|-|\kappa|}x^{|\mu|-|\kappa|}.
\end{equation}
This is a weight-preserving bijection between a valley together with its seed
and a peak. For a same-sign exchange, the size identity for \(C\) gives
the analogous identity with no seed factor. At a free endpoint, let
\(c=u\) or \(c=v\). The folded rule
\(\nu=S_{a,a,cx}(\lambda,\lambda,\kappa)\) gives
\begin{equation}\label{eq:sampler-folded-weight}
c^{|\nu|}x^{|\nu|-|\lambda|}
=(cx)^m c^{|\kappa|}(c^2x)^{|\lambda|-|\kappa|}.
\end{equation}
Thus elimination of an outward boundary strip changes its edge variable from
\(x\) to \(c^2x\); the sampling move reverses this operation.

\begin{lemma}[The finite bijection]
\label{lem:sampler-finite-bijection}
Let \(\mathcal R_K\) be the admissible sequences on the canonical word
\[
(i^+_1,+),\ldots,(i^+_{N_+},+),
(i^-_1,-),\ldots,(i^-_{N_-},-),
\]
with edge variable \(q_0^{2K}x_i\) at label \(i\), and endpoint weights
\(u,v\). Denote their monomial weight by \(W_{\mathrm{res},K}\).
The complete stages and the final insertion schedule define a bijection
\[
(\rho,\boldsymbol m)\longleftrightarrow\eta,
\qquad \rho\in\mathcal R_K,\quad\eta\in\mathcal B,
\]
where \(\boldsymbol m=(m_e)_{e\in\mathcal E_K}\) ranges over one independent
geometric or Bernoulli seed index for every stochastic move in the fixed
schedule. Moreover,
\begin{equation}\label{eq:sampler-global-weight}
W^\partial(\eta)=W_{\mathrm{res},K}(\rho)
 \prod_{e\in\mathcal E_K}\xi_e^{m_e}.
\end{equation}
\end{lemma}
\begin{proof}
For this statement allow any admissible residual sequence, rather than only
the constant sequence used in initialization. The frontier invariant in
Section~\ref{sec:sampler-coordinate-recursion} shows that a complete stage
changes only the current edge variable, from \(q_0^{2h}x_i\) to
\(q_0^{2h-2}x_i\), and restores the canonical order. This holds for arbitrary
admissible labels: each exchange uses the stated strip relations, and each
folded boundary move uses two identical copies of its neighbouring partition.
After the minus stages at \(h=1\), the unprocessed plus variables are
\(q_0^2x_i\), and the processed minus variables are \(x_i\).

The final insertion invariant \(W_r\) then restores the original label order
and changes each remaining plus variable to \(x_i\). In particular, its last
sweep creates an unreflected pair only when \(i\in I_+,j\in I_-\) and
\(i<j\). All its moves, and all complete-stage moves, are the local
bijections of Section~6.1. Their order depends only on the labelled word,
not on the partitions or seed values. Their composition is therefore a
bijection; reversing the fixed list of moves recovers every intermediate
partition and every seed uniquely. Multiplying
\eqref{eq:sampler-local-weight}, \eqref{eq:sampler-folded-weight}, and the
same-sign identities cancels all intermediate monomials and proves
\eqref{eq:sampler-global-weight}.
\end{proof}

Define \(E_K\subset\mathcal B\) by this inverse bijection: a terminal sequence
belongs to \(E_K\) precisely when its residual sequence is constant,
\(\rho_0=\cdots=\rho_N=\Lambda\), with \(\Lambda_1\le K\).
For a constant sequence all edge strips are empty, so that
\(W_{\mathrm{res},K}(\rho)=q_0^{|\Lambda|}\). Consequently the finite
unnormalized terminal weight is exactly
\begin{equation}\label{eq:sampler-restricted-weight}
W_K^\partial(\eta):=W^\partial(\eta)\mathbf1_{E_K}(\eta).
\end{equation}
Write
\[
Z_K=Z^{(K)}_{f,f}:=\sum_{\eta\in\mathcal B}W_K^\partial(\eta),
\qquad
\operatorname{Pr}^{(K)}_{f,f}(\eta):=\frac{W_K^\partial(\eta)}{Z_K}.
\]
This specifies the finite reflected truncation as a restriction of the
original Gibbs weight, rather than only as the law of a generated array.

\begin{lemma}[Bookkeeping]
\label{lem:coordinate-bookkeeping}
Use \(\mathsf A_k,\mathsf B_k,\mathsf C_k\) and \(z_{ij}\) from
Proposition~\ref{p26}, and put
\[
Q_K:=\prod_{h=1}^K(1-q_0^h)^{-1},
\qquad U_D:=\prod_{\substack{i\in I_+,\,j\in I_-\\i<j}}z_{ij}.
\]
Then
\begin{equation}\label{eq:sampler-finite-normalizer}
Z_K=Q_K U_D
 \left(\prod_{k=1}^{2K}\mathsf A_k\mathsf C_k\right)
 \left(\prod_{k=1}^{2K-1}\mathsf B_k\right).
\end{equation}
\end{lemma}
\begin{proof}
The constant residual sequence has total weight \(Q_K\), by
\eqref{eq:LambdaK}. The finite bijection and
\eqref{eq:sampler-global-weight} identify its extension with all seed choices
with the weight in \eqref{eq:sampler-restricted-weight}. Thus its mass is
\(Q_K\) times the product of the seed normalizers.

The two boundary moves for each label at layer \(h\) give its factors in
\(\mathsf A_{2h-1}\mathsf A_{2h}\). Each unordered same-sign pair is
encountered twice, once when each of its two labels is processed, giving
\(\mathsf C_{2h-1}\mathsf C_{2h}\). In a plus stage the unprocessed partners
use \(\xi^{++}_{h,1}\), and the processed partners use
\(\xi^{++}_{h,0}\), as required by their current edge variables.
The minus stages give every cross-sign pair at level \(2h-1\).
For \(h>1\), the plus stages give every cross-sign pair at level \(2h-2\).
At \(h=1\), the final insertion instead gives precisely \(U_D\), with
no other level-zero pairs. Lemma~\ref{lem:coordinate-local-normalizations}
therefore gives \eqref{eq:sampler-finite-normalizer}.

As a separate check, the residual canonical partition function is
\[
Z_{\mathrm{res},K}=Q_\infty\,\mathsf B_{2K}
 \prod_{k\ge2K+1}\mathsf A_k\mathsf B_k\mathsf C_k,
\qquad Q_\infty:=\prod_{h\ge1}(1-q_0^h)^{-1}.
\]
Indeed its edge variables are \(q_0^{2K}x_i\); its canonical unreflected
cross factor is \(\mathsf B_{2K}\). Substitution in Proposition~\ref{p26}
gives this formula and
\((Z_K/Q_K)Z_{\mathrm{res},K}=Z_{f,f}\), in agreement with the bijection.
In particular, the finite normalization is not a common cutoff \(k\le K\)
in all the infinite products.
\end{proof}

\begin{lemma}[Normal convergence of masses]
\label{lem:terminal-mass-convergence}
Under Assumption~\ref{ass:admissible},
\[
0<Z_{f,f}<\infty,\qquad Z_K\longrightarrow Z_{f,f},
\qquad
\sum_{\eta\in\mathcal B}|W_K^\partial(\eta)-W^\partial(\eta)|
=Z_{f,f}-Z_K\longrightarrow0.
\]
\end{lemma}
\begin{proof}
For a fixed finite graph, every factor in the positive product of
Proposition~\ref{p26} is finite. The one-body and auxiliary factors have
geometrically small arguments of order \(q_0^k\), while pair factors have
arguments of order \(q_0^{2k}\), up to constants independent of \(k\).
Hence the logarithmic tails are absolutely summable. Formula
\eqref{eq:sampler-finite-normalizer} omits auxiliary indices \(h>K\),
one-body and same-sign indices \(k>2K\), and cross-sign indices
\(k\ge2K\). Each omitted tail tends to one, proving \(Z_K\to Z_{f,f}\).
Finally, \eqref{eq:sampler-restricted-weight} gives
\(0\le W_K^\partial\le W^\partial\); summing their difference proves
the last identity.
\end{proof}

\begin{lemma}[Convergence of terminal-boundary weights]
\label{lem:terminal-boundary-weight-convergence}
Under Assumption~\ref{ass:admissible}, for every \(\eta\in\mathcal B\),
\(W_K^\partial(\eta)\to W^\partial(\eta)\).
\end{lemma}
\begin{proof}
The nonnegative difference at a fixed \(\eta\) is bounded by
\(Z_{f,f}-Z_K\), which tends to zero by the preceding lemma.
\end{proof}

\begin{theorem}[Correctness and convergence of the coordinate sampler]
\label{thm:truncated-sampler}
Assume Assumption~\ref{ass:admissible}. For every \(K\ge1\), the sampler
outputs a terminal sequence with law
\[
\operatorname{Pr}^{(K)}_{f,f}
=\operatorname{Pr}_{f,f}(\,\cdot\mid E_K).
\]
With the convention \(\|P-Q\|_{\mathrm{TV}}=\frac12\sum_\eta|P(\eta)-Q(\eta)|\),
its error is exactly
\begin{equation}\label{eq:sampler-tv-error}
\left\|\operatorname{Pr}^{(K)}_{f,f}-\operatorname{Pr}_{f,f}\right\|_{\mathrm{TV}}
=1-\frac{Z_K}{Z_{f,f}}\longrightarrow0.
\end{equation}
\end{theorem}
\begin{proof}
The initial partition and all stochastic seeds are independent, with joint
probability proportional to
\(q_0^{|\Lambda|}\mathbf1_{\{\Lambda_1\le K\}}\prod_e\xi_e^{m_e}\).
The finite bijection and its weight identity carry this law to
\(W_K^\partial/Z_K\). This proves exactness at fixed \(K\) and the
conditional-law statement. Since
\(\operatorname{Pr}_{f,f}(E_K)=Z_K/Z_{f,f}\), the total-variation distance
from this conditional law to the original law is
\(1-Z_K/Z_{f,f}\). Lemma~\ref{lem:terminal-mass-convergence} proves its
convergence to zero. No nesting assumption on \(E_K\) is needed.
\end{proof}
\end{revision}

\section{An explicit example}\label{sect:ex}

\begin{example}\label{ex81}\normalfont
We illustrate the frozen-boundary equation \eqref{fb} and the truncated
coordinate sampler of Section~\ref{sect:sampler} in one explicit periodic
case.  Throughout this example we take \(\beta=1\).  Let
\[
   V_0=0,\qquad V_1=1,\qquad m=1,\qquad n=4,
\]
\[
   (a_1,b_1)=(L,-),\qquad
   (a_2,b_2)=(L,+),\qquad
   (a_3,b_3)=(R,+),\qquad
   (a_4,b_4)=(R,-),
\]
and
\[
   \tau_1=\tau_4=1,\qquad
   \tau_2=\tau_3=0.3,\qquad
   u=v=0.1.
\]
\currentrev{For \(0<\epsilon<1/4\), set
\[
   v^{(\epsilon)}_0=0,\qquad
   v^{(\epsilon)}_1
   =4\left\lfloor\frac{1}{4\epsilon}\right\rfloor,
\]
and put \(l^{(\epsilon)}=v^{(\epsilon)}_0\),
\(r^{(\epsilon)}=v^{(\epsilon)}_1\).
Extend the four-letter word above periodically to every integer in
\([l^{(\epsilon)}..r^{(\epsilon)}]\), including both endpoints. Thus
\[
   (a_{l^{(\epsilon)}},b_{l^{(\epsilon)}})
   =(a_{r^{(\epsilon)}},b_{r^{(\epsilon)}})=(R,-).
\]
Then \(\epsilon r^{(\epsilon)}\le1\) and
\(\epsilon r^{(\epsilon)}\to1\).
}

Specializing \eqref{dgc} and \eqref{dfuvk}, we obtain
\begin{revision}
\begin{align*}
G_{\chi}(w)
&=
\left(\frac{w-e}{w-e^{\chi}}\right)
\left(\frac{1-0.3\,w}{1-0.3\,e^{-\chi}w}\right)
\left(\frac{w+e^{\chi}}{w+e}\right)
\left(\frac{1+0.3\,e^{-\chi}w}{1+0.3\,w}\right),
\\
F_{u,v,k}(w)
&=
\left(\frac{w-(0.01)^{2k}e}{w-(0.01)^{2k}}\right)
\left(\frac{w-(0.1)^{4k-2}e}{w-(0.1)^{4k-2}}\right)
\\
&\qquad\times
\left(\frac{1-0.3\,w(0.01)^{2k}}
           {1-0.3\,e^{-1}w(0.01)^{2k}}\right)
\left(\frac{1-0.3\,w(0.1)^{4k-2}}
           {1-0.3\,e^{-1}w(0.1)^{4k-2}}\right)
\\
&\qquad\times
\left(\frac{w+(0.01)^{2k}}{w+e(0.01)^{2k}}\right)
\left(\frac{w+(0.1)^{4k-2}}{w+e(0.1)^{4k-2}}\right)
\\
&\qquad\times
\left(\frac{1+0.3\,e^{-1}(0.01)^{2k}w}
           {1+0.3\,(0.01)^{2k}w}\right)
\left(\frac{1+0.3\,e^{-1}(0.1)^{4k-2}w}
           {1+0.3\,(0.1)^{4k-2}w}\right).
\end{align*}
\end{revision}
Thus
\[
   S_\chi(w):=G_\chi(w)\prod_{k\ge1}F_{u,v,k}(w).
\]

We now verify the hypotheses used in the frozen-boundary theorem.  Assumption~\ref{ap5}
is built into the periodic construction above.  For \(b_i=+\), the \(q\)-volume weights
satisfy \(x_i^{(\epsilon)}\le 0.3\), while for \(b_i=-\) they satisfy
\(x_i^{(\epsilon)}\le e\).  \currentrev{Hence every unreflected pair with \(i<j\), \(b_i=+\),
\(b_j=-\), and \(a_i=a_j\) satisfies}
\[
   x_i^{(\epsilon)}x_j^{(\epsilon)}\le 0.3e<1.
\]
Moreover,
\[
   u x_i^{(\epsilon)},\ v x_i^{(\epsilon)}\le 0.1e<1,
   \qquad
   u^2\bigl(x_i^{(\epsilon)}x_j^{(\epsilon)}\bigr),
   \ v^2\bigl(x_i^{(\epsilon)}x_j^{(\epsilon)}\bigr)
   \le 0.01e^2<1.
\]
The remaining reflected factors carry still smaller powers of \(u\) and \(v\).  Therefore
the admissibility condition Assumption~\ref{ass:admissible} holds uniformly.

Assumption~\ref{ap64}(1) and Assumption~\ref{ap64}(3) are vacuous, since no two
distinct residues have the same pair \((a,b)\).  Assumption~\ref{ap64}(2) reduces to the
two inequalities
\[
   \tau_1^{-1}\tau_2=0.3<e^{-1},
   \qquad
   \tau_4^{-1}\tau_3=0.3<e^{-1},
\]
so Assumption~\ref{ap64} holds.

For \(0<\chi<1\), the active labels in the four zero--pole families are
\[
   E_{1,1,>,1}=1,\qquad
   E_{2,1,<,1}=1,\qquad
   E_{4,1,>,0}=1,\qquad
   E_{3,1,<,0}=1,
\]
and all other labels are inactive.  Hence Assumption~\ref{ap65} reduces to
\[
   0.1<e^{-1},\qquad
   0.1<e^{-1},\qquad
   0.1<1,\qquad
   0.1<1,
\]
which are all true.  Therefore Assumption~\ref{ap62} follows from
Lemma~\ref{l64}, and the contour and logarithmic branch required in the asymptotic
formulas are supplied by Corollary~\ref{cor:automatic-branch}.

Finally, direct evaluation of the displayed factors gives
\[
   S_\chi(0)=1,\qquad
   C_{\chi,K}:=\lim_{w\to\infty}S_{\chi,K}(w)=1
   \quad\text{for every }K\ge1.
\]
Also \(e^\chi\in D_{1,K}(\chi)\) for \(0<\chi<1\), so \(D_{1,K}(\chi)\neq\varnothing\).
Thus the nonexceptional endpoint conditions in Theorem~\ref{p412} are satisfied exactly
when
\[
   e^{-4\kappa}\neq 1,
   \qquad\text{equivalently}\qquad
   \kappa\neq0.
\]
Consequently, Theorem~\ref{p412} applies to every regular frozen-boundary point of this
example with
\[
   0<\chi<1,\qquad \kappa\neq0.
\]
The line \(\kappa=0\) is exceptional for the theorem as stated.

Since \(n=4\) and \(\beta=1\), the first equation in \eqref{fb} becomes
\begin{equation}\label{fbex0}
   S_\chi(w)=
   G_{\chi}(w)\prod_{k=1}^{\infty}F_{u,v,k}(w)
   =
   e^{-4\kappa}.
\end{equation}
For the second equation in \eqref{fb}, write
\[
   \frac{d}{dw}\log F_{u,v,k}(w)=g_k(w),
\]
where
\begin{revision}
\allowdisplaybreaks[1]
\begin{align}
g_k(w)
&:=
\frac{1}{w-(0.01)^{2k}e}-\frac{1}{w-(0.01)^{2k}}
\notag\\
&\quad
+\frac{1}{w-(0.1)^{4k-2}e}-\frac{1}{w-(0.1)^{4k-2}}
\notag\\
&\quad
+\frac{1}{w-(0.01)^{-2k}(0.3)^{-1}}
-\frac{1}{w-e(0.01)^{-2k}(0.3)^{-1}}
\notag\\
&\quad
+\frac{1}{w-(0.1)^{2-4k}(0.3)^{-1}}
-\frac{1}{w-e(0.1)^{2-4k}(0.3)^{-1}}
\notag\\
&\quad
+\frac{1}{w+(0.01)^{2k}}
-\frac{1}{w+e(0.01)^{2k}}
\notag\\
&\quad
+\frac{1}{w+(0.1)^{4k-2}}
-\frac{1}{w+e(0.1)^{4k-2}}
\notag\\
&\quad
+\frac{1}{w+e(0.01)^{-2k}(0.3)^{-1}}
-\frac{1}{w+(0.01)^{-2k}(0.3)^{-1}}
\notag\\
&\quad
+\frac{1}{w+e(0.1)^{2-4k}(0.3)^{-1}}
-\frac{1}{w+(0.1)^{2-4k}(0.3)^{-1}}.
\label{dgk}
\end{align}
\end{revision}
Therefore the second equation in \eqref{fb} is
\begin{align}
0
&=
\frac{1}{w-e}-\frac{1}{w-e^{\chi}}
+\frac{1}{w-(0.3)^{-1}}-\frac{1}{w-e^{\chi}(0.3)^{-1}}
\notag\\
&\quad
+\frac{1}{w+e^{\chi}}-\frac{1}{w+e}
+\frac{1}{w+e^{\chi}(0.3)^{-1}}-\frac{1}{w+(0.3)^{-1}}
+\sum_{k=1}^{\infty}g_k(w).
\label{fb1}
\end{align}

For numerical plotting we replace \(S_\chi\) by
\[
   S_{\chi,K}(w):=G_\chi(w)\prod_{k=1}^{K}F_{u,v,k}(w)
\]
and truncate the sum in \eqref{fb1} at the same level \(K\).  Figure~\ref{fig:fb}
shows the finite-\(K\) double-root locus obtained from \eqref{fbex0}--\eqref{fb1} with
\(K=10\). Figure~\ref{fig:simu} shows a sample from the truncated doubly
free-boundary model, generated by the sampler of Section~\ref{sect:sampler}
with \(K=10\), \(\epsilon=1/280\), \(l=0\), and \(r=280\).
Columns \(1,\ldots,280\) form 70 complete periods; the additional
left endpoint column \(0\) gives \(N=281\) columns in total.

\begin{figure}[htbp]
\centering
\includegraphics[width=.6\textwidth]{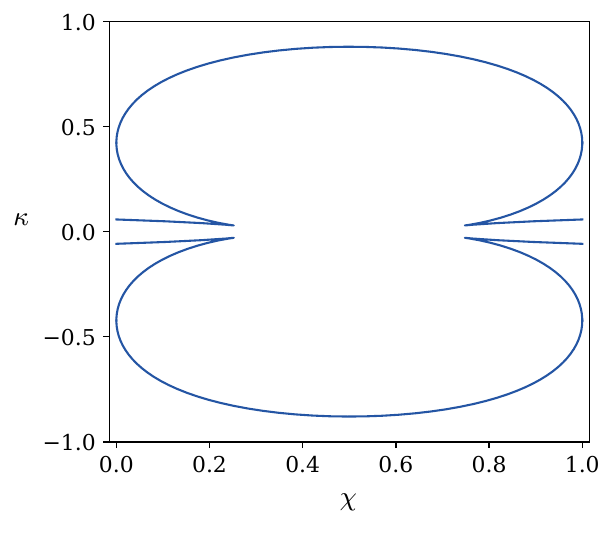}
\caption{Approximation of the frozen boundary in Example~\ref{ex81}, obtained by
truncating the system \eqref{fbex0}--\eqref{fb1} at level \(K=10\).}
\label{fig:fb}
\end{figure}

\begin{figure}[htbp]
\centering
\includegraphics[width=.9\textwidth]{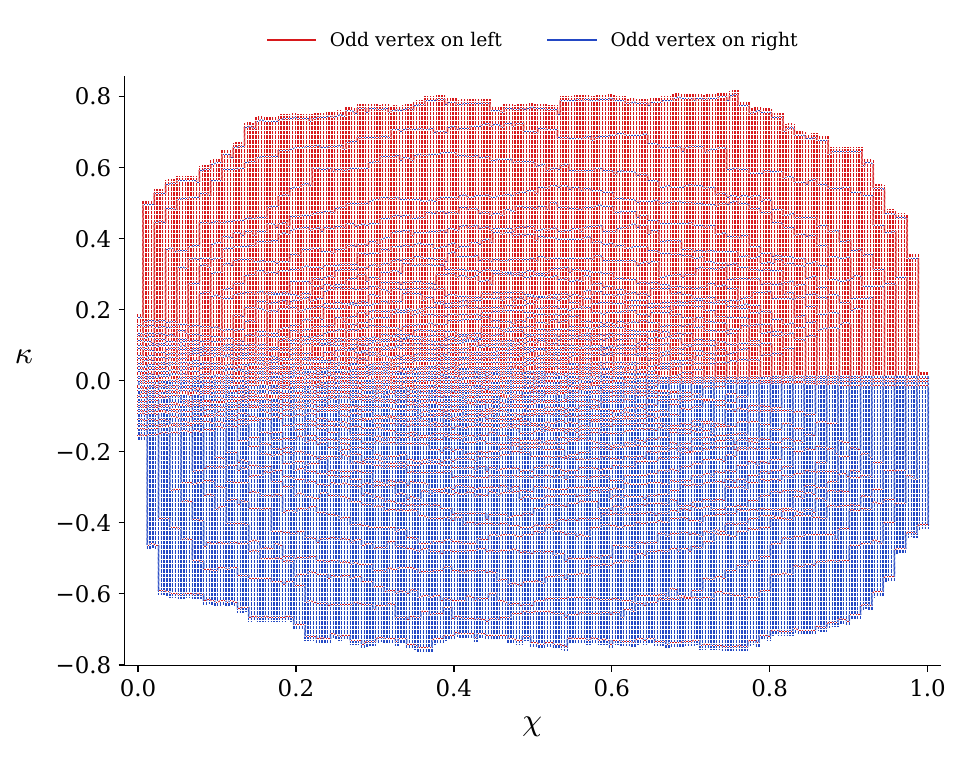}
\caption{A sample from Example~\ref{ex81} with \(K=10\) and 70 complete
periods: \(\epsilon=1/280\), \(l=0\), \(r=280\), and \(N=281\).
Red (blue) dimers have their odd vertex on the left (right).
The zero-charge representative is drawn in coordinates
\(\chi=\epsilon x_{\mathrm{emb}}/2\), \(\kappa=\epsilon y\), where
\(x_{\mathrm{emb}}\) is the embedded horizontal coordinate.
Trivial horizontal tails are omitted.}
\label{fig:simu}
\end{figure}
\FloatBarrier
\end{example}

\appendix

\section{Technical Results}\label{sc:dmp}

\subsection{Macdonald polynomials}

Let $\mathcal A$ be a $\ZZ_{\ge 0}$-graded algebra,
\[
\mathcal A=\bigoplus_{k\ge 0}\mathcal A_k.
\]
Its topological completion $\overline{\mathcal A}$ is the algebra of formal series
\[
a=\sum_{k=0}^{\infty} a_k,
\qquad a_k\in \mathcal A_k.
\]
For $a\in\overline{\mathcal A}$, define the lower degree by
\begin{equation}\label{dldeg}
\operatorname{ldeg}(a):=\inf\{k\ge 0:a_k\neq 0\},
\end{equation}
with the convention $\operatorname{ldeg}(0)=+\infty$.

\begin{definition}\label{dsp1}
Let $\Lambda_X$ be the algebra of symmetric functions in a countable set of variables
$X=(x_1,x_2,\ldots)$. A \emph{specialization} is an algebra homomorphism
\[
\rho:\Lambda_X\to \mathcal B,
\]
where $\mathcal B$ is a commutative $\CC$-algebra. For $f\in \Lambda_X$, we write
$f(\rho):=\rho(f)$.
\end{definition}

Define the power-sum symmetric functions by
\begin{equation}\label{dpn}
p_n(X):=
\begin{cases}
1, & n=0,\\[1mm]
\displaystyle\sum_{x_i\in X} x_i^n, & n\ge 1.
\end{cases}
\end{equation}
Then $\Lambda_X$ is generated by $\{p_n(X)\}_{n\ge 1}$, so a specialization is uniquely
determined by the values $\{p_n(\rho)\}_{n\ge 1}$; see \cite{bc13}.

The generating function of a specialization $\rho$ is
\[
H(\rho;z):=\exp\left(\sum_{n\ge 1}\frac{p_n(\rho)z^n}{n}\right).
\]
If $\rho_1,\rho_2$ are specializations and $c\in\CC$, define the specializations
$\rho_1\cup\rho_2$, $c\rho$, and $c*\rho$ by
\[
H(\rho_1\cup\rho_2;z):=H(\rho_1;z)H(\rho_2;z),
\qquad
H(c\rho;z):=H(\rho;cz),
\qquad
H(c*\rho;z):=[H(\rho;z)]^c.
\]
Equivalently, for every $n\ge 1$,
\begin{align}
p_n(\rho_1\cup\rho_2)=p_n(\rho_1)+p_n(\rho_2),
\qquad
p_n(c\rho)=c^n p_n(\rho),
\qquad
p_n(c*\rho)=c\,p_n(\rho).\label{dru}
\end{align}

Define
\begin{equation}\label{dh1}
H(\rho_1;\rho_2)
:=
\sum_{\lambda\in\YY}s_\lambda(\rho_1)s_\lambda(\rho_2)
=
\exp\left(\sum_{n\ge 1}\frac{p_n(\rho_1)p_n(\rho_2)}{n}\right),
\end{equation}
and
\begin{equation}\label{dht}
\widetilde{H}(\rho)
:=
\sum_{\lambda\in\YY}s_\lambda(\rho)
=
\exp\left(
\sum_{n\ge 1}\frac{p_{2n-1}(\rho)}{2n-1}
+
\sum_{n\ge 1}\frac{p_n(\rho)^2}{2n}
\right).
\end{equation}
In particular, if $\rho_z$ is the specialization of a single variable $z$, then
$H(\rho;\rho_z)=H(\rho;z)$. Moreover,
\[
H(\rho;\rho_1\cup\rho_2)=H(\rho;\rho_1)H(\rho;\rho_2),
\]
and
\[
\widetilde{H}(\rho_1\cup\rho_2)
=
\widetilde{H}(\rho_1)\widetilde{H}(\rho_2)H(\rho_1;\rho_2).
\]

For $q,t\in(0,1)$ and $\lambda\in\YY$, let $P_\lambda(X;q,t)$ denote the normalized
Macdonald symmetric polynomial characterized by the triangular expansion
\[
P_\lambda(X;q,t)
=
m_\lambda(X)+\sum_{\mu<\lambda}u_{\lambda\mu}(q,t)\,m_\mu(X),
\]
where $m_\mu$ is the monomial symmetric function; see \cite{IGM15}. The family
$\{P_\lambda(X;q,t)\}_{\lambda\in\YY}$ forms a basis of $\Lambda_X$.

For a partition $\lambda=(\lambda_1,\lambda_2,\ldots)$, define
\[
p_\lambda(\rho):=\prod_{i\ge 1}p_{\lambda_i}(\rho).
\]
Then $\{p_\lambda(X)\}_{\lambda\in\YY}$ is also a basis of $\Lambda_X$.

For partitions $\lambda,\mu\in\YY$, define the Macdonald scalar product by
\begin{equation}\label{dsp}
\langle p_\lambda,p_\mu\rangle
=
\delta_{\lambda\mu}\,
z_\lambda
\prod_{i=1}^{\ell(\lambda)}\frac{1-q^{\lambda_i}}{1-t^{\lambda_i}},
\end{equation}
where
\[
z_\lambda:=\prod_{j\ge 1}j^{m_j(\lambda)}m_j(\lambda)!,
\]
and $m_j(\lambda)$ is the multiplicity of the part $j$ in $\lambda$.
When $q=t$, the right-hand side does not depend on either $q$ or $t$.

Let $Q_\lambda(X;q,t)$ be the scalar multiple of $P_\lambda(X;q,t)$ determined by
\[
\langle P_\lambda(X;q,t),Q_\lambda(X;q,t)\rangle=1.
\]

For $\lambda\in\YY$, define skew Macdonald symmetric functions by
\[
P_\lambda(X,Y;q,t)=\sum_{\mu\in\YY}P_{\lambda/\mu}(X;q,t)P_\mu(Y;q,t),
\]
and
\[
Q_\lambda(X,Y;q,t)=\sum_{\mu\in\YY}Q_{\lambda/\mu}(X;q,t)Q_\mu(Y;q,t).
\]
For a single variable $x$, one has
\[
P_{\lambda/\mu}(x;q,t)
=
\mathbf 1_{\mu\prec\lambda}\,\psi_{\lambda/\mu}(q,t)\,x^{|\lambda|-|\mu|},
\]
and
\[
Q_{\lambda/\mu}(x;q,t)
=
\mathbf 1_{\mu\prec\lambda}\,\phi_{\lambda/\mu}(q,t)\,x^{|\lambda|-|\mu|},
\]
where $\psi_{\lambda/\mu}(q,t)$ and $\phi_{\lambda/\mu}(q,t)$ are independent of $x$.
Moreover,
\[
\left.\psi_{\lambda/\mu}(q,t)\right|_{q=t}
=
\left.\phi_{\lambda/\mu}(q,t)\right|_{q=t}
=
1;
\]
see Remarks 1 on p.~346 of \cite{IGM15}.

\begin{lemma}\label{la1}
Let $X$ be a countable set of variables. Then
\[
P_{\lambda/\mu}(X;q,t)=Q_{\lambda/\mu}(X;q,t)=0
\]
unless $\mu\subseteq\lambda$. If $\mu\subseteq\lambda$, then both
$P_{\lambda/\mu}(X;q,t)$ and $Q_{\lambda/\mu}(X;q,t)$ are homogeneous in $X$ of degree
$|\lambda|-|\mu|$.
\end{lemma}

\begin{proof}
See \cite[(7.7), p.~344]{IGM15}.
\end{proof}

In particular, when $q=t$,
\begin{equation}\label{csm}
P_\lambda(X;t,t)=Q_\lambda(X;t,t)=s_\lambda(X),
\qquad
P_{\lambda/\mu}(X;t,t)=Q_{\lambda/\mu}(X;t,t)=s_{\lambda/\mu}(X).
\end{equation}
See \cite[(4.14), p.~324]{IGM15}.

\begin{definition}
Let $k\in\ZZ_{>0}$ and let $q,t>0$. For any analytic symmetric function
\[
F(X)=\sum_{\lambda\in\YY}c_\lambda P_\lambda(X;q,t),
\]
define the Macdonald difference operator $D_{-k,X;q,t}$ by
\begin{equation}\label{ngt}
D_{-k,X;q,t}F(X)
:=
\sum_{\lambda\in\YY}
c_\lambda
\left[
(1-t^{-k})\sum_{i=1}^{\ell(\lambda)}(q^{\lambda_i}t^{-i+1})^k+t^{-k\ell(\lambda)}
\right]
P_\lambda(X;q,t).
\end{equation}
\end{definition}

Let $W=(w_1,\ldots,w_k)$ be an ordered set of variables. Define
\begin{equation}\label{ddf}
D(W;q,t)
:=
\frac{(-1)^{k-1}}{(2\pi\mathbf{i})^k}
\frac{
\sum_{i=1}^{k}\frac{w_k t^{k-i}}{w_i q^{k-i}}
}{
\prod_{i=1}^{k-1}\left(1-\frac{tw_{i+1}}{qw_i}\right)
}
\prod_{1\le i<j\le k}
\frac{\left(1-\frac{w_i}{w_j}\right)\left(1-\frac{q w_i}{t w_j}\right)}
{\left(1-\frac{w_i}{t w_j}\right)\left(1-\frac{q w_i}{w_j}\right)}
\prod_{i=1}^{k}\frac{dw_i}{w_i}.
\end{equation}
\rev{The kernel $L(W,X;q,t)$ is defined in \eqref{dh}.}

The following proposition is a slightly more general form of \cite[Proposition~4.10]{GZ16}.

\begin{proposition}\label{pa2}
Assume one of the following conditions holds:
\begin{enumerate}
\item $q,t\in(0,1)$;
\item $q,t\in(1,\infty)$.
\end{enumerate}
Let $f:\CC\to\CC$ be analytic in a neighborhood of $0$, with $f(0)\neq 0$.
Let $g:\CC\to\CC$ be analytic in a neighborhood of $0$ and satisfy
\[
g(z)=\frac{f(z)}{f(q^{-1}z)}
\]
there. Then
\begin{equation}\label{ss}
D_{-k,X;q,t}\left(\prod_{x\in X}f(x)\right)
=
\left(\prod_{x\in X}f(x)\right)
\oint\cdots\oint
D(W;q,t)\,L(W,X;q,t)\,\prod_{i=1}^{k}g(w_i),
\end{equation}
where the contours satisfy:
\begin{itemize}
\item all contours lie in a neighborhood of $0$ where both $f$ and $g$ are analytic;
\item each contour encloses $0$ and $\{qx:x\in X\}$;
\item in case \emph{(1)}, one has $|w_i|\le |t w_{i+1}|$ for all $i\in[k-1]$;
\item in case \emph{(2)}, one has $|w_i|\le |q^{-1} w_{i+1}|$ for all $i\in[k-1]$.
\end{itemize}
\end{proposition}

\begin{proof}
See \cite[Proposition~A.2]{LV21}.
\end{proof}

\begin{lemma}\label{la3}
Let
\[
(a;q)_\infty:=\prod_{r=0}^{\infty}(1-aq^r),
\]
and define
\begin{equation}\label{defPPprime}
\Pi(X,Y;q,t)
:=
\prod_{x\in X}\prod_{y\in Y}\frac{(txy;q)_\infty}{(xy;q)_\infty},
\qquad
\Pi'(X,Y):=
\prod_{x\in X}\prod_{y\in Y}(1+xy).
\end{equation}
Then
\[
\sum_{\lambda\in\YY}P_\lambda(X;q,t)Q_\lambda(Y;q,t)
=
\sum_{\lambda\in\YY}P_{\lambda'}(X;q,t)Q_{\lambda'}(Y;q,t)
=
\Pi(X,Y;q,t),
\]
and
\[
\sum_{\lambda\in\YY}P_\lambda(X;q,t)P_{\lambda'}(Y;t,q)
=
\sum_{\lambda\in\YY}Q_\lambda(X;q,t)Q_{\lambda'}(Y;t,q)
=
\Pi'(X,Y).
\]
In particular, when $q=t$,
\[
\sum_{\lambda\in\YY}s_\lambda(X)s_\lambda(Y)
=
\prod_{x\in X}\prod_{y\in Y}\frac{1}{1-xy},
\]
and
\[
\sum_{\lambda\in\YY}s_\lambda(X)s_{\lambda'}(Y)
=
\prod_{x\in X}\prod_{y\in Y}(1+xy).
\]
\end{lemma}

\begin{proof}
See \cite[(2.5), (4.13), (5.4), Section~VI]{IGM15}.
\end{proof}

\begin{lemma}\label{la6}
Let $\Pi$, $\Pi'$, and $L$ be as in \eqref{defPPprime} and \eqref{dh}. Then
\[
\Pi(X,Y;q,t)
=
\exp\left(
\sum_{n=1}^{\infty}
\frac{1-t^n}{1-q^n}\frac{p_n(X)p_n(Y)}{n}
\right),
\]
\[
\Pi'(X,Y)
=
\exp\left(
\sum_{n=1}^{\infty}\frac{(-1)^{n+1}}{n}p_n(X)p_n(Y)
\right),
\]
and
\[
L(X,Y;q,t)
=
\exp\left(
\sum_{n=1}^{\infty}\frac{1-t^{-n}}{n}p_n(qX^{-1})p_n(Y)
\right).
\]
\end{lemma}

\begin{proof}
The first identity follows from \cite[p.~310]{IGM15}. The other two follow from
\[
\Pi'(X,Y)=\bigl[\Pi(-X,Y;0,0)\bigr]^{-1},
\qquad
L(X,Y;q,t)=\Pi(qX^{-1},Y;0,t^{-1}).
\]
\end{proof}

\begin{lemma}\label{al51}
(\cite[Lemma~5.7]{Ah18})
Let $\theta\in(0,\pi)$ and $\xi>0$. Define
\[
R_{\epsilon,\theta,\xi}
:=
\{w\in\CC:\operatorname{dist}(w,[1,\infty))\le \xi\}
\cap
\{w\in\CC:|\arg(w-(1-\epsilon))|\le \theta\}.
\]
Let $\alpha>0$, and let $N(\epsilon)\in\ZZ_{>0}$ satisfy
\[
\limsup_{\epsilon\to0}\epsilon N(\epsilon)>0.
\]
Then, for fixed $\theta\in(0,\pi)$ and $\xi>0$,
\[
\frac{(z;e^{-\epsilon})_{N(\epsilon)}}{(e^{-\epsilon\alpha}z;e^{-\epsilon})_{N(\epsilon)}}
=
\left(
\frac{1-z}{1-e^{-\epsilon N(\epsilon)}z}
\right)^{\alpha}
\exp\left(
O\!\left(
\frac{\epsilon\min\{|z|,|z|^2\}}{|1-z|}
\right)
\right)
\]
uniformly for $z\in\CC\setminus R_{\epsilon,\theta,\xi}$ and all sufficiently small $\epsilon$.
\end{lemma}

\begin{lemma}\label{lb2}
(\cite[Corollary~A.2]{GZ16})
Let $d,h,k\in\ZZ_{>0}$, and let $f,g_1,\ldots,g_d$ be meromorphic functions whose possible
poles are contained in $\{z_1,\ldots,z_h\}$. Then for $k\ge 2$,
\begin{revision}
\[
\begin{aligned}
&\frac{1}{(2\pi\mathbf{i})^k}
\oint\cdots\oint
\frac{1}{(v_2-v_1)\cdots(v_k-v_{k-1})}
\prod_{j=1}^{d}\left(\sum_{i=1}^{k}g_j(v_i)\right)
\prod_{i=1}^{k}f(v_i)\,dv_i\\
&\qquad=
\frac{k^{d-1}}{2\pi\mathbf{i}}
\oint f(v)^k\prod_{j=1}^{d}g_j(v)\,dv,
\end{aligned}
\]
\end{revision}
where all contours contain $\{z_1,\ldots,z_h\}$, and on the left-hand side the $v_i$-contour
is contained in the $v_j$-contour whenever $i<j$.
\end{lemma}

\begin{lemma}\label{la9}
(Hurwitz's theorem)
Let $\{f_k\}$ be a sequence of holomorphic functions on a connected open set $G$ that
converges uniformly on compact subsets of $G$ to a holomorphic function $f$, and assume
that $f$ is not identically zero on $G$. If $f$ has a zero of order $m$ at $z_0$, then for every
sufficiently small $\rho>0$, and for all sufficiently large $k$, the function $f_k$ has exactly
$m$ zeros in the disk
\[
|z-z_0|<\rho,
\]
counted with multiplicity. Moreover, these zeros converge to $z_0$ as $k\to\infty$.
\end{lemma}

\begin{proof}
See \cite[p.~178, proof of Theorem~2]{AF79}.
\end{proof}

\section*{Statements and Declarations}
\textbf{Competing Interests.} The author declares that there are no competing interests.

\textbf{Funding.} The author was supported in part by National Science Foundation grant DMS1608896 and Simons Foundation grant 638143.

\textbf{Data Availability.} The numerical data underlying Figures 7.1 and 7.2 are available from the corresponding author on reasonable request.

\end{document}